\documentclass[smallcondensed,natbib]{svjour3}	

\usepackage{bm}
\usepackage{amsmath}
\usepackage{amssymb}
\usepackage{subfigure}
\usepackage[colorlinks=true, pdfstartview=FitV, linkcolor=black, citecolor= black, urlcolor= black]{hyperref}
\usepackage{overcite}
\usepackage{footnpag}			      	
\usepackage{booktabs}
\usepackage{graphicx}
\usepackage{xcolor}

\newcommand\hl[1]{#1}                               

\begin{document}

\title{Rapid and Accurate Methods for Computing Whiskered Tori and their Manifolds in Periodically Perturbed Planar Circular Restricted 3-Body Problems}
\titlerunning{Computing Whiskered Tori and their Manifolds in Periodically Perturbed PCRTBP}

\author{Bhanu Kumar \and Rodney L. Anderson \and Rafael de la Llave}

%
%

\institute{B. Kumar \and R. de la Llave
\at School of Mathematics, Georgia Institute of Technology, 686 Cherry St. NW, Atlanta, GA 30332.,
USA \\\email{bkumar30@gatech.edu}
\and
R.L. Anderson
\at Jet Propulsion Laboratory, California Institute of Technology, 4800 Oak Grove Drive, Pasadena, CA 91109
}

\maketitle{}

\begin{abstract}
When the planar circular restricted 3-body problem (RTBP) is periodically perturbed, families of unstable periodic orbits break up into whiskered tori, with most tori persisting into the perturbed system. In this study, we 1) develop a quasi-Newton method which simultaneously solves for the tori and their center, stable, and unstable directions; 2) implement continuation by both perturbation as well as rotation numbers; 3) compute Fourier-Taylor parameterizations of the stable and unstable manifolds; 4) regularize the equations of motion; and 5) globalize these manifolds. Our methodology improves on efficiency and accuracy compared to prior studies, and applies to a variety of periodic perturbations. We demonstrate the tools near resonances in the planar elliptic RTBP.
\end{abstract}

\section{Introduction}

Numerous studies have been carried out in recent years where quasi-periodic orbits of various restricted 3 or 4-body models have been computed and used for applications to space mission design. For instance, \citet{Farres2017} studied periodic and quasi-periodic orbits in the phase space of the Augmented Hill 3-Body problem near the L\textsubscript{1} and L\textsubscript2 libration points. \citet{olikaraThesis} applied collocation methods to the computation of invariant tori near L\textsubscript1 and L\textsubscript2 in both spatial circular restricted 3-body problem (CRTBP) as well as periodically-perturbed planar CRTBP models. And looking further in the past, the book series of \citet{simoetal} presented many other computational methods and applications for quasi-periodic orbits near libration points. However, all of these studies, as well as almost all other prior research, use methods of computing tori which require solving large-dimensional linear systems of equations at each step of the differential correction. Furthermore, while the quasi-periodic orbits are computed successfully using those methods, stability information including stable and unstable directions must be computed separately. Also, in most prior work, these linear stable/unstable directions are directly used as approximate local stable/unstable manifolds for the tori, neglecting higher order terms and thus losing accuracy. 

Another characteristic of the vast majority of prior research, including the previously mentioned studies, is that the analysis focuses on tori associated with the libration points, such as Lissajous or quasi-halo orbits.  \citet{Bosanac2018} did compute invariant tori near stable resonant periodic orbits in the planar circular restricted 3-body problem (PCRTBP), but the tori computed are stable, without stable or unstable manifolds. Unstable resonant periodic orbits and their stable and unstable manifolds are known to be important mechanisms of dynamical transport in the interior and exterior realms of the CRTBP  \citep{KoLoMaRo}, and these orbits have seen significant interest and use as a tool for trajectory design in multi-body systems. For example, out of the nine Titan-to-Titan encounters made by Cassini between July 2013 and June 2014, eight of the nine resulting transfers involved resonances \citep{vaqueroCassini}. More recently, the baseline mission design for the Europa Lander mission concept made profitable use of these mechanisms for the final approach to the surface of Europa \citep{Anderson2019}. For further examples and more background
on resonant orbits, see \citet{Anderson2016}. 

In this study, we develop efficient algorithms which enable simultaneous computation of not only unstable invariant tori, but also of their center, stable, and unstable directions (also known as bundles) in periodically perturbed PCRTBP models. Solving for bundles alongside the tori actually allows us to avoid solving large linear systems,  thus improving the algorithmic efficiency of our method compared to tori-only methods used in previous investigations. We apply our tools to the computation of unstable tori and bundles near PCRTBP resonances, using the Jupiter-Europa planar elliptic RTBP as the dynamical model for demonstration and a solution tolerance of $10^{-7}$. Next, we use the results of the preceding step to start a recursive parameterization method  \citep{CabreFontichLlave, huguet2012, haroetal} for the computation of high order Fourier-Taylor approximations of the stable and unstable manifolds of the tori. We demonstrate improvements in manifold accuracy as compared with the linear manifold approximations used in other studies; these parameterizations can also be differentiated, which is useful for computing intersections of manifolds. Finally, we develop a Levi-Civita regularization for the equations of motion, which is used to globalize the parameterized manifolds \hl{even when they pass through singularities of the equations of motion}. 

We have included several proofs throughout this paper to justify our methods and motivate possible adaptations; these may be skipped without detriment by readers primarily interested in details of the algorithm implementation. 

\section{Models and Background} \label{modelsection}
\subsection{Planar Circular Restricted 3-body Problem}
In this study, we consider periodic perturbations of the PCRTBP. The PCRTBP describes the motion of an infinitesimally small particle (thought of as a spacecraft) under the gravitational influence of two large bodies of masses $m_{1}$ and $m_{2}$, collectively referred to as the primaries. In this model, $m_{1}$ and $m_{2}$ revolve about their barycenter in a circular Keplerian orbit. Units are also normalized so that the distance between the two primaries becomes 1, their period of revolution becomes $2 \pi$, and $\mathcal{G}(m_{1}+m_{2})$ becomes 1. We define a mass ratio $\mu = \frac{m_{2}}{m_{1}+ m_{2}}$, and use a synodic, rotating non-inertial Cartesian coordinate system centered at the primaries' barycenter such that $m_{1}$ and $m_{2}$ are always on the $x$-axis. In the planar CRTBP, we also assume that the spacecraft moves in the same plane as the primaries. In this case, the equations of motion are Hamiltonian with form \citep{celletti}

\begin{equation} \label{pcrtbpH_EOM} \dot x = \frac{\partial H_{0}}{\partial p_{x}} \quad \dot y = \frac{\partial H_{0}}{\partial p_{y}} \quad \quad \dot p_{x} = -\frac{\partial H_{0}}{\partial x} \quad \dot p_{y} = -\frac{\partial H_{0}}{\partial y} \end{equation}
\begin{equation}  \label{pcrtbpH} H_{0}(x,y,p_x,p_{y})= \frac{p_{x}^{2}+p_{y}^{2}}{2} + p_{x}y -p_{y}x - \frac{1-\mu}{r_{1}} - \frac{\mu}{r_{2}} \end{equation}
where $r_{1} = \sqrt{(x+\mu)^{2} + y^{2}}$ is the distance from the spacecraft to $m_{1}$ and $r_{2} = \sqrt{(x-1+\mu)^{2} + y^{2}} $ is the distance to $m_{2}$. 

There are two important properties of Eq. \eqref{pcrtbpH_EOM}-\eqref{pcrtbpH} to note. First of all, the Hamiltonian in Eq. \eqref{pcrtbpH} is autonomous and is hence an integral of motion. Hence, trajectories in the PCRTBP are restricted to 3-dimensional energy submanifolds of the state space satisfying $H(x,y,p_{x}, p_{y})=$ constant. The second property is that the equations of motion have a time-reversal symmetry. Namely, if $(x(t), y(t), t)$ is a solution of Eq. \eqref{pcrtbpH_EOM}-\eqref{pcrtbpH} for $t > 0$, then $(x(-t), -y(-t), t)$ is a solution for $t < 0$. 

\subsection{Periodic Perturbations of the PCRTBP} \label{pertPCRTBPsection}
The PCRTBP model exhibits many of the important dynamical phenomena present in multi-body celestial systems. However, there are many effects which are not included in the PCRTBP; many of these other influences on the spacecraft motion act in an approximately time-periodic manner, while preserving the Hamiltonian nature of the system. Here we will study dynamical models where one such periodic forcing effect is considered in addition to the PCRTBP. The equations of motion in this case are given by Eq. \eqref{H_EOM} along with time-periodic Hamiltonian \eqref{perturbed_H} 
\begin{equation} \label{H_EOM} \dot x = \frac{\partial H_{\varepsilon}}{\partial p_{x}} \quad \dot y = \frac{\partial H_{\varepsilon}}{\partial p_{y}} \quad \quad \dot p_{x} = -\frac{\partial H_{\varepsilon}}{\partial x} \quad \dot p_{y} = -\frac{\partial H_{\varepsilon}}{\partial y} \quad \quad \dot \theta_{p} = \Omega_{p} \end{equation}
\begin{equation}  \label{perturbed_H} H_{\varepsilon}(x,y,p_x,p_{y}, \theta_{p})= H_{0}(x,y,p_x,p_{y})+ H_{1}(x,y,p_x,p_{y}, \theta_{p}; \varepsilon)\end{equation}
where $\theta_{p} \in \mathbb{T}$ is an angle, $H_{0}$ is the PCRTBP Hamiltonian given by Eq. \eqref{pcrtbpH},  $H_{1}$ is the perturbation by the time-periodic effect and satisfies $H_{1}(x,y,p_x,p_{y}, \theta_{p}; 0)=0$, and $\varepsilon > 0$ and $\Omega_{p}$ are the perturbation parameter and perturbation frequency, respectively. $\varepsilon$ signifies the strength of the perturbation, $\varepsilon=0$ being the unperturbed PCRTBP, and $\Omega_{p}$ is a known constant frequency. The perturbation from $H_{1}$ is $2\pi / \Omega_{p}$ periodic, with $\theta_{p}$ being the phase of the perturbation. Note that the Hamiltonian function will no longer be constant along trajectories. 

There are many different Hamiltonian periodically perturbed PCRTBP models of interest for applications. A common perturbation added to the PCRTBP is that of a third large body revolving in a circle (or approximate circle) around $m_{1}$ or $m_{2}$. Examples of these restricted 4-body models include the bicircular problem  \citep{simo1995bicircular}, the coherent quasi-bicircular problem  \citep{andreu1998quasi}, and the Hill restricted 4-body problem \citep{scheeres1998restricted}. Another common periodically-perturbed PCRTBP model is the planar elliptic RTBP (PERTBP).

\subsection{Planar Elliptic Restricted 3-body Problem}
The tools we develop in this paper are applicable to a wide variety of Hamiltonian periodic perturbations as discussed in the previous section. However, in this study, we use the PERTBP for numerical demonstration of their usage.  The PERTBP has the same assumptions as the PCRTBP except that one allows $m_{1}$ and $m_{2}$ to move around their barycenter in an elliptical Keplerian orbit of eccentricity $\varepsilon > 0$. The length unit is normalized such that the semi-major axis of the $m_{1}$-$m_{2}$ orbit is 1. As the period of the primaries' orbit is $2\pi$, we have that the perturbation frequency $\Omega_{p} = 1$, so we can consider $\theta_{p} = t$ modulo $2\pi$. 

The PERTBP model we use is essentially the same as that used by \citet{hiday1994transfers}, except for a transformation from position-velocity to position-momentum coordinates and a restriction to the $xy$-plane. The coordinate system is again such that $m_{1}$ and $m_{2}$ are always on the $x$-axis with the origin at their barycenter. However, the distance between them is now time-periodic, with periapse at $t=0$; this is different from the pulsating coordinates used by \citet{szebehely1969}. The equations of motion are Eq. \eqref{H_EOM} with time-periodic Hamiltonian 
\begin{equation}  \label{pertbpH} H_{\varepsilon}(x,y,p_x,p_{y},t)= \frac{p_{x}^{2}+p_{y}^{2}}{2} + n(t)(p_{x}y -p_{y}x) - \frac{1-\mu}{r_{1}} - \frac{\mu}{r_{2}} \end{equation}
where we have $n(t) = \frac{\sqrt{1-\varepsilon^{2}}}{(1-\varepsilon \cos E(t))^{2}}$, $r_{1} = \sqrt{(x+\mu(1-\varepsilon \cos E(t)))^{2} + y^{2}} $ and $ r_{2} = \sqrt{(x-(1-\mu)(1-\varepsilon \cos E(t)))^{2} + y^{2}}$. $E(t)$ is the $2\pi$-periodic eccentric anomaly of the elliptical $m_{1}$-$m_{2}$ orbit, and can be computed by solving the standard Kepler's equation \hl{$M = E - \varepsilon \sin E$} \citep{bmw}. $n(t)$ is the time derivative of the $m_{1}$-$m_{2}$ true anomaly. From Eq. \eqref{H_EOM} and \eqref{pertbpH}, we have $p_{x}=\dot x-n(t)y$ and $p_{y}=\dot y+n(t)x$.

\subsection{Resonant Periodic Orbits} \label{resonantpo}

\hl{Mean-motion resonances are PCRTBP periodic orbits which, by definition, persist from elliptical orbits of the Kepler problem, and hence are not in the center manifold of any of the libration points}. Their main characteristic is that their orbital periods are nearly rational multiples of $2\pi$, the period of $m_{1}$-$m_{2}$ motion (the periods become exact rational multiples of $2\pi$ as $\mu \rightarrow 0$). A family of resonant periodic orbits is characterized by a ratio $m:n$, $m, n \in \mathbb{Z}^{+}$. This notation means that in an inertial reference frame, the spacecraft makes approximately $m$ revolutions about $m_{1}$ in the time that $m_{1}$ and $m_{2}$ revolve $n$ times around their barycenter. 

For a given resonance $m:n$ in the PCRTBP with $\mu>0$, there typically exist one stable and one unstable resonant periodic orbit inside the submanifold $H_{0}(x,y,p_{x},p_{y})=E$, for each fixed value of $E$ in some interval of energy values $[E_{min}, E_{max}]$ \citep{kumar2021journal}. This gives us continuous families of stable and unstable resonant periodic orbits; the periods of the orbits within a given  family vary with $E$. The unstable resonant periodic orbits have monodromy (Floquet) matrix eigenvalues 1, 1, $\lambda_{s}$, and $\lambda_{u} = \lambda_{s}^{-1}$, where $| \lambda_{s} | < 1$. Thus, there are 2D stable and unstable manifolds attached to the unstable resonant periodic  orbits. These manifolds serve as low-energy pathways to and from these periodic orbits. Furthermore, the manifolds of  different resonances at the same energy level can intersect in the PCRTBP, giving heteroclinic connections which allow for propellant-free resonance transitions. For more details see \citet{kumar2021journal}.

\section{Normally Hyperbolic Invariant Manifolds and Existence of Tori} \label{nhim}

As discussed in Section \ref{resonantpo}, for each value of $E$ in some interval of energy values $[E_{min}, E_{max}]$, there exists one unstable $m:n$ resonant periodic orbit in the PCRTBP. Each of these periodic orbits is diffeomorphic to a circle $\mathbb{T}$. Now, consider the union of all the unstable $m:n$ resonant periodic orbits for all values of $E \in [E_{min}, E_{max}]$. The resulting set is a 2D manifold $\Xi$ diffeomorphic to $\mathbb{T} \times [0,1]$ in the 4D PCRTBP phase space. Furthermore, at each point of $\Xi$, there are stable and unstable directions transverse to the manifold, which come from the stable and unstable eigenvectors of the periodic orbits which make up this manifold. Since the phase space is 4D and $\Xi$ is 2D, at any point of $\Xi$, the stable and unstable directions together with the 2D manifold tangent space span the entire phase space. This means that $\Xi$ is a normally hyperbolic invariant manifold (NHIM) of the PCRTBP flow; in fact, any family of unstable PCRTBP periodic orbits forms such a NHIM. For a rigorous definition of NHIMs for flows, see \citet{fenichel1971persistence}. 

NHIMs are important because they persist under sufficiently small perturbations of the equations of motion \citep{fenichel1971persistence}; as our numerical results later in this paper will demonstrate, the perturbations we study are indeed ``sufficiently small". However, to apply this NHIM persistence result to our case of time-periodic perturbations, the original and perturbed systems must be defined on the same phase space. Hence, we take the PCRTBP from its original 4D phase space $(x,y,p_{x}, p_{y})$ to the 5D extended phase space $(x,y,p_{x}, p_{y},\theta_{p})$, $\theta_{p} \in \mathbb{T}$. We define $\dot \theta_{p} = \Omega_{p}$  for the unperturbed PCRTBP in extended phase space, while $x,y,p_{x},$ and $p_{y}$ still follow Equations \eqref{pcrtbpH_EOM} and \eqref{pcrtbpH}. Hence, periodic orbits of period $T_{1}$ from the original 4D PCRTBP phase space become 2D quasi-periodic orbits in the PCRTBP defined on the extended phase space, with one of the frequencies being $\Omega_{1}=2\pi/T_{1}$ and the other being $\Omega_{p}$, unless $\Omega_{1}/\Omega_{p}$ is rational. The NHIM $\Xi$ from the original phase space becomes the NHIM $\bar \Xi = \Xi \times \mathbb{T}$ in the extended phase space due to the extra angle. Hence, $\bar \Xi$ is diffeomorphic to $\mathbb{T}^{2} \times [0,1]$. 

Now, since the PCRTBP and its NHIM have been transferred to the same extended phase space as the periodically-perturbed models, we can conclude that for $\varepsilon > 0$ sufficiently small, $\bar \Xi$ will persist as a NHIM $\bar \Xi_{\varepsilon}$ of the perturbed equations of motion \eqref{H_EOM} and \eqref{perturbed_H}. $\bar \Xi_{\varepsilon}$ will be diffeomorphic to $\bar \Xi$ and hence also to $\mathbb{T}^{2} \times [0,1]$. Furthermore, note that $\bar \Xi$ in the extended phase space PCRTBP is foliated by 2D invariant tori, since $\Xi$ was foliated by periodic orbits. From KAM theory \citep{capinski2016}, we can expect that inside $\bar \Xi_{\varepsilon}$, the invariant tori from $\bar \Xi$ with sufficiently non-resonant frequencies $\Omega_{1}$ and $\Omega_{p}$ will also persist after perturbation with only small gaps between them. Hence, we will focus this study on these 2D tori in the periodically-perturbed PCRTBP models. 

\subsection{Stroboscopic Maps, NHIMs, and Invariant Circles} \label{stroboscopic}

Any 2D invariant torus in the periodically-perturbed PCRTBP extended phase space can be parameterized as the image of a function of two angles $K_2 : \mathbb{T}^{2} \rightarrow \mathbb{R}^{4} \times \mathbb{T}$. A quasi-periodic trajectory $\bold{x}(t)$ lying on this torus can be expressed as
	\begin{equation} \bold{x}(t) = K_2(\theta, \theta_p)  \quad \quad \quad \theta = \theta_{0}+\Omega_1 t, \quad  \theta_p = \theta_{p,0}+\Omega_p t \end{equation}
where $\theta_{0}$ and $\theta_{p,0}$ are determined from the initial condition $\bold{x}(0)$, and $\Omega_{1} = 2\pi/T_{1}$, where $T_{1}$ is the period of the PCRTBP periodic orbit associated with the torus. $\theta_{p}$ is the same perturbation phase angle defined in Section \ref{pertPCRTBPsection}, so one of the two torus frequencies will be $\Omega_p$. We can then define the stroboscopic map $F_{\varepsilon}: \mathbb{R}^{4} \times \mathbb{T} \rightarrow \mathbb{R}^{4} \times \mathbb{T}$ as the time-$2\pi/\Omega_p$ mapping of extended phase space points by the equations of motion \eqref{H_EOM} and \eqref{perturbed_H} with perturbation parameter $\varepsilon$. We find that
	 \begin{equation} \label{invariance_fK2} F_{\varepsilon}(K_2(\theta, \theta_p)) = K_2(\theta+\omega, \theta_p), \text{ where  } \omega = 2\pi \Omega_1/\Omega_p \end{equation} 
since the angle $\theta_p$ advances by $2\pi$ in the time $2\pi/\Omega_{p}$ and is therefore invariant under $F_{\varepsilon}$. Hence, we can fix $\theta_p$ (for our PERTBP test case we choose $\theta_{p}=0$), and then define $K(\theta) = K_2(\theta, \theta_p)$. Then, Eq. \eqref{invariance_fK2} becomes 
\begin{equation} \label{invarianceeps} F_{\varepsilon}(K(\theta)) = K(\theta+\omega)  \end{equation} 
By ignoring the last fixed $\theta_{p}$ component of the extended phase space and a slight abuse of notation, we can consider $F_{\varepsilon}:\mathbb{R}^{4} \rightarrow \mathbb{R}^{4}$ and $K:\mathbb{T} \rightarrow \mathbb{R}^{4}$. 

Eq. \eqref{invarianceeps} implies that $K$ parameterizes an invariant 1D torus of the map $F_{\varepsilon}$. It is significantly more computationally efficient to compute 1D invariant tori (invariant circles) $K$ of the map $F_{\varepsilon}$ in 4D phase space than 2D tori $K_{2}$ of the flow in the 5D extended phase space. The reason for this is that the reduction in the dimension of the torus helps mitigate the curse of dimensionality (see remark \ref{dimRemark}). Hence, from this point onwards, we will consider the map $F_{\varepsilon}$ and its invariant circles and manifolds, rather than invariant objects of the continuous time flow. Similar approaches are also used by \citet{zhang} and \citet{haro2021flow}. Note that the computation of $F_{\varepsilon}$ is just the integration of an ODE. 

As a final note, the stroboscopic map allows us to use the theory of NHIMs of maps \citep{fenichel1971persistence, hirschPughShub} to understand the presence of invariant circles in periodically-perturbed PCRTBP models. In particular, note that unstable periodic orbits of the unperturbed PCRTBP are also unstable invariant circles of the map $F_{\varepsilon = 0}$. Hence, the PCRTBP flow NHIM $\Xi$ defined at the beginning of Section \ref{nhim} is also a NHIM of the map $F_{\varepsilon = 0}$. Just as in the case of flows, the theory shows that NHIMs of maps persist under sufficiently small perturbations of the map. Hence, for sufficiently small $\varepsilon > 0$, $\Xi$ will persist as a NHIM $\Xi_{\varepsilon}$ of $F_{\varepsilon}$, with $\Xi_{\varepsilon}$ diffeomorphic to $\Xi$ and hence also to $\mathbb{T} \times [0,1]$. Furthermore, since $\Xi$ is foliated by invariant circles whose rotation numbers satisfy a twist condition, KAM theory \citep{capinski2016} tells us that that the invariant circles with sufficiently irrational rotation number $\omega$ persist inside $\Xi_{\varepsilon}$ for $\varepsilon > 0$. 

\begin{remark} \label{dimRemark}
\hl{The evaluation of $F_{\varepsilon}$ can be computationally expensive. Hence, one may wonder if the dimension reduction actually helps the computation efficiency or not. However, the problems of propagating in time and computing the tori are numerically very different; while numerical integration remains very feasible for all the values of $\varepsilon$, tori can break down for larger values of the parameter. Hence, the flow torus parameterization $K_{2}$ has very anisotropic regularity and behavior. It remains extremely smooth in the flow direction, but in the transversal direction, it may lose differentiability. Using algorithms that recognize this effect is advantageous. }

\hl{Also, the problem of integrating ODE's has been extensively studied over many years and there are many efficient algorithms that can be tried, including adaptive algorithms that use smaller step sizes on small spots where the equation is stiff. However, computing the 2D torus parameterization $K_{2}$ requires a uniform grid discretizing $\mathbb{T}^{2}$, which would result in unnecessarily large numbers of discretization points throughout the trajectory. For our algorithm, the operation count is close to linear in the number of grid points, so the cost of adding one more dimension would be significant. Finally, also note that numerically integrating a grid of points is very readily parallelizable by assigning each trajectory to a thread. }

\end{remark}

\section{A Parameterization Method for Computing Invariant Tori and Bundles} \label{torisection}

In this section, we develop and implement a parameterization method for the simultaneous computation of unstable invariant tori as well as their center, stable, and unstable directions, also known as bundles, for stroboscopic maps of periodically-perturbed PCRTBP models. The method works both for tori with cylindrical stable/unstable bundles as well as for those whose bundles are M\"obius strips (see Section \ref{continuationSection}). We present the analytical details and derivation of the method, as well as the considerations required for its discretization and numerical implementation in a computer program. Our method is broadly inspired by those of \citet{haroetal}, except for the additional presence of a center bundle which is not considered by them and requires extra calculations. A different but conceptually related method can also be found in the work of \citet{fontichLlaveSire}. 

\subsection{The Parameterization Method for Invariant Manifolds} \label{paramsectiongeneral}

The parameterization method is a general technique for the computation of many kinds of invariant objects in dynamical systems, including tori and stable and unstable manifolds. \citet{haroetal} describe several applications. The idea is that given a map $F: \mathbb{R}^{d} \rightarrow \mathbb{R}^{d}$, if we know that there is an $F$-invariant object diffeomorphic to some model manifold $\mathcal{M}$, then we can solve for a function $W:\mathcal M \rightarrow \mathbb{R}^{d}$ and a diffeomorphism $f: \mathcal M \rightarrow \mathcal M$ such that the invariance equation
\begin{equation}  \label{invariancequation}   F(W(s)) = W(f(s)) \end{equation}
holds for all $s \in \mathcal M$. $W$ is referred to as the parameterization of the invariant manifold, and $f$ as the internal dynamics on $\mathcal M$. Eq. \eqref{invariancequation} means that $F$ maps the image $W(\mathcal M)$ into itself, so that $W(\mathcal M)$ is the invariant object in the full space $\mathbb{R}^{d}$. 

\subsection{Equations for Parameterization Method for Invariant Tori and Bundles}

For notational convenience, denote the stroboscopic map $F_{\varepsilon}$ from Section \ref{stroboscopic} as $F$ from now on. Assume we are computing an $\varepsilon > 0$ invariant circle corresponding to a PCRTBP periodic orbit of known period $T_{1}$; this fixes the rotation number $\omega=2\pi \Omega_1/\Omega_p$ since $\Omega_1=2\pi/T_{1}$. As given in Eq. \eqref{invarianceeps}, we wish to find a parameterization $K:\mathbb{T} \rightarrow \mathbb{R}^{4}$ of the $F$-invariant circle satisfying the torus invariance equation 
\begin{equation} \label{invariance} F(K(\theta)) = K(\theta+\omega)  \end{equation} 
Eq. \eqref{invariance} is equivalent to the framework of Section \ref{paramsectiongeneral} with $\mathcal M = \mathbb{T}$ and $f(s) = s+\omega$. In addition, for our quasi-Newton method, we will add another equation to be solved for matrix-valued periodic functions $P(\theta)$, $\Lambda(\theta): \mathbb{T} \rightarrow \mathbb{R}^{4 \times 4}$ such that
		 \begin{equation}  \label{bundleEquations} DF(K(\theta)) P(\theta) = P(\theta+\omega) \Lambda(\theta) \end{equation} 
Furthermore, we mandate that $\Lambda(\theta)$ has the form 
\begin{equation} \label{Lambdaform} \Lambda(\theta) = \begin{bmatrix}
1 &  T(\theta)   & 0 & 0 \\ 0 &  1   & 0 & 0 \\ 0 & 0  & \lambda_s(\theta) & 0 \\ 0 &  0 & 0 & \lambda_u(\theta) \end{bmatrix}  \end{equation}
for some functions $T(\theta), \lambda_{s}(\theta), \lambda_{u}(\theta):\mathbb{T} \rightarrow \mathbb{R}$ to be solved for. The form of Eq. \eqref{Lambdaform} is motivated by geometric considerations that we will detail in Section \ref{understandP}. 

As will be explained at the end of Section \ref{Kstep}, solving simultaneously for $K$, $P$, and $\Lambda$ is actually more efficient than solving for $K$ alone; the quasi-Newton method we will present for solving Eq. \eqref{invariance}-\eqref{bundleEquations} uses the near-diagonal form of $\Lambda$ to decouple the linear system of equations we get in each differential correction step. The method will require only algebraic operations, phase shifts, and the solving of 1D equations for scalar-valued functions. 

\subsection{Understanding the $P$ and $\Lambda$ Matrices} \label{understandP}

In addition to their numerical utility, $P$ and $\Lambda$ have a geometric significance which will be useful when computing stable and unstable manifolds later on. Since $K$ is contained in the 2D normally hyperbolic invariant manifold $\Xi_{\varepsilon}$ defined at the end of Section \eqref{stroboscopic}, we know that there are tangent, center, stable, and unstable directions to the torus at each point $K(\theta)$. The columns of $P$ will be these four vector bundles, with $\lambda_{s}(\theta)$ and $\lambda_{u}(\theta)$ set to the stable and unstable multipliers for the corresponding bundles. To see why, consider Eq. \eqref{bundleEquations} column by column. 

Let $\bold{v}_{t}(\theta), \bold{v}_{c}(\theta) $, $ \bold{v}_{s}(\theta) $, and $\bold{v}_{u}(\theta)$ denote the first, second, third, and fourth columns of $P(\theta)$, respectively. Then, Equations \eqref{bundleEquations} and \eqref{Lambdaform} are equivalent to
\begin{align} \label{col1} DF(K(\theta)) \bold{v}_{t}(\theta) &= \bold{v}_{t}(\theta+\omega) \\
 \label{col2} DF(K(\theta)) \bold{v}_{c}(\theta) &= T(\theta) DK(\theta+\omega)+\bold{v}_{c}(\theta+\omega) \\
 \label{col3} DF(K(\theta)) \bold{v}_{s}(\theta) &= \lambda_{s}(\theta) \bold{v}_{s}(\theta+\omega) \\
 \label{col4} DF(K(\theta)) \bold{v}_{u}(\theta) &= \lambda_{u}(\theta) \bold{v}_{u}(\theta+\omega) \end{align}
First of all, note that Eq. \eqref{col3}-\eqref{col4} are the definition of stable and unstable bundles $\bold{v}_{s}(\theta)$ and $\bold{v}_{u}(\theta)$ and multipliers $\lambda_{s}(\theta)$ and $\lambda_{u}(\theta)$ for the torus $K$. Hence, the third and fourth columns of $P$ satisfy Eq. \eqref{bundleEquations} if and only if they are torus stable and unstable bundles, respectively. Also, differentiating Eq. \eqref{invariance} gives
\begin{equation} \label{tangent} DF(K(\theta)) DK(\theta) = DK(\theta+\omega) \end{equation}
which shows that $\bold{v}_{t}(\theta) = DK(\theta)$ solves Eq. \eqref{col1}. As a result, the first column of $P$ can be set as the torus tangent bundle $DK(\theta)$; in fact, if Eq. \eqref{bundleEquations} has a solution, it is easy to show that column 1 of $P$ \emph{must} be $\alpha DK(\theta)$ for some $\alpha \in \mathbb{R}$.

Finally, since $F$ is a Hamiltonian flow map and hence is symplectic, given $K(\theta)$, $\bold{v}_{s}(\theta)$, and $\bold{v}_{u}(\theta)$, we can find $\bold{v}_{c}(\theta)$ solving Eq. \eqref{col2} for some function $T: \mathbb{T} \rightarrow \mathbb{R}$; we postpone the description and proof of how to compute such a $\bold{v}_{c}(\theta)$ to Section \ref{continuationSection} where the method will be used. Any such $\bold{v}_{c}(\theta)$ is known as a symplectic conjugate to $DK(\theta)$, and is a center direction to the torus $K$ \citep{Llave_2005}. Hence, column 2 of $P$ satisfies Eq. \eqref{bundleEquations} if and only if it is a symplectic conjugate center bundle. We should note that Eq. \eqref{bundleEquations} is actually underdetermined; symplectic conjugates are not unique, and we can change the scales of the stable and unstable bundles at each $\theta$; we will take advantage of this in Section \ref{constantLambda} to make $\Lambda$ constant. 

\subsection{Summary of Steps for Quasi Newton-Method for Tori and Bundles}

We will now develop our quasi-Newton method for solving Eq. \eqref{invariance} and \eqref{bundleEquations}. Before presenting the details of the method, we give a brief overview. Assume we have an approximate solution $(K,P,\Lambda)$ for Eq. \eqref{invariance}-\eqref{bundleEquations}. Then, we will
\begin{enumerate}
\item Compute $E(\theta) = F(K(\theta)) - K(\theta+\omega)$, $E_{red}(\theta) = P^{-1}(\theta+\omega)DF(K(\theta)) P(\theta) -  \Lambda(\theta)$
\item Solve $-P^{-1}(\theta+\omega)E(\theta) = \Lambda(\theta)\xi(\theta) -   \xi(\theta+\omega)$ for $\xi:\mathbb{T} \rightarrow \mathbb{R}^{4}$ using Eq. \eqref{xi1}-\eqref{xi4} and set $K(\theta)$ equal to $K(\theta) + P(\theta) \xi(\theta)$ (details given in Section \ref{Kstep}). 
\item Set column 1 of $P(\theta)$ to $DK(\theta)$. Recompute $DF(K(\theta))$ and $E_{red}(\theta)$. 
\item Solve $-E_{red}(\theta) = \Lambda(\theta) Q(\theta) - Q(\theta + \omega) \Lambda(\theta) - \Delta \Lambda(\theta)$ for $Q:\mathbb{T} \rightarrow \mathbb{R}^{4 \times 4}$ and $\Delta \Lambda$ using Eq. \eqref{E_LC}-\eqref{E_UU}. Set $P(\theta)$ equal to $P(\theta) + P(\theta) Q(\theta)$ and $\Lambda(\theta)$ equal to $\Lambda(\theta) + \Delta \Lambda(\theta)$ (details given in Section \ref{Pstep}).  
\item Return to step 1 and repeat correction until $E$ and $E_{red}$ are within tolerance. 
\end{enumerate} 

\subsection{Quasi-Newton Step for Correcting $K$} \label{Kstep}

We seek to solve Eq. \eqref{invariance} and \eqref{bundleEquations} for $K$, $P$, and $\Lambda$. All the entries of $\Lambda$ are fixed as 0 or 1 as shown in Eq. \eqref{Lambdaform} except for $T(\theta)$, $\lambda_s(\theta)$, and $\lambda_u(\theta)$. We will now derive an iterative step that, given an approximate solution $(K,P,\Lambda)$ of Eq. \eqref{invariance} and \eqref{bundleEquations}, produces a much more accurate one. Define the errors
\begin{equation} \label{Edef} E(\theta) = F(K(\theta)) - K(\theta+\omega) \end{equation}
\begin{equation} \label{Ereddef} E_{red}(\theta) = P^{-1}(\theta+\omega)DF(K(\theta)) P(\theta) -  \Lambda(\theta) \end{equation} 
We then need to find corrections $\Delta K$, $\Delta P$, and $\Delta \Lambda$ to cancel $E$ and $E_{red}$. We start with $\Delta K$; write $\Delta K(\theta) = P(\theta) \xi(\theta)$. We will solve for $\xi : \mathbb{T} \rightarrow \mathbb{R}^{4}$ satisfying
	\begin{equation} \label{xiEquation} \eta(\theta) \stackrel{\text{def}}{=} -P^{-1}(\theta+\omega)E(\theta) = \Lambda(\theta)\xi(\theta) -   \xi(\theta+\omega) \end{equation}  
\begin{claim}
For $\omega$ sufficiently irrational and $E$ and $E_{red}$ sufficiently small, if $\xi$ solves Eq. \eqref{xiEquation}, then adding $\Delta K = P \xi$ to $K$ reduces the error $E$ quadratically.
\end{claim}
\begin{remark}
We use the phrase ``sufficiently irrational" when describing conditions on $\omega$ that ensure the validity of our quasi-Newton method. For those aware of KAM theory, what we mean by this is that $\omega$ is Diophantine, as most numbers are \citep{kamTutorial}. This condition is useful due to the classic small-divisors problem when solving cohomological equations (see Eq. \eqref{cohomological}). 
\end{remark}
\begin{proof}
Substitute $K(\theta)+\Delta K(\theta)$ into the RHS of Eq. \eqref{Edef}. Assuming that $\Delta K$ is small enough (true for $E$ sufficiently small and $\omega$ sufficiently irrational), we can expand Eq. \eqref{Edef} in Taylor series to get 
\begin{align} \label{deriveDeltaK} \begin{split}
E_{new}&(\theta) = F(K(\theta) + \Delta K(\theta)) - [K(\theta+\omega) + \Delta K(\theta+\omega)] \\
=&F(K(\theta)) +DF(K(\theta))\Delta K(\theta) + \mathcal O(\Delta K(\theta)^{2}) - [K(\theta+\omega) + \Delta K(\theta+\omega)] \\
=&E(\theta) +DF(K(\theta)) \Delta K(\theta) - \Delta K(\theta+\omega) + \mathcal O(\Delta K(\theta)^{2})\\ 
\end{split} \end{align}  
$\Delta K = P \xi$, and Eq.  \eqref{Ereddef} implies $DF(K(\theta)) P(\theta) = P(\theta+\omega) \left[ \Lambda(\theta) + E_{red}(\theta) \right]$. Thus,
\begin{align} \begin{split} \label{xiWithQuadratic} E_{new}&(\theta) = E(\theta) + DF(K(\theta)) P(\theta) \xi(\theta) - P(\theta+\omega)  \xi(\theta+\omega) + \mathcal O(\xi(\theta)^{2}) \\
& =E(\theta) + P(\theta+\omega) \left[ \Lambda(\theta)\xi(\theta) + E_{red}(\theta)\xi(\theta) - \xi(\theta+\omega)  \right]  + \mathcal O(\xi(\theta)^{2}) \\
&=  P(\theta+\omega) E_{red}(\theta)\xi(\theta) + \mathcal O(\xi(\theta)^{2})  \\
\end{split} \end{align} 
where the last line follows from Eq. \eqref{xiEquation}. For $\omega$ sufficiently irrational, $\xi$ will be similar in magnitude to $E$, so $E_{red}(\theta) \xi(\theta)$ will be quadratically small, comparable to $E_{red}(\theta)E(\theta)$. Hence, as long as $E$ (and hence $\xi$ and $\Delta K$) are small enough that the Taylor expansion in Eq. \eqref{deriveDeltaK} is valid, and the $\mathcal  O(\xi^{2})$ terms of the Taylor expansion are small, the new error $E_{new}$ will be quadratically smaller than $E$. \qed
\end{proof}

To solve Eq. \eqref{xiEquation}, let $\xi(\theta) = \begin{bmatrix} \xi_{1} & \xi_{2}  & \xi_{3} & \xi_{4} \end{bmatrix}^{T}$ and $\eta_(\theta) = \begin{bmatrix} \eta_{1} & \eta_{2}  & \eta_{3} & \eta_{4} \end{bmatrix}^{T}$. As $\Lambda$ is nearly diagonal, we can write Eq. \eqref{xiEquation}  component-wise as 
		\begin{gather}  \label{xi1} \eta_{1}(\theta)-T(\theta) \xi_{2}(\theta)= \xi_{1}(\theta) -   \xi_{1}(\theta+\omega) \\
	  \label{xi2} \eta_{2}(\theta) = \xi_{2}(\theta) -   \xi_{2}(\theta+\omega) \\
	  \label{xi3} \eta_{3}(\theta) = \lambda_{s}(\theta)\xi_{3}(\theta) -   \xi_{3}(\theta+\omega) \\
	  \label{xi4} \eta_{4}(\theta) = \lambda_{u}(\theta)\xi_{4}(\theta) -   \xi_{4}(\theta+\omega) 
	  \end{gather}  
\subsubsection{Fixed-Point Iteration: Solving for $\xi_{3}$ and $\xi_{4}$} \label{fixedPointIter}
To solve for $\xi_{3}$ and $\xi_{4}$, rewrite Equations \eqref{xi3} and \eqref{xi4} in the form 
\begin{equation} \label{xi3contract} \xi_{3}(\theta) = \lambda_{s}(\theta-\omega)\xi_{3}(\theta-\omega) -  \eta_{3}(\theta-\omega)  \stackrel{\text{def}}{=} [A(\xi_{3})](\theta)\end{equation}    
\begin{equation} \label{xi4contract} \xi_{4}(\theta) = \lambda_{u}^{-1}(\theta) \left[ \eta_{4}(\theta) +   \xi_{4}(\theta+\omega) \right]  \stackrel{\text{def}}{=}  [B(\xi_{4})](\theta) \end{equation}  
We define $A$ as a map from functions to functions, which sends any $f(\theta): \mathbb{T} \rightarrow \mathbb{R}$ to the new function $[A(f)](\theta) = \lambda_{s}(\theta-\omega)f(\theta-\omega) -  \eta_{3}(\theta-\omega)$; $B$ is defined similarly using Eq. \eqref{xi4contract}. To find $\xi_{3}$ and $\xi_{4}$, let $\xi_{3,0} = \xi_{4,0} = 0$ and iterate $\xi_{3,n+1} = A(\xi_{3,n})$ and $\xi_{4,n+1} = B(\xi_{4,n})$ repeatedly, starting at $n=0$. The iterations will converge to the desired solutions $\xi_{3}$ and $\xi_{4}$ of Eq. \eqref{xi3contract} and \eqref{xi4contract}. We now explain why. 
\begin{lemma}
$A$, $B$ are contraction maps, hence the iterations $\xi_{3,n+1} = A(\xi_{3,n})$, $\xi_{4,n+1} = B(\xi_{4,n})$ \hl{uniformly} converge exponentially fast as $n \rightarrow \infty$ to the solutions $\xi_{3}$ and $\xi_{4}$. 
\end{lemma}
\begin{proof}
Note that $|\lambda_{s}(\theta)| < 1$ and $|\lambda_{u}^{-1}(\theta)| < 1$ for all $\theta \in \mathbb{T}$. Let $f_{1},f_{2} : \mathbb{T} \rightarrow \mathbb{R}$ be two continuous functions, and define $C = \max_{\theta \in \mathbb{T}} | \lambda_{s}(\theta) | < 1$. We have that
\begin{align} \begin{split}  \max_{\theta \in \mathbb{T}}  \| [A(&f_{1})](\theta) - [A(f_{2})](\theta) \| \\
&=  \max_{\theta \in \mathbb{T}} \| \lambda_{s}(\theta-\omega)f_{1}(\theta-\omega) -  \lambda_{s}(\theta-\omega)f_{2}(\theta-\omega)  \|  \\
&\leq  C \max_{\theta \in \mathbb{T}} \| f_{1}(\theta-\omega) - f_{2}(\theta-\omega)  \| = C \max_{\theta \in \mathbb{T}} \| f_{1}(\theta) - f_{2}(\theta) \| 
\end{split} \end{align} 
As $C<1$, $A$ is a contraction map under the uniform norm; the same can be shown for $B$ very similarly. The contraction mapping theorem \citep{chicone2006} tells us that every such map has a unique fixed point; furthermore, the fixed point can be found by iterating any value in the domain of the map forwards until convergence. The solutions of Equations \eqref{xi3contract} and \eqref{xi4contract} are by definition the fixed points of contraction maps $A$ and $B$. Hence, the iterations converge to $\xi_{3}$ and $\xi_{4}$. \qed
\end{proof}
\begin{remark} \label{fourierRemark1}
If $\lambda_{s}$ and $\lambda_{u}$ are constant, Fourier methods (see Section \ref{cohomsection}) can also be used to solve Eq. \eqref{xi3}-\eqref{xi4}. This is useful if $\lambda_{s}, \lambda_{u} \approx 1$. When using the quasi-Newton method for continuation, one can ensure constant $\lambda_{s}$, $\lambda_{u}$ throughout the correction by applying the procedure of Section \ref{constantLambda} to the solution $K, P, \Lambda$ used for continuation initialization, and following the instructions in Remark \ref{keepLambdaConstant} when correcting $P, \Lambda$ during each quasi-Newton step. We did not ensure constant $\lambda_{s}$, $\lambda_{u}$ in our algorithm implementation, and used fixed-point iteration instead. 
\end{remark}
\subsubsection{Cohomological Equations: Solving for $\xi_{1}$ and $\xi_{2}$} \label{cohomsection}

We next solve Eq. \eqref{xi2} for $\xi_{2}$, which is then used in the LHS of Eq. \eqref{xi1} to solve for $\xi_{1}$. In both cases, we must solve cohomological equations of form
\begin{equation} \label{cohomological} b(\theta) = a(\theta) - a(\theta+\omega) \end{equation}
where $b$ is known and $a$ is not. This can easily be solved by Fourier series; let $\hat a(k)$ and $\hat b(k)$ be the $k$th Fourier coefficients of $a$ and $b$. Then, Eq. \eqref{cohomological} becomes
\begin{equation} \sum_{k \in \mathbb{Z}} \hat b(k) e^{jk\theta} = \sum_{k \in \mathbb{Z}} \hat a(k) e^{jk\theta}  - \sum_{k \in \mathbb{Z}} \hat a(k) e^{jk(\theta+\omega)} = \sum_{k \in \mathbb{Z}} \hat a(k)(1-e^{jk\omega}) e^{jk\theta} \end{equation}
where $j=\sqrt{-1}$. Then, setting $\hat a(k) = \hat b(k) (1-e^{jk\omega})^{-1}$ allows us to compute $a(\theta)$ except for $\hat a(0)$; the formal series for $a$ thus defined will converge on $\mathbb{T}$ for $\omega$ sufficiently irrational \citep{Russ75}. Observe that a necessary condition for the existence of a solution is $\hat b(0)=0$; in the $k=0$ case, $\hat a(0)$ cancels out on the right hand side of Eq. \eqref{cohomological} and can hence take any value, making the solution $a$ non-unique. $\hat a(0)$ and $\hat b(0)$ are simply the averages of $a$ and $b$ on $\mathbb{T}$. 

We first solve Eq. \eqref{xi2} for $\hat \xi_{2} (k)$, $k \neq 0$, using the Fourier series method. To set $\hat \xi_{2} (0) $, first find the average $\alpha$ of $\eta_{1} - T \times [\xi_{2} - \hat \xi_{2} (0)]$. Then, choose $\hat \xi_{2} (0) = \alpha / \hat T(0)$; this makes the LHS of Eq. \eqref{xi1} have zero average when solving for $\xi_{1}$, since
\begin{align} \begin{split} \int_0^{2\pi} \eta_{1} - T \xi_{2} \, d\theta &= \int_0^{2\pi}  \eta_{1} - T \left[\xi_{2} - \hat \xi_{2} (0)\right] \, d\theta -  \hat \xi_{2}(0) \int_0^{2\pi}  T \, d\theta \\ 
&= \alpha - \hat \xi_{2} (0) \hat T(0) =0 \end{split} \end{align}
 With $\xi_{2}$ fully solved, we then solve Eq. \eqref{xi1} for $\hat \xi_{1} (k)$, $k \neq 0$ and arbitrarily choose $\hat \xi_{1} (0) = 0$. Finally, with all four components of $\xi$ solved, we set $K(\theta)$ equal to $K(\theta) + P(\theta) \xi(\theta)$, concluding the $K$ correction part of the quasi-Newton step. 

In practice, when solving Eq. \eqref{xi2} for $\xi_{2}$, we find that the average of $\eta_{2}(\theta)$, the left hand side of Eq. \eqref{xi2}, is not exactly zero; we ignore this nonzero average and solve for the $\hat \xi_{2} (k)$ anyways as described earlier. $\hat \eta_{2}(0)$ decreases to zero with each quasi-Newton step as the method converges, so we are able to solve Eq. \eqref{xi2} more and more exactly; this is a result of the vanishing lemma of \citet{fontichLlaveSire}, which is applicable since $F$ is an exact symplectic map due to being the fixed-time map of a Hamiltonian system on $\mathbb{R}^{4}$ \citep{gole}. Also, for those familiar with the parameterization method for invariant tori, note that the choice of $\hat \xi_{1} (0) = 0$ takes care of the translation non-uniqueness of solutions of Eq. \eqref{invariance} without requiring extra constraint equations. 

\begin{remark}
There are methods of numerically solving for $\Delta K$ without using $P$ or $\Lambda$, including single-shooting \citep{Farres2017} and collocation \citep{olikaraThesis}. These methods effectively discretize $\theta$ on a grid of $N$ points and solve a linearized version of Eq. \eqref{deriveDeltaK} directly for $\Delta K$ at those $\theta$ values. This requires solving at least a $4N$ dimensional linear system at each correction step. Gaussian elimination applied to this will hence have a computational complexity of $O(N^{3})$ and require $O(N^{2})$ storage. However, by using $P$ and the nearly diagonal $\Lambda$, we decouple the equations and avoid this large dimensional system. The complexity of our quasi-Newton method is $O(N \log N)$ (as some steps use FFT), with $O(N)$ required storage. Furthermore, our method gives not just $K$, but also the bundle and Floquet matrices $P$ and $\Lambda$. The most expensive step in our method is the computation (using numerical integration) of $F$ and $DF$ on the grid of $N$ different values of $\theta$, which is easy to parallelize on the computer. 
\end{remark}

\subsection{Quasi-Newton Step for Correcting $P$ and $\Lambda$} \label{Pstep}

Using the newly computed $K(\theta)$, we set the first column of $P(\theta)$ to $DK(\theta)$, and then recompute $DF(K(\theta))$ and $E_{red}(\theta)$ using Eq. \eqref{Ereddef}. Finding $\Delta P(\theta)$ and $\Delta \Lambda(\theta)$ to cancel $E_{red}$ then follows similar methodology as $\Delta K$. Let $\Delta P (\theta) = P(\theta) Q(\theta)$; we will solve for $Q$ and $\Delta \Lambda : \mathbb{T} \rightarrow \mathbb{R}^{4 \times 4}$ satisfying
\begin{equation}  \label{qEquation} -E_{red}(\theta) = \Lambda(\theta) Q(\theta) - Q(\theta + \omega) \Lambda(\theta) - \Delta \Lambda(\theta)\end{equation}  
\begin{claim}
For $\omega$ sufficiently irrational and $E_{red}$ sufficiently small, if $Q$ and $\Delta \Lambda$ solve Eq. \eqref{qEquation}, then adding $\Delta P = P Q$ to $P$ and $\Delta \Lambda$ to $\Lambda$ reduces $E_{red}$ quadratically.
\end{claim}
\begin{proof}
Substitute $P+PQ$ and $\Lambda+\Delta \Lambda$ into Eq. \eqref{bundleEquations} to define
\begin{align} \begin{split} 
\mathcal{E}(\theta)=DF(K(&\theta)) [P(\theta)+P(\theta)Q(\theta)] \\
&- [P(\theta+\omega)+P(\theta+\omega)Q(\theta+\omega)] [\Lambda(\theta)+\Delta \Lambda(\theta)] \\
\end{split} \end{align}  
Using $E_{red}(\theta) = P^{-1}(\theta+\omega)DF(K(\theta)) P(\theta) -  \Lambda(\theta)$, we then find that
\begin{align} \begin{split} \label{bundleDerive}
P(\theta&+\omega)^{-1} \mathcal{E}(\theta) =E_{red}(\theta)+P(\theta+\omega)^{-1}DF(K(\theta))P(\theta)Q(\theta) \\
&\quad \quad \quad \quad \quad \quad \, \, \, \quad \quad \quad \quad \quad \quad \quad \quad - Q(\theta+\omega) [\Lambda(\theta)+\Delta \Lambda(\theta)] - \Delta \Lambda(\theta) \\
  &=E_{red}(\theta)+[\Lambda(\theta)+E_{red}(\theta) ]Q(\theta) - Q(\theta+\omega) [\Lambda(\theta)+\Delta \Lambda(\theta)] - \Delta \Lambda(\theta) \\
    &=E_{red}(\theta) Q(\theta) - Q(\theta+\omega) \Delta \Lambda(\theta) 
\end{split} \end{align}  
where the last line follows due to Eq. \eqref{qEquation}. Evaluating Eq. \eqref{Ereddef} with $P+PQ$ and $\Lambda+\Delta \Lambda$ in place of $P$ and $\Lambda$ and denoting the result as $E_{red,new}$, we have 
\begin{align} \begin{split} E_{red,new}(\theta)&=[P(\theta+\omega)+P(\theta+\omega)Q(\theta+\omega)]^{-1}\mathcal{E}(\theta) \\
&=[I+Q(\theta+\omega)]^{-1}P(\theta+\omega)^{-1}\mathcal{E}(\theta) \\
&=[I+Q(\theta+\omega)]^{-1} [E_{red}(\theta) Q(\theta) - Q(\theta+\omega) \Delta \Lambda(\theta)]  \\
\end{split} \end{align}  
Now, for $\omega$ sufficiently irrational, $Q$ and $\Delta \Lambda$ will be similar in magnitude to $E_{red}$. Hence, if $E_{red}$ is small, then $E_{red,new}$ will be quadratically smaller like $E_{red}^{2}$.  \qed
\end{proof}

Since $\Lambda$ is nearly diagonal, the equations for the different entries of $Q$ and $\Delta \Lambda$ following from Eq. \eqref{qEquation} are almost completely decoupled from each other. Write 
\begin{equation} \begin{gathered} \label{Ecomponents} E_{red}(\theta) = \begin{bmatrix}
E_{LL}(\theta) &  E_{LC}(\theta)    & E_{LS}(\theta)  & E_{LU}(\theta)  \\ E_{CL}(\theta) &  E_{CC}(\theta)    & E_{CS}(\theta)  & E_{CU}(\theta) \\ E_{SL}(\theta) & E_{SC}(\theta)    & E_{SS}(\theta)  & E_{SU}(\theta) \\ E_{UL}(\theta) &  E_{UC}(\theta)    & E_{US}(\theta)  & E_{UU}(\theta) \end{bmatrix} \\ 
Q(\theta) = \begin{bmatrix}
0&  Q_{LC}(\theta)    & Q_{LS}(\theta)  & Q_{LU}(\theta)  \\ 0 &  Q_{CC} (\theta)   & Q_{CS}(\theta)  & Q_{CU}(\theta) \\ 0 & Q_{SC} (\theta)   & Q_{SS} (\theta) & Q_{SU}(\theta) \\ 0 &  Q_{UC}(\theta)    & Q_{US}(\theta)  & Q_{UU}(\theta) \end{bmatrix} \quad \Delta \Lambda(\theta) = \begin{bmatrix}
0 &  \Delta T(\theta)   & 0 & 0 \\ 0 &  0   & 0 & 0 \\ 0 & 0  & \Delta \lambda_s(\theta) & 0 \\ 0 &  0 & 0 & \Delta \lambda_u(\theta) \end{bmatrix} \end{gathered} \end{equation}
As the first column of $P(\theta)$ is fixed to be $DK(\theta)$, we fix the first column of $Q(\theta)$ to be zero so that the first column of $\Delta P$ is zero as well. We can then write columns 2-4 of Eq. \eqref{qEquation} entry by entry to get 12 scalar equations 
\begin{gather}  \label{E_LC} -E_{LC}(\theta)-T(\theta) Q_{CC}(\theta)= Q_{LC}(\theta) - Q_{LC}(\theta+\omega) - \Delta T(\theta) \\
 \label{E_LS} -E_{LS}(\theta) -T(\theta) Q_{CS}(\theta)= Q_{LS}(\theta) - \lambda_{s}(\theta) Q_{LS}(\theta+\omega) \\
 \label{E_LU} -E_{LU}(\theta) -T(\theta) Q_{CU}(\theta)= Q_{LU}(\theta) - \lambda_{u}(\theta) Q_{LU}(\theta+\omega)  \\
 \label{E_CC} -E_{CC}(\theta) = Q_{CC}(\theta) - Q_{CC}(\theta+\omega) \\
 \label{E_CS} -E_{CS}(\theta) = Q_{CS}(\theta) - \lambda_{s}(\theta) Q_{CS}(\theta+\omega) \\
 \label{E_CU} -E_{CU}(\theta) = Q_{CU}(\theta) - \lambda_{u}(\theta) Q_{CU}(\theta+\omega)  \\
  \label{E_SC} -E_{SC}(\theta) = \lambda_{s}(\theta) Q_{SC}(\theta) - Q_{SC}(\theta+\omega) \\
 \label{E_SS} -E_{SS}(\theta) = \lambda_{s}(\theta) Q_{SS}(\theta) - \lambda_{s}(\theta) Q_{SS}(\theta+\omega) - \Delta \lambda_{s}(\theta) \\
 \label{E_SU} -E_{SU}(\theta) = \lambda_{s}(\theta) Q_{SU}(\theta) - \lambda_{u}(\theta) Q_{SU}(\theta+\omega)  \\
  \label{E_UC} -E_{UC}(\theta) = \lambda_{u}(\theta) Q_{UC}(\theta) - Q_{UC}(\theta+\omega) \\
 \label{E_US} -E_{US}(\theta) = \lambda_{u}(\theta) Q_{US}(\theta) - \lambda_{s}(\theta) Q_{US}(\theta+\omega) \\
 \label{E_UU} -E_{UU}(\theta) = \lambda_{u}(\theta) Q_{UU}(\theta) - \lambda_{u}(\theta) Q_{UU}(\theta+\omega) - \Delta \lambda_{u}(\theta)
\end{gather}  

First, we solve Equations \eqref{E_CS}, \eqref{E_CU}, \eqref{E_SC}, \eqref{E_SU},  \eqref{E_UC}, and \eqref{E_US}, followed by Eq. \eqref{E_LS} and \eqref{E_LU} (after back substitution), using the exact same method that was used to solve Eq. \eqref{xi3} and \eqref{xi4}; rearrange each equation so that its solution is the fixed point of an appropriately defined contraction map, which is then iterated to convergence. Such maps always multiply their input by either $\lambda_{s}$ or $\lambda_{u}^{-1}$, or both. Eq. \eqref{E_CC} is solved using the Fourier method of Section \ref{cohomsection}; we arbitrarily choose $\hat Q_{CC}(0) = 0$. We ignore the nonzero average of $E_{CC}$, which goes to zero with each quasi-Newton step without affecting method convergence due to $F$ being symplectic (see Appendix \ref{appendix} for the proof of this result). Finally, the solutions to Eq. \eqref{E_LC}, \eqref{E_SS}, and \eqref{E_UU} are non-unique; we choose $Q_{LC}=Q_{SS}=Q_{UU}=0$, so that we have $\Delta T(\theta)=E_{LC}(\theta)+T(\theta) Q_{CC}(\theta) $, $\Delta \lambda_{s}(\theta)=E_{SS}(\theta)$, and $\Delta \lambda_{u}(\theta)=E_{UU}(\theta)$. 

Once $Q$ and $\Delta \Lambda$ are known, we set $P(\theta)$ equal to $P(\theta) + P(\theta) Q(\theta)$, $\Lambda(\theta)$ equal to $\Lambda(\theta) + \Delta \Lambda(\theta)$, and then recompute $E(\theta)$ and $E_{red}(\theta)$ using Eq. \eqref{Edef} and \eqref{Ereddef}. Finally, we go back to the quasi-Newton step for correcting the torus parameterization $K(\theta)$ and repeat the entire method until $E$ and $E_{red}$ are within tolerance. 

\begin{remark} \label{keepLambdaConstant}
If $T$, $\lambda_{s}$, and $\lambda_{u}$ are constant, we can choose the non-unique solutions of Eq. \eqref{E_LC}, \eqref{E_SS}, and \eqref{E_UU} such that they remain constant. In particular, choose $\Delta T$, $\Delta \lambda_{s}$, and $\Delta \lambda_{u}$ as the (constant) averages of $E_{LC}$, $E_{SS}$ and $E_{UU}$, respectively, and solve for $Q_{LC}$, $Q_{SS}$, and $Q_{UU}$ using Fourier methods. Our experience was that this choice of solution negatively affected the numerical stability of our method, however; thus, we did not keep $T$, $\lambda_{s}$, and $\lambda_{u}$ constant in our implementation.
\end{remark}

\subsection{A Remark on Convergence} 

The focus of this paper is to specify the algorithms, provide details of implementation, and give practical results of the implementation in physical problems. Nevertheless, we wish to mention that there are results which rigorously prove that our algorithm converges when given initial $K,P,\Lambda$ with small enough error (depending on some condition numbers). Due to the practical focus of this paper, we will not go into detail, but we want to give a flavor of the argument. For readers interested primarily in applications, this section can be skipped. 

The convergence is due to the so-called Kolmogorov-Arnold-Moser (KAM) theory, which is a very far reaching generalization of the Newton method. In particular, we take advantage of the recent developments in a-posteriori versions of KAM theory \citep{fontichLlaveSire}, which does not require an integrable system, only approximate solutions of functional equations. We present some salient features. For any analytic function of an angle $u:\mathbb{T} \rightarrow \mathbb{C}^{n}$, define $\|u \|_\rho = \sup_{ | \operatorname{Im} z| \le \rho} |u(z) |$. It is possible to show \citep{Russ75}  that the solutions of \eqref{cohomological} satisfy  $ \|a\|_{\rho -\delta} \le C_{1} \delta^{-\tau} \|b\|_\rho$ for some $C_{1}, \tau > 0$. That is, if the right hand side is analytic in a certain complex domain, the solution is analytic in a slightly smaller complex domain, and we have estimates of the size in terms of the domain lost; note that both domains contain all real angles from 0 to $2\pi$, which is what we are actually interested in. The well known Cauchy estimates \citep{ahlfors} for derivatives of a function in a slightly smaller domain have the same form. 

The formal procedure we have given indeed reduces $E$ and $E_{red}$, to something quadratically smaller, but, performing the estimates with care, only in a slightly smaller complex domain. Denoting the invariance error and the reducibility errors after one quasi-Newton step by $ E_{new}$ and $ E_{red,new}$, we have that 
\begin{equation} \label{tameestimates} 
\| E_{new} \|_{\rho - \delta} +  \| E_{red,new} \|_{\rho -\delta}
  \le C_{2} \delta^{-2\tau-2} 
\left( \|E \|_{\rho } +  \| E_{red} \|_{\rho} \right)^2
\end{equation} 
for some $C_2 > 0$. There are standard arguments in KAM theory (\emph{``hard implicit function 
theorems''}, see \citet{kamTutorial}) which show that, given an algorithm satisfying Eq. \eqref{tameestimates} and a sufficiently small initial error, the algorithm step can be iterated infinitely many times to convergence in a domain slightly smaller than the original.  These estimates also show that the final answer is close to the initial approximation of $K,P,\Lambda$ if the initial error is small enough. 

\subsection{Modifying $P$ for Constant $\Lambda$} \label{constantLambda}

Let $K$, $P$, and $\Lambda$ be a solution to Eq. \eqref{invariance}-\eqref{bundleEquations}. For purposes of numerical stability as well as stable/unstable manifold computation (see Section \ref{parambigsection}), it can be useful to modify columns 2, 3, and 4 of $P$ in such a way that $\Lambda= P^{-1}(\theta+\omega)DF(K(\theta)) P(\theta)$ becomes a constant matrix of the form in Eq. \eqref{Lambdaform}, i.e. $T(\theta)$, $\lambda_{s}(\theta)$, and $\lambda_{u}(\theta)$ become constant. This can also enable the usage of Fourier methods instead of fixed point iteration during the quasi-Newton method (see Remarks \ref{fourierRemark1} and \ref{keepLambdaConstant}). 

For columns 3 and 4 of $P$, the stable and unstable bundles $\bold{v}_{s}(\theta)$ and $\bold{v}_{u}(\theta)$ respectively, just a simple rescaling is needed to make $\lambda_{s}(\theta)$ and $\lambda_{u}(\theta)$ constant. Let $\bar \lambda_s, \bar \lambda_u \in \mathbb{R}$ and $a_s,a_u:\mathbb{T} \rightarrow \mathbb{R}$ be the solutions to
\begin{gather} \label{scalevs} \log (\lambda_s(\theta)) - \log (\bar \lambda_s) = \log(a_s(\theta+\omega))- \log(a_s(\theta)) \\
 \label{scalevu}  \log (\lambda_u(\theta)) - \log (\bar \lambda_u) =\log(a_u(\theta+\omega))  - \log(a_u(\theta)) \end{gather}
We choose $\bar \lambda_s = \exp \left[ \frac{1}{2\pi}\int_{0}^{2\pi} \log (\lambda_s(\theta)) \, d\theta \right]$ so the LHS of Eq. \eqref{scalevs} has zero average. Letting $u(\theta) = \log(a_s(\theta))$, Eq. \eqref{scalevs} becomes a cohomological equation of form Eq. \eqref{cohomological} which can be solved for $u$ by the Fourier series method; we choose $\hat u(0)=0$. This gives $a_{s}(\theta)=e^{u(\theta)}$. We can solve Eq. \eqref{scalevu} for $a_{u}(\theta)$ in the exact same manner. Finally, one can replace columns 3 and 4 of $P$ by $a_s(\theta) \bold{v}_{s}(\theta)$ and $ a_u(\theta) \bold{v}_{u}(\theta)$, and replace $\lambda_{s}(\theta)$ and $\lambda_{u}(\theta)$ in $\Lambda$ by $\bar \lambda_s$ and $\bar \lambda_{u}$. We prove that this works now. 
\begin{lemma}
If $\bold{v}_{s}(\theta)$ and $\bold{v}_{u}(\theta)$ satisfy Eq. \eqref{col3}-\eqref{col4}, and $a_{s}(\theta)$ and $a_{u}(\theta)$ satisfy Eq. \eqref{scalevs}-\eqref{scalevu}, then $\bold{v}_{s,new}(\theta) =  a_{s}(\theta) \bold{v}_{s}(\theta) $ and $\bold{v}_{u,new}(\theta) =  a_{u}(\theta) \bold{v}_{u}(\theta) $ satisfy
\begin{equation} DF(K(\theta)) \bold{v}_{s,new}(\theta) = \bar \lambda_s \bold{v}_{s,new}(\theta+\omega) \end{equation}
\begin{equation} DF(K(\theta)) \bold{v}_{u,new}(\theta) = \bar \lambda_u \bold{v}_{u,new}(\theta+\omega) \end{equation}
\end{lemma}
\begin{proof}
We prove the result for $\bold{v}_{s,new}$; the case of $\bold{v}_{u,new}$ can be proven in the exact same manner. Since $\bold{v}_{s}(\theta)$ satisfies Eq. \eqref{col3}, we have
\begin{align} \begin{split} DF(K(\theta)) \bold{v}_{s,new}(\theta) &= DF(K(\theta)) a_{s}(\theta) \bold{v}_{s}(\theta)  =a_{s}(\theta) \lambda_{s}(\theta) \bold{v}_{s}(\theta+\omega)  \\
&= a_{s}(\theta+\omega) \bar \lambda_{s} \bold{v}_{s}(\theta+\omega) = \bar \lambda_{s} \bold{v}_{s,new}(\theta+\omega)   \end{split} \end{align}
where $a_{s}(\theta) \lambda_{s}(\theta) = a_{s}(\theta+\omega) \bar \lambda_{s} $ follows from exponentiating Eq. \eqref{scalevs}. \qed
\end{proof}

We can also modify the second column of $P$, the symplectic conjugate center direction $\bold{v}_{c}(\theta)$, to make $T(\theta)$ constant. This is possible because as mentioned in Section \ref{understandP}, the symplectic conjugate is not unique; given $a:\mathbb{T} \rightarrow \mathbb{R}$ and $\bold{v}_{c}(\theta)$ satisfying Eq. \eqref{col2}, the function $\bold{v}_{c}(\theta) + a(\theta) DK(\theta)$ also solves Eq. \eqref{col2} except with a change in $T(\theta)$ (which was anyways arbitrary). Hence, we choose $a(\theta)$ which kills all variation of $T(\theta)$ about its average $\hat T(0)$. The equation for this is:
\begin{equation} \label{Tkill}-(T(\theta) - \hat T(0))  = a(\theta) - a(\theta + \omega) \end{equation}
which can be solved using the Fourier series method given for Eq. \eqref{cohomological}. Then, one simply adds $a(\theta) DK(\theta)$ to column 2 of $P$ and replaces $T(\theta)$ with $ \hat T(0)$ in $\Lambda$. Note that the LHS of Eq. \eqref{Tkill} has average zero, so a solution $a(\theta)$ can be found. 

\begin{lemma}
If $\bold{v}_{c}(\theta)$ and $a(\theta)$ satisfy Eq. \eqref{col2} and \eqref{Tkill}, respectively, then the function $\bold{v}_{c,new}(\theta) = \bold{v}_{c}(\theta) + a(\theta) DK(\theta)$ is also a symplectic conjugate and satisfies
\begin{equation} DF(K(\theta)) \bold{v}_{c,new}(\theta) = \hat T(0) DK(\theta+\omega)+\bold{v}_{c,new}(\theta+\omega) \end{equation}
\end{lemma}
\begin{proof}
Since $\bold{v}_{c}(\theta)$ satisfies Eq. \eqref{col2} and $DK(\theta)$ satisfies Eq. \eqref{tangent}, we have
\begin{align} \begin{split} DF(K(\theta)) \bold{v}_{c,new}(\theta) &= DF(K(\theta)) \left[ \bold{v}_{c}(\theta) + a(\theta) DK(\theta) \right] \\
&= \left[T(\theta)+a(\theta) \right] DK(\theta+\omega)+\bold{v}_{c}(\theta+\omega) \\
&= \left[\hat T(0)+a(\theta+\omega) \right] DK(\theta+\omega)+\bold{v}_{c}(\theta+\omega) \\
&= \hat T(0) DK(\theta+\omega)+\bold{v}_{c,new}(\theta+\omega) \end{split} \end{align}
where the relation $T(\theta)+a(\theta) = \hat T(0)+a(\theta+\omega)$ follows from Eq. \eqref{Tkill}. \qed
\end{proof}

\subsection{Initialization for Continuation by $\varepsilon$} \label{continuationSection}

To compute invariant circles and bundles of stroboscopic maps in periodically-perturbed PCRTBP models with some desired perturbation parameter $\varepsilon_{f} > 0$, we start from periodic orbits and their bundles in the unperturbed PCRTBP ($\varepsilon = 0$) and continue by $\varepsilon$ until the torus and bundles for $\varepsilon = \varepsilon_{f}$ are found. Our quasi-Newton method-based continuation follows the standard procedure; choose a number of continuation steps $n$, take an invariant circle and bundles from the $\varepsilon = 0$ system, and use them as an initial guess for the quasi-Newton method to solve for the circle and bundles in the $\varepsilon = \varepsilon_{f}/n$ system. Similarly, for $i = 0, \dots, n-1$, use the solution from the $\varepsilon_{f} i/n$ system as an initial guess for the solution in the $\varepsilon_{f} (i+1)/ n$ system. Once $i=n-1$, we have the torus and bundles for $\varepsilon = \varepsilon_{f}$. We need to find the $\varepsilon = 0$ solution to initialize the continuation, however.

To get $K(\theta)$, $P(\theta)$, and $\Lambda(\theta)$ solving Eq. \eqref{invariance} and \eqref{bundleEquations} for the $\varepsilon = 0$ PCRTBP case,  one needs to first choose a periodic orbit which is to be continued (recall that PCRTBP periodic orbits are also invariant circles of the stroboscopic map  $F=F_{\varepsilon = 0}$, unless the orbit period is resonant with $2\pi/\Omega_{p}$). From this periodic orbit, we get its period $T_{1}$ and hence the rotation number $\omega=4\pi^{2}/(T_{1}\Omega_{p})$, as well as a point $\bold{x}_{0}$ lying on the orbit. Let $\phi(\bold{x}, t)$ denote the time-$t$ map of the point $\bold{x} \in \mathbb{R}^{4}$ by the PCRTBP equations of motion. Then, we can take $K(\theta) = \phi(\bold{x}_{0}, T_{1} \frac{\theta}{2\pi})$ if the periodic orbit monodromy matrix stable and unstable eigenvalues are positive. If they are negative, though, the ``double covering" trick of \citet{haroLlave2007} needs be used so that $P(\theta)$ can be continuously defined (as the stable/unstable bundles are M\"obius strips in this case). For this, set $K(\theta) = \phi(\bold{x}_{0}, 2T_{1} \frac{\theta}{2\pi})$ and $\omega=2\pi^{2}/(T_{1}\Omega_{p})$ so that $K$ sweeps over the periodic orbit twice as $\theta$ goes from 0 to $2\pi$. In either case, it is easy to verify that $K(\theta)$ satisfies Eq. \eqref{invariance}. 

Next, set the first column of $P(\theta)$ to be $DK(\theta)$, and set the third and fourth columns of $P(\theta)$ as the stable and unstable unit eigenvectors of the periodic orbit monodromy matrix at the point $K(\theta)$. Denote these stable and unstable eigenvectors as $\bold{v}_{s}(\theta)$ and $\bold{v}_{u}(\theta)$, respectively. One needs to make sure that the directions of $\bold{v}_{s}(\theta)$ and $\bold{v}_{u}(\theta)$ at each point $K(\theta)$ are chosen such that they are continuously oriented functions of $\theta$; this is always possible if $K$ is defined as previously described. Finally, finding the second column of $P$ requires some extra calculations. 

As mentioned in Section \ref{understandP}, the second column of $P$ represents the symplectic conjugate direction to $DK(\theta)$ and is part of the center bundle. The first step in its computation is to compute $\lambda_{s}(\theta)$ and $\lambda_{u}(\theta) : \mathbb{T} \rightarrow \mathbb{R}$ such that 
 \begin{equation} \label{stablelambda} DF(K(\theta)) \bold{v}_{s}(\theta) = \lambda_{s}(\theta) \bold{v}_{s}(\theta+\omega) \end{equation}
\begin{equation} \label{unstablelambda} DF(K(\theta)) \bold{v}_{u}(\theta) = \lambda_{u}(\theta) \bold{v}_{u}(\theta+\omega) \end{equation}
which can be done since $DF(K(\theta))$ maps the stable and unstable bundles into themselves. Next, find functions $A(\theta), B(\theta), C(\theta)$, and $D(\theta) : \mathbb{T} \rightarrow \mathbb{R}$ such that
\begin{align} \label{abcd} \begin{split} DF(K(\theta)) \frac{J^{-1} DK(\theta)}{ \|DK(\theta)\|^{2}} = A(\theta) &DK(\theta+\omega) + B(\theta) \frac{J^{-1} DK(\theta+\omega)}{ \|DK(\theta+\omega)\|^{2}} \\ 
&+ C(\theta) \bold{v}_{s}(\theta+\omega) + D(\theta) \bold{v}_{u}(\theta+\omega) 
\end{split} \end{align} 
where $  J= \begin{bmatrix}
0_{2 \times 2}   & I_{2 \times 2}   \\ -I_{2 \times 2}  &  0_{2 \times 2} \end{bmatrix} $ is the $4 \times 4$ matrix of the symplectic form in the usual Euclidean metric on $\mathbb{R}^{4}$. All the quantities in \eqref{abcd} are known except for $A,B,C,$ and $D$. We can therefore consider Eq. \eqref{abcd} as a system of linear equations for $A, B, C,$ and $D$ which can be solved. One will find that $B(\theta) = 1$; this occurs as a result of symplectic geometric considerations (see Eq. \eqref{BisOne}). After this, we solve for functions $f_{1}(\theta), f_{2}(\theta) : \mathbb{T} \rightarrow \mathbb{R}$ such that
\begin{equation} \label{f1} C(\theta) =  f_{1}(\theta+\omega) - \lambda_{s}(\theta)f_{1}(\theta)  \end{equation}
\begin{equation} \label{f2} D(\theta) =  f_{2}(\theta+\omega) - \lambda_{u}(\theta)f_{2}(\theta)  \end{equation}  
which can be done using the same contraction map iteration method used to solve Equations \eqref{xi3} and \eqref{xi4} in Section \ref{fixedPointIter}. Finally, we can express the second column of $P(\theta)$, the symplectic conjugate direction $\bold{v}_{c}(\theta)$, as 
\begin{equation} \label{sympconj} \bold{v}_{c}(\theta) = \frac{J^{-1} DK(\theta)}{ \|DK(\theta)\|^{2}} + f_{1}(\theta) \bold{v}_{s}(\theta)  + f_{2}(\theta) \bold{v}_{u}(\theta) \end{equation}
With $P(\theta)$ known, to find $\Lambda$ one can simply use $\Lambda(\theta) = P^{-1}(\theta+\omega)DF(K(\theta)) P(\theta) $, after which we can start the continuation. As long as the previous steps were followed correctly, $\Lambda$ will be of the form given in Eq. \eqref{Lambdaform}. To see this, recall the discussion in Section \ref{understandP}, and note that the first, third, and fourth columns of $P$ satisfy Equations \eqref{col1}, \eqref{col3}, and \eqref{col4}. Hence, we just need to show that the second column of $P$ satisfies Eq. \eqref{col2}. We prove this now. 
\begin{lemma} \label{sympconjLemma}
For some $T:\mathbb{T} \rightarrow \mathbb{R}$. the function $\bold{v}_{c}(\theta)$ defined in Eq. \eqref{sympconj} satisfies  \begin{equation}  DF(K(\theta)) \bold{v}_{c}(\theta) = T(\theta) DK(\theta+\omega) + \bold{v}_{c}(\theta+\omega)  \end{equation}
\end{lemma}
\begin{proof}
Applying Eq. \eqref{sympconj} and then Equations \eqref{stablelambda}, \eqref{unstablelambda}, and \eqref{abcd}, we have 
\begin{align} \begin{split} DF(K(\theta)) \bold{v}_{c}(\theta) &= DF(K(\theta)) \left( \frac{J^{-1} DK(\theta)}{ \|DK(\theta)\|^{2}} + f_{1}(\theta) \bold{v}_{s}(\theta)  + f_{2}(\theta) \bold{v}_{u}(\theta) \right) \\
= &A(\theta) DK(\theta+\omega) + B(\theta) \frac{J^{-1} DK(\theta+\omega)}{ \|DK(\theta+\omega)\|^{2}} \\ 
&+ \left(C(\theta) + \lambda_{s}(\theta)f_{1}(\theta) \right) \bold{v}_{s}(\theta+\omega) + \left( D(\theta) + \lambda_{u}(\theta)f_{2}(\theta) \right)\bold{v}_{u}(\theta+\omega) 
\end{split} \end{align} 
Recalling Equations \eqref{f1} and \eqref{f2}, we thus have that
\begin{align} \begin{split} \label{DFtimesvc} DF(K(\theta)) \bold{v}_{c}(\theta) = &A(\theta) DK(\theta+\omega) + B(\theta) \frac{J^{-1} DK(\theta+\omega)}{ \|DK(\theta+\omega)\|^{2}} \\ 
&+ f_{1}(\theta+\omega)  \bold{v}_{s}(\theta+\omega) + f_{2}(\theta+\omega) \bold{v}_{u}(\theta+\omega) 
\end{split} \end{align} 
Flow maps of Hamiltonian systems are symplectic \citep{thirring}. Hence, $F$ satisfies $\Omega(\bold{v}_{1}, \bold{v}_{2}) = \Omega(DF(K(\theta)) \bold{v}_{1}, DF(K(\theta)) \bold{v}_{2})$ for all $\bold{v}_{1}$, $\bold{v}_{2} \in \mathbb{R}^{4}$, where $\Omega$ is the bilinear symplectic form defined on Euclidean $\mathbb{R}^{4}$ as $ \Omega(\bold{v}_{1}, \bold{v}_{2}) =  \bold{v}_{1}^{T} J \bold{v}_{2} $. It is easy to see that $\Omega(\bold{v}_{1}, \bold{v}_{1}) = 0$ for any $\bold{v}_{1} \in \mathbb{R}^{4}$. Furthermore, defining $L = \max_{\theta \in \mathbb{T}} | \lambda_{s}(\theta) | < 1$ and recalling equations \eqref{tangent} and \eqref{stablelambda}, we have that
\begin{align} \begin{split} \label{deriveZero} \max_{\theta \in \mathbb{T}} |\Omega(DK&(\theta),\bold{v}_{s}(\theta)) |= \max_{\theta \in \mathbb{T}} \left| \Omega \left(DF(K(\theta)) DK(\theta), DF(K(\theta) ) \bold{v}_{s}(\theta) \right) \right| \\
&=\max_{\theta \in \mathbb{T}}  \left| \Omega \left(DK(\theta+\omega), \lambda_s(\theta) \bold{v}_{s}(\theta+\omega) \right) \right| \\
&=\max_{\theta \in \mathbb{T}}  |\lambda_s(\theta)|  \left| \Omega \left(DK(\theta+\omega), \bold{v}_{s}(\theta+\omega) \right) \right| \\
& \leq L \max_{\theta \in \mathbb{T}} \left| \Omega \left(DK(\theta+\omega), \bold{v}_{s}(\theta+\omega) \right) \right| = L \max_{\theta \in \mathbb{T}} \left| \Omega \left(DK(\theta), \bold{v}_{s}(\theta) \right) \right| \\
\end{split} \end{align} 
which implies that $\max_{\theta \in \mathbb{T}} \left| \Omega \left(DK(\theta), \bold{v}_{s}(\theta) \right) \right| = 0$ since $0 < L < 1$. Thus, for all $\theta \in \mathbb{T}$, $\Omega \left(DK(\theta), \bold{v}_{s}(\theta) \right) =0$ . We can also show that $\Omega \left(DK(\theta), \bold{v}_{u}(\theta) \right) = 0$ in a very similar manner to Eq. \eqref{deriveZero}. Hence, using Eq. \eqref{sympconj} for $\bold{v}_{c}$, we find 
\begin{align} \label{areaOne} \begin{split} \Omega(DK(\theta),\bold{v}_{c}(\theta)) &= \Omega \left(DK(\theta),\frac{J^{-1} DK(\theta)}{ \|DK(\theta)\|^{2}} + f_{1}(\theta) \bold{v}_{s}(\theta)  + f_{2}(\theta) \bold{v}_{u}(\theta) \right) \\
&= \Omega \left(DK(\theta),\frac{J^{-1} DK(\theta)}{ \|DK(\theta)\|^{2}} \right) \\
&= DK(\theta)^{T}J \frac{J^{-1} DK(\theta)}{ \|DK(\theta)\|^{2}}  = \frac{ DK(\theta)^{T} DK(\theta)}{ \|DK(\theta)\|^{2}} = 1
\end{split} \end{align} 
Since $F$ is a symplectic map, using Eq. \eqref{DFtimesvc} we have that 
\begin{align} \begin{split} \label{BisOne} 1 = \Omega &(DK(\theta),\bold{v}_{c}(\theta)) \\
= \Omega &\left(DF(K(\theta)) DK(\theta),DF(K(\theta)) \bold{v}_{c}(\theta) \right) \\
= \Omega &\left(DK(\theta+\omega),A(\theta) DK(\theta+\omega) + B(\theta) \frac{J^{-1} DK(\theta+\omega)}{ \|DK(\theta+\omega)\|^{2}} \right. \\
& \quad \quad \quad \quad \quad \quad \left. + f_{1}(\theta+\omega)  \bold{v}_{s}(\theta+\omega) + f_{2}(\theta+\omega) \bold{v}_{u}(\theta+\omega) \vphantom{\frac{J^{-1} DK(\theta+\omega)}{ \|DK(\theta+\omega)\|^{2}}} \right) \\
= \Omega &\left(DK(\theta+\omega), B(\theta) \frac{J^{-1} DK(\theta+\omega)}{ \|DK(\theta+\omega)\|^{2}} \right) \\
=B(&\theta) DK(\theta+\omega)^{T} J  \frac{J^{-1} DK(\theta+\omega)}{ \|DK(\theta+\omega)\|^{2}} = B(\theta)
\end{split} \end{align} 
proving that $B(\theta) =1$. Therefore, substituting this into Eq. \eqref{DFtimesvc} gives
\begin{align} \begin{split} \label{subInto} DF(K(\theta)) \bold{v}_{c}(\theta) = &A(\theta) DK(\theta+\omega) +  \frac{J^{-1} DK(\theta+\omega)}{ \|DK(\theta+\omega)\|^{2}} \\ 
&+ f_{1}(\theta+\omega)  \bold{v}_{s}(\theta+\omega) + f_{2}(\theta+\omega) \bold{v}_{u}(\theta+\omega) 
\end{split} \end{align} 
Finally, we see from Eq. \eqref{sympconj} that the last 3 terms on the RHS of Eq. \eqref{subInto} are precisely $\bold{v}_{c}(\theta+\omega)$. Letting $T(\theta) = A(\theta)$, we hence conclude that 
\begin{equation}  DF(K(\theta)) \bold{v}_{c}(\theta) = T(\theta) DK(\theta+\omega) + \bold{v}_{c}(\theta+\omega) \end{equation} 
which is what we sought to prove. \qed
\end{proof}

\subsection{Computing Families of Tori: Continuation by $\omega$} \label{continuationOmegaSection}

The continuation by $\varepsilon$ described in Section \ref{continuationSection} is carried out with $\omega$ fixed. However, in the PCRTBP, periodic orbits occur in one-parameter families, with varying rotation numbers $\omega$ under $F_{\varepsilon=0}$. The same is true of invariant circles of $F_{\varepsilon}$ for $\varepsilon=\varepsilon_{f} > 0$ as well. To compute the family of $F_{\varepsilon_{f}}$-invariant tori corresponding to a PCRTBP periodic orbit family, one option is to continue several different periodic orbits from that family (corresponding to different $\omega$ values) by $\varepsilon$. However, this is inefficient, as the tori and bundles computed for $\varepsilon < \varepsilon_{f}$ are not of interest. Instead, it is better to first compute just one invariant circle of $F_{\varepsilon_{f}}$, along with its bundles, at some rotation number $\omega = \omega_{0}$ using continuation by $\varepsilon$. After this, one can continue the $\omega_{0}$ circle/bundles by $\omega$, with $\varepsilon = \varepsilon_{f}$ fixed. The continuation by $\omega$ is quite similar to the continuation by $\varepsilon$; given a known exact solution to Eq. \eqref{invariance} and \eqref{bundleEquations} for $\omega= \omega_{i}$, $i \in \mathbb{Z}$, one uses this to form an initial guess for the quasi-Newton method to compute the torus/bundles for $\omega = \omega_{i+1}= \omega_{i}+\Delta \omega_{i}$. This recursively gives us tori for a range of $\omega$ values. The $\Delta \omega_{i}$ are called the continuation step sizes. 

One can use the torus and bundles for $\omega= \omega_{i}$ directly as an initial guess for $\omega = \omega_{i}+\Delta \omega_{i}$. However, it is extremely easy to use $K$, $P$, and $\Lambda$ from $\omega_{i}$ to compute a better initial guess for the $\omega_{i}+\Delta \omega_{i}$ torus, which aids in quasi-Newton method convergence. Assume that $\Lambda$ has constant $T(\theta) = T$ (apply the procedure from Section \ref{constantLambda} to $P$ and $\Lambda$ if necessary). Then, using $\bold{v}_{c}(\theta)$ to denote column 2 of $P$, the initial guess for the $\omega_{i}+\Delta \omega_{i}$ torus parameterization should be $K_{new}(\theta) = K(\theta) + (\Delta \omega_{i}/T) \bold{v}_{c}(\theta)$. We justify this now. 

\begin{claim}
If $K$, $P$, and $\Lambda$ (with $\Lambda$ constant) solve Eq. \eqref{invariance}-\eqref{bundleEquations} for $\omega = \omega_{i}$, then $K_{new}(\theta) =K(\theta) +  (\Delta \omega_{i}/T) \bold{v}_{c}(\theta)$ solves Eq. \eqref{invariance} for $\omega = \omega_{i}+\Delta \omega_{i}$ up to $ \mathcal O(\Delta \omega_{i}^{2})$. 
\end{claim}
\begin{proof}
For notational convenience, write $\omega$ and $\Delta \omega$ in place of $\omega_{i}$ and $\Delta \omega_{i}$, respectively. Then, evaluating $F(K_{new}(\theta))-K_{new}(\theta+\omega+\Delta \omega)$, we find this equals 
\begin{align}  \begin{split}
F&\bigg(K(\theta) + \frac{\Delta \omega}{T}  \bold{v}_{c}(\theta)\bigg) - \left[K(\theta+\omega+\Delta \omega) + \frac{\Delta \omega}{T}  \bold{v}_{c}(\theta+\omega+\Delta \omega)\right] \\
&\begin{aligned} 
=F(K(\theta)) +&DF(K(\theta)) \frac{\Delta \omega}{T}  \bold{v}_{c}(\theta) + \mathcal O(\Delta \omega^{2}) \\ 
-& \left[K(\theta+\omega) + \Delta \omega DK(\theta+\omega) + \frac{\Delta \omega}{T}  \bold{v}_{c}(\theta+\omega) + \mathcal  O(\Delta \omega^{2}) \right]  
\end{aligned} \\
&\begin{aligned} 
= \frac{\Delta \omega}{T} \left[   DF(K(\theta)) \bold{v}_{c}(\theta) - T \, DK(\theta+\omega) - \bold{v}_{c}(\theta+\omega) \right] + \mathcal O(\Delta \omega^{2}) = \mathcal O(\Delta \omega^{2})
\end{aligned} \\
\end{split} \end{align}  
where the last equality follows from Eq. \eqref{col2}. \qed
\end{proof}
\begin{remark}
Using the Poincar\'e-Lindstedt method, it is possible to get higher order expansions in $\Delta \omega_{i}$ for $K_{new}$ than the linear approximation $K(\theta) +  (\Delta \omega_{i}/T) \bold{v}_{c}(\theta)$. This requires significant extra computations which we decided not to carry out. 
\end{remark}

There is one more difference between continuation by $\omega$ and continuation by $\varepsilon$. The continuation by $\varepsilon$ uses a fixed step size $\varepsilon_{f}/n$. However, for continuation by $\omega$, the step size $\Delta \omega_{i}$ must be varied due to quasi-Newton method divergence for insufficiently irrational $\omega$. At such $\omega$ values, the PCRTBP invariant circle breaks down after the perturbation $\varepsilon=\varepsilon_{f}$, leading to a gap between tori at smaller and larger rotation numbers. Our continuation needs to ``jump" over this gap. Suppose we have a torus and bundles for $\omega=\omega_{i}$, and let $\varphi_{i}$ denote the largest of $\Delta \omega_{i-1}, \Delta \omega_{i-2}, \dots, \Delta \omega_{i-5}$. It is natural to try $\Delta \omega_{i}=\varphi_{i}$. If the quasi-Newton method diverges for $\omega=\omega_{i}+\varphi_{i}$, however, then instead one can try $\Delta \omega_{i}=\varphi_{i}/2$; if this still does not work, try $\Delta \omega_{i}=\varphi_{i}/2^{2}$, and so on until we find a $\Delta \omega_i$ that works. Once we have the circle/bundles for $\omega_{i+1}=\omega_{i}+\Delta \omega_{i}$, we repeat the process. 

\begin{figure}
\begin{centering}
\includegraphics[width=0.5\columnwidth]{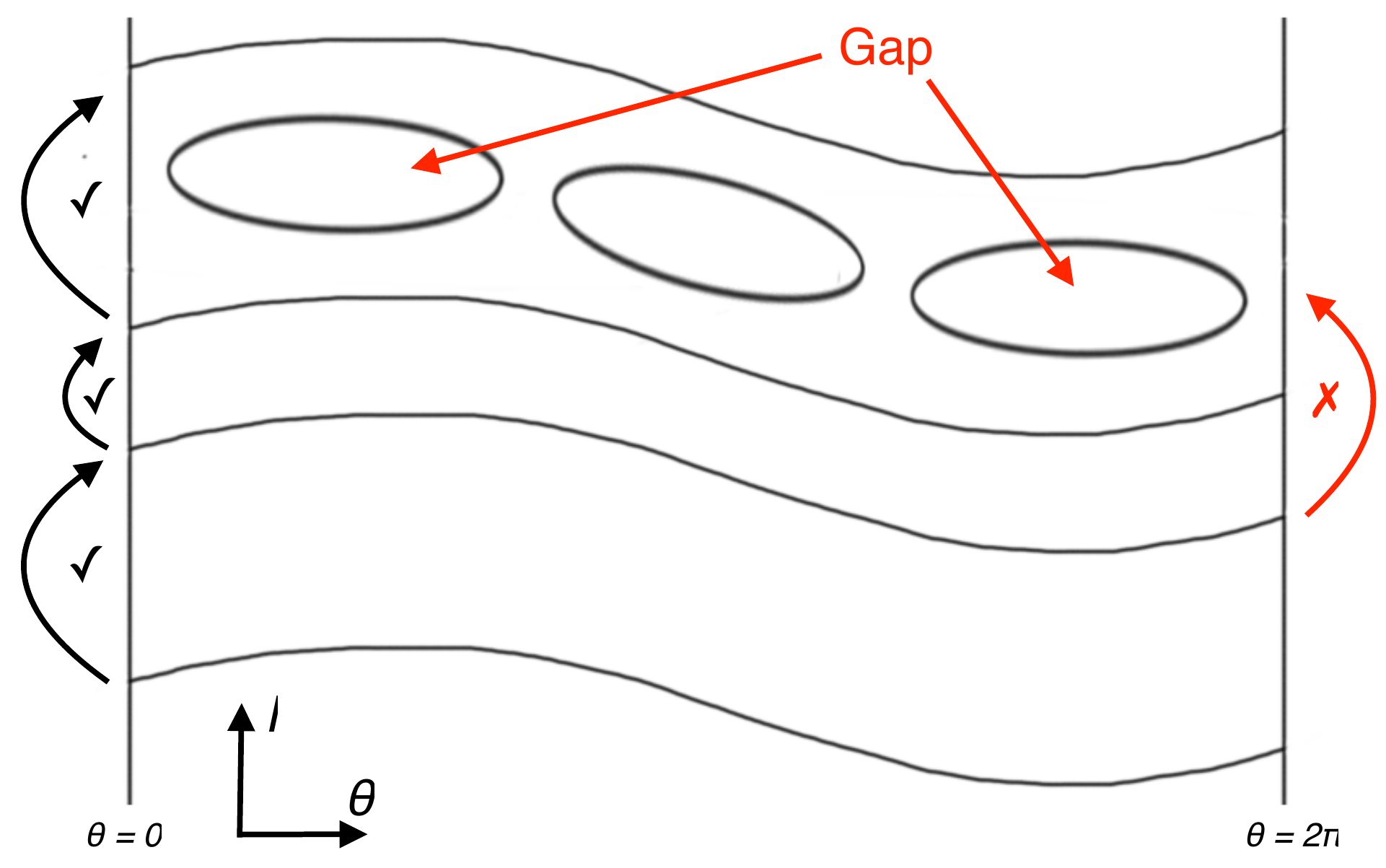}
\caption{\label{fig:gap} Schematic of crossing gaps during $\omega$ continuation (consider the $\theta=0$ and $\theta=2\pi$ lines to be glued together to form a cylinder)} 
\end{centering}
\end{figure}
In this procedure, it is possible for $\Delta \omega_{i+1}$ to be larger than $\Delta \omega_{i}$, which is what allows us to ``jump" over gaps in the tori. For example, suppose that $\Delta \omega_{i-1}$ through $\Delta \omega_{i-5}$ are all equal to $\phi = \frac{1+\sqrt 5}{2} \times 10^{-4}$, so $\varphi_{i}=\phi$. Then, it can happen that the quasi-Newton method diverges for $\omega=\omega_{i}+\phi$, but converges for $\omega = \omega_{i}+\phi/2= \omega_{i+1}$, so $\Delta \omega_{i} = \phi/2$. However, $\varphi_{i+1}$ will still equal $\phi$, so we try $\Delta \omega_{i+1} = \phi$. If the quasi-Newton method converges for $\omega = \omega_{i+1} + \phi = \omega_{i}+ 3\phi/2$, then we will have crossed the torus gap encountered earlier at $\omega=\omega_{i}+\phi$. This is schematically illustrated in Fig. \ref{fig:gap}; we draw the tori on a projection of the 2D cylindrical NHIM $\Xi_{\varepsilon}$ defined in Section \ref{stroboscopic} (we let $(\theta, I)$ be coordinates on $\Xi_{\varepsilon}$).

\subsection{Discretization and Implementation} \label{discretizationSection}

When implementing the previously-described methods on a computer, it is necessary to discretize all the functions used as well as the operations on them. We represent $K$, $P$, $\Lambda$, and other functions of $\theta$ as arrays of their values on a discrete grid of $N$ evenly spaced $\theta$ values $\theta_{i}= 2\pi i/N$, $i = 0, \dots, N-1$. Many operations on functions can be carried out element-wise on these arrays; such operations include basic scalar arithmetic, matrix multiplication, and matrix inversion. For instance, given arrays of values $P(\theta_{i})$ and $\xi(\theta_{i})$, we can calculate a new array of $N$ values $\Delta K(\theta_{i}) = P(\theta_{i}) \xi(\theta_{i})$ (note that the $\Delta K$ array will actually contain $4N$ floating point numbers, since each $\Delta K(\theta_{i}) \in \mathbb{R}^{4}$). 

Other operations are more efficiently carried out using Fourier coefficients. For instance, given an array of function values $a(\theta_{i})$ for some $a:\mathbb{T}\rightarrow \mathbb{R}$, we can use
\begin{gather} \label{translate} a(\theta_{i})  = \frac{1}{N} \sum_{k =0}^{N-1} \hat a(k) e^{jk\theta_{i}}  \rightarrow a(\theta_{i}+\omega)  = \frac{1}{N} \sum_{k =0}^{N-1} [\hat a(k)e^{jk\omega}] e^{jk\theta_{i}}  \\\label{differentiate} a(\theta_{i})  = \frac{1}{N} \sum_{k =0}^{N-1} \hat a(k) e^{jk\theta_{i}}  \rightarrow Da(\theta_{i})  = \frac{1}{N} \sum_{k =0}^{N-1} [j k \hat a(k)] e^{jk\theta_{i}}  \end{gather}
to translate or differentiate $a$. We use the fast Fourier transform (FFT) to get $N$ Fourier coefficients $\hat a(k)$, multiply each $\hat a(k)$ by $e^{jk\omega}$ (translation) or $jk$ (differentiation), and then take the inverse FFT to get an array of values $a(\theta_{i}+\omega) $ or $Da(\theta_{i})$. Solving cohomological equations like \eqref{cohomological} also requires working with Fourier coefficients as described in Section \ref{Kstep}.  

A few numerical problems were experienced due to the discretization of continuous functions on the computer. Let $\text{Trans}_{\omega}(a(\theta_{i}))$ represent the array of $a(\theta_{i}+\omega)$ values found by applying the algorithm of Eq. \eqref{translate} to the array of $a(\theta_{i})$ values. The first problem was that $\text{Trans}_{\omega}(F(\theta_{i})G(\theta_{i})) \neq \text{Trans}_{\omega}$($F(\theta_{i})) \times \text{Trans}_{\omega}(G(\theta_{i}))$; multiplying two arrays and then translating the result gives a different result than first translating the two arrays and then multiplying the results. Also, $F(\theta_{i}) \neq \text{Trans}_{-\omega}(\text{Trans}_{+\omega}(F(\theta_{i}))) $ when $N$ is even and real-to-complex FFT is used (for real data). These issues can prevent quasi-Newton method convergence. 

Both inequalities are most pronounced when the high-frequency discrete Fourier transform coefficients of $F$ or $G$ are large in magnitude. It is not expected that the tori or bundles we compute should have significant high-frequency oscillations as a function of $\theta$. Hence, a solution to these two problems was to run $K(\theta_{i})$ and $P(\theta_{i})$ through a lowpass filter during the first two or three quasi-Newton steps, as well as when the quasi-Newton method would start diverging; note that this is somewhat reminiscent of Arnold's use of truncated Fourier series with successively increasing cutoff frequencies in his proof of the KAM theory \citep{kamTutorial}. We also found that modifying $P$ between continuation steps to make $\Lambda$ constant, as described in Section \ref{constantLambda}, greatly mitigates these numerical issues as well. Finally, if the high-frequency Fourier coefficients keep becoming large after each quasi-Newton step despite filtering and constant $\Lambda$, we increase the number of Fourier modes used (equivalent to increasing $N$). 

One phenomenon we noticed during the implementation of our quasi-Newton method-based continuation was that generally, the torus parameterization $K(\theta)$ converges to within a given error tolerance before the bundle and Floquet matrices $P(\theta)$ and $\Lambda(\theta)$. We can use this to further improve the numerical stability of our quasi-Newton method, by using the converged $K(\theta)$ to directly compute the full $P$ and $\Lambda$ matrices. To do this, as mentioned earlier, the first column of $P$ is simply $DK(\theta)$. The third and fourth columns of $P$ (the stable and unstable bundles $\bold{v}_{s}$ and $\bold{v}_{u}$) can be found using the ``power method", as is described by \citet{haroLlave}. For this, first set $\bold{v}_{s,0}(\theta)$ and $\bold{v}_{u,0}(\theta)$ equal to the unit-length normalized third and fourth columns of $P(\theta)$ (the unconverged, approximate stable and unstable bundles). This is then followed by the iteration
\begin{equation} \bold{v}_{s,i+1}(\theta) = \frac{DF(K(\theta))^{-1} \bold{v}_{s,i}(\theta+\omega)}{\| DF(K(\theta))^{-1} \bold{v}_{s,i}(\theta+\omega) \|} \end{equation}
\begin{equation} \bold{v}_{u,i+1}(\theta) = \frac{DF(K(\theta-\omega)) \bold{v}_{u,i}(\theta-\omega)}{\| DF(K(\theta-\omega)) \bold{v}_{u,i}(\theta-\omega) \|} \end{equation}
which should converge after a few iterations (in practice, we also run each $\bold{v}_{s,i}(\theta)$, $\bold{v}_{u,i}(\theta)$ through a lowpass filter after its computation). Once the iterations have converged, we use the methods of Section \ref{constantLambda} to rescale $\bold{v}_{s}$ and $\bold{v}_{u}$ to ensure constant $\lambda_s$ and $\lambda_u$. Finally, we can use the exact same method presented in Section \ref{continuationSection} to compute the second column of $P$ from the known $DK$, $\bold{v}_{s}$, and $\bold{v}_{u}$ (see Eq. \eqref{abcd}-\eqref{sympconj}); the method of Section \ref{constantLambda} is then applied to make $T$ constant. This gives us the final $P$ and $\Lambda$ matrices which in our experience not only usually satisfy Eq. \eqref{bundleEquations} to within tolerance (sometimes one last quasi-Newton correction step is required), but also have smaller high-order Fourier coefficients than the earlier approximate $P$ and $\Lambda$; this further improves our method's numerical stability. 

\subsection{Numerical Results in the PERTBP} \label{numResults}

We implemented and successfully applied the methods described in the previous sections to the computation of invariant circles and their bundles for the Jupiter-Europa PERTBP stroboscopic map. We used a tolerance of $10^{-7}$ in Eq. \eqref{invariance}-\eqref{bundleEquations}. The circles and bundles were found by first continuing Jupiter-Europa PCRTBP unstable resonant periodic orbits by eccentricity $\varepsilon$ to $\varepsilon = 0.0094$ (the real value) for fixed $\omega$, and then continuing the circles and bundles by $\omega$ while fixing $\varepsilon = 0.0094$. 

Both 3:4 as well as 5:6 resonant tori were computed. Fig. \ref{fig:56plot1} shows the continuation by eccentricity of a 5:6 resonant periodic orbit from the PCRTBP to an invariant circle of the PERTBP stroboscopic map; for this, we used $N=2048$ discretization $\theta_{i}$ values. The plot on the right zooms into the region near Europa; the leftmost curve there corresponds to $\varepsilon = 0.0094$, which is to be expected as Europa's periapsis moves leftwards as $\varepsilon$ increases. Fig. \ref{fig:34plot1} shows the continuation of a 3:4 resonant torus in the physical $\varepsilon = 0.0094$ Jupiter-Europa PERTBP by $\omega$, which yields a family of resonant tori in the system for $\omega \in [1.536217, 1.567314]$; $N$ ranged from 1024 to 32768, with larger $N$ required for tori passing close to the singularity in the equations of motion at Europa. Fig. \ref{fig:56famplot1} shows a family of Jupiter-Europa PERTBP 5:6 resonant tori, also generated using continuation by $\omega$. This family, like the PCRTBP 5:6 resonant orbit family, does not have monotonically increasing or decreasing rotation numbers; $\omega$ starts at 1.035166 for the leftmost torus in the zoomed-in plot, decreases to 1.027137, and then increases to 1.040911 for the rightmost torus. Thus, we first continued two different PCRTBP 5:6 resonant orbits by $\varepsilon$ to get two PERTBP tori, one in each of the two sections of tori with monotone $\omega$. These two tori were then continued by $\omega$ to sweep out the tori in their corresponding sections. For this case, $N$ ranged from 1024 to 4096.

\begin{figure}
\includegraphics[width=0.499\columnwidth]{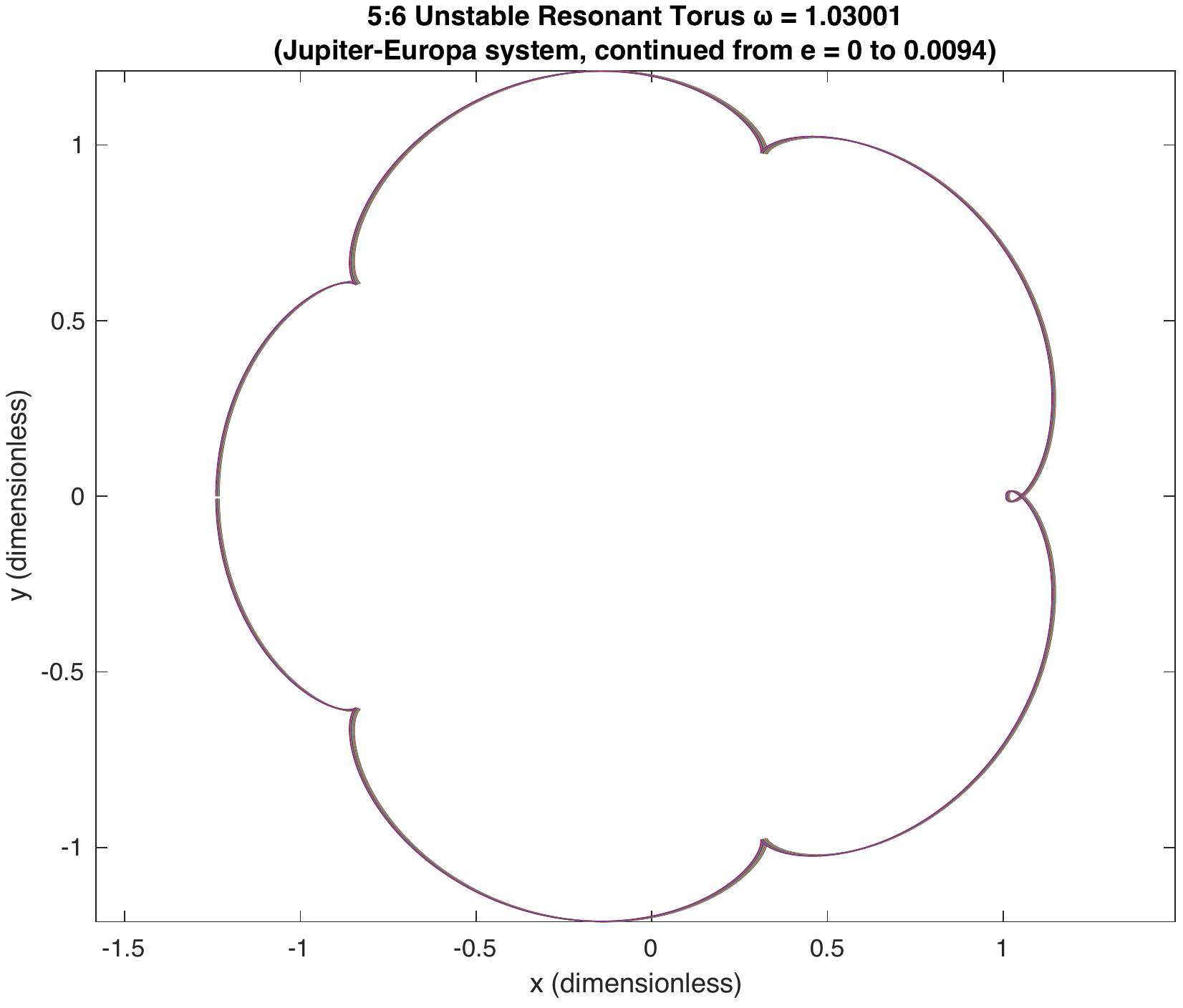}
\includegraphics[width=0.499\columnwidth]{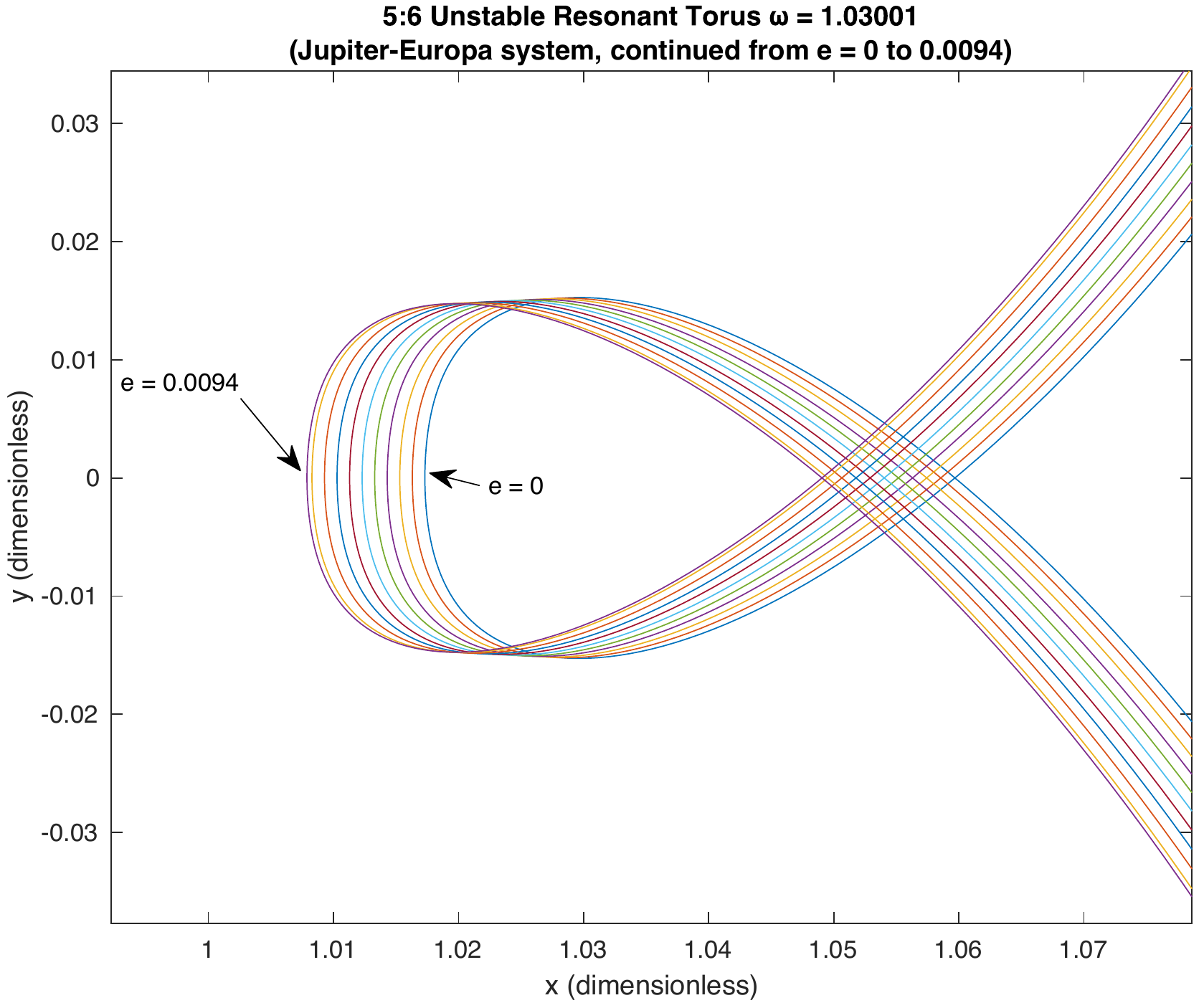}
\caption{\label{fig:56plot1}  Continuation of 5:6 Jupiter-Europa PERTBP resonant torus from $\varepsilon = 0$ to $0.0094$} 
\end{figure}

\begin{figure}
\includegraphics[width=0.5\columnwidth]{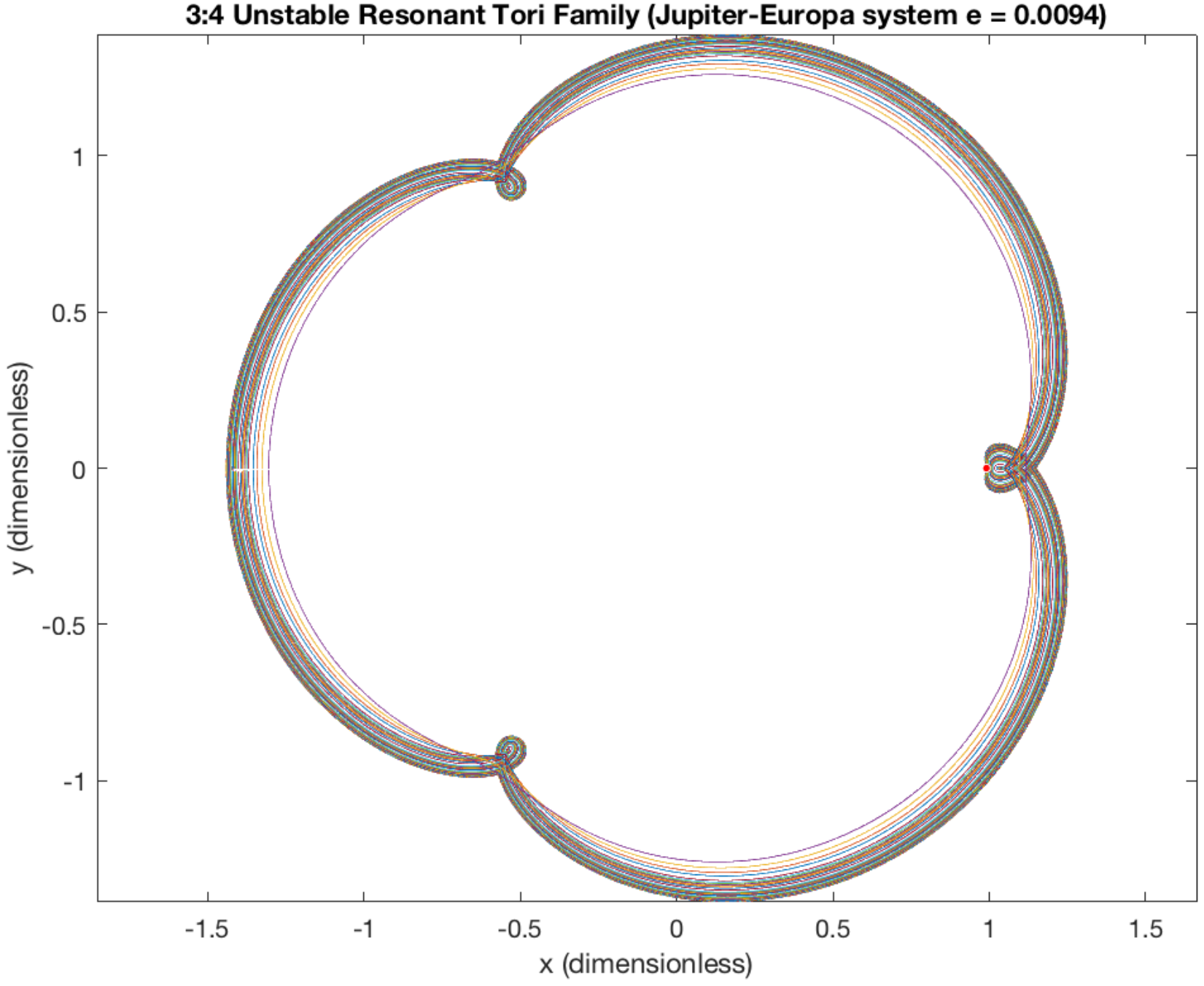}
\includegraphics[width=0.5\columnwidth]{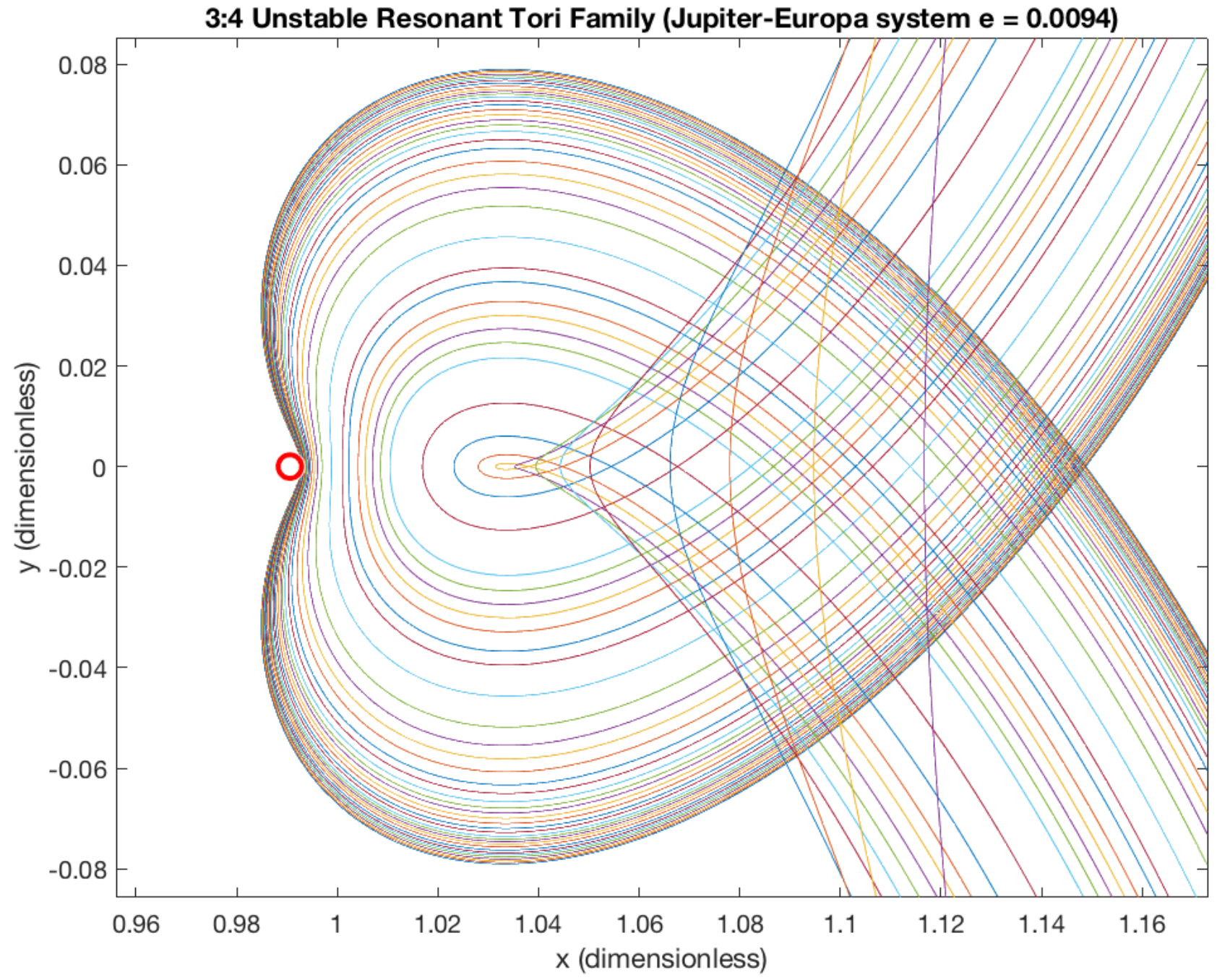}
\caption{\label{fig:34plot1}  Continuation of $\varepsilon = 0.0094$ Jupiter-Europa PERTBP 3:4 resonant tori by $\omega$ (Europa surface shown as red circle)}
\end{figure}

\begin{figure}
\includegraphics[width=0.5\columnwidth]{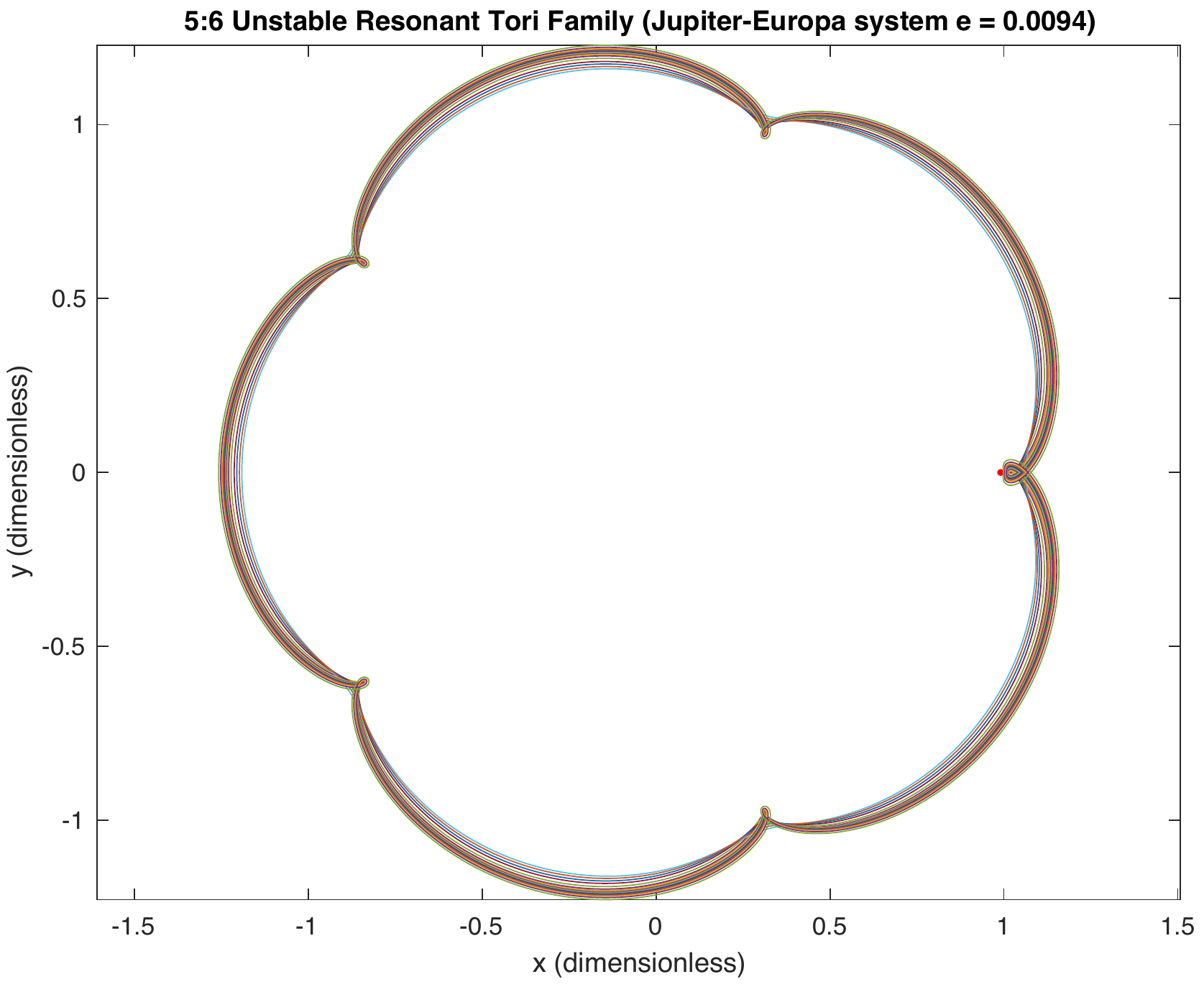}
\includegraphics[width=0.5\columnwidth]{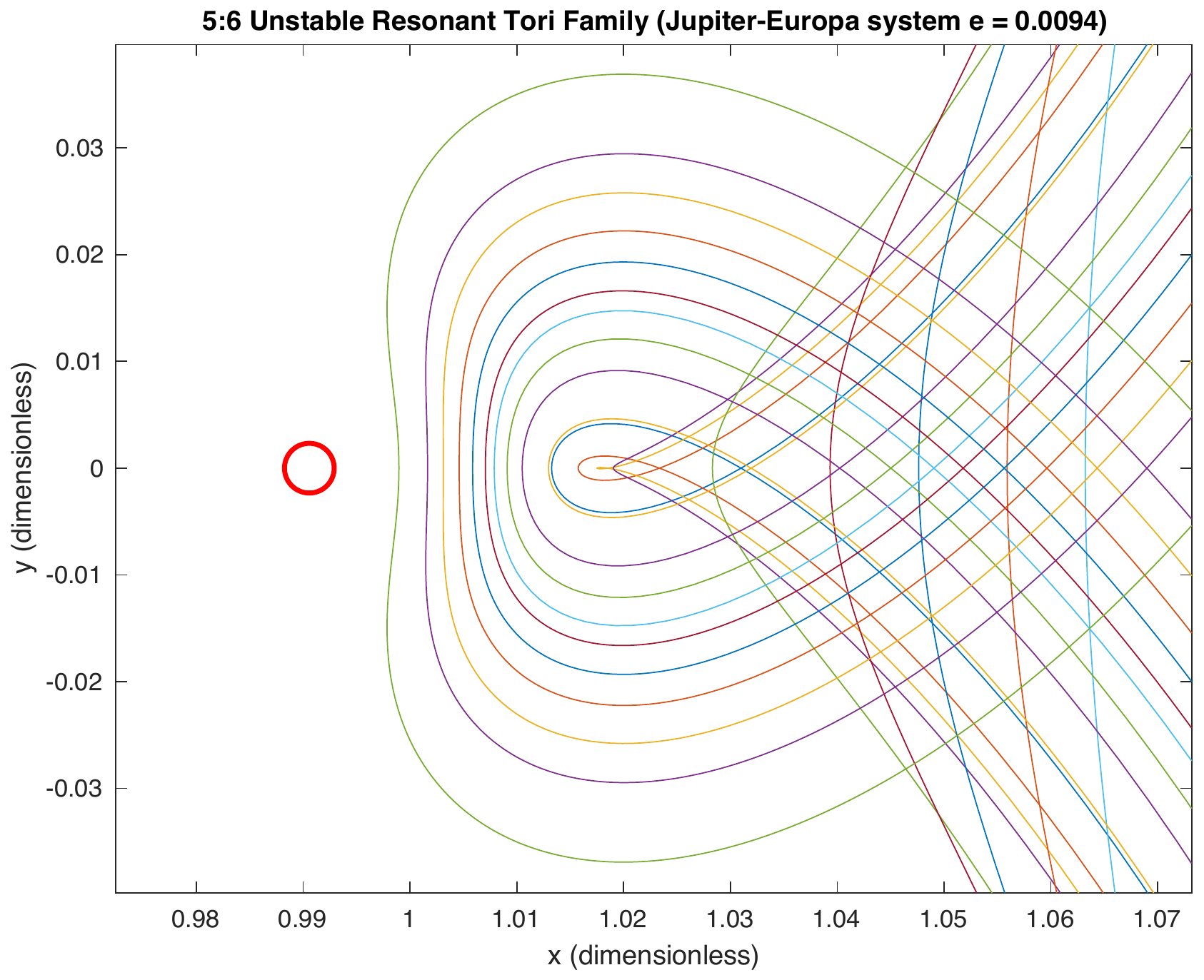}
\caption{\label{fig:56famplot1}  Continuation of $\varepsilon = 0.0094$ Jupiter-Europa PERTBP 5:6 resonant tori by $\omega$ (Europa surface shown as red circle)}
\end{figure}

After computing tori in the physical Jupiter-Europa PERTBP with $\varepsilon = 0.0094$, we also tested our quasi-Newton method to see if it would work for larger $\varepsilon$. Fig. \ref{fig:34plotall} shows selected tori from the continuation of a 3:4 resonant periodic orbit from the PCRTBP to an invariant circle of the PERTBP with Jupiter-Europa mass ratio $\mu$, but $\varepsilon = 0.206$. This eccentricity is larger than that of the Sun-Mercury system, which has one of the most eccentric two-body orbits of any pair of large solar system bodies. We used $N=1024$ and a continuation step size of $\Delta \varepsilon = 0.0005$ (the quasi-Newton method failed to converge for larger step sizes); every 20th torus is shown in the figure. As can be seen, our method was robust even for large values of the perturbation $\varepsilon$. On a 2017-era quad-core i7 laptop CPU, our Julia program took only about 230 seconds for the entire continuation to $\varepsilon = 0.206$, and less than 10 s for continuation to the physical value $\varepsilon = 0.0094$. 

\begin{figure}
\includegraphics[width=0.5\columnwidth]{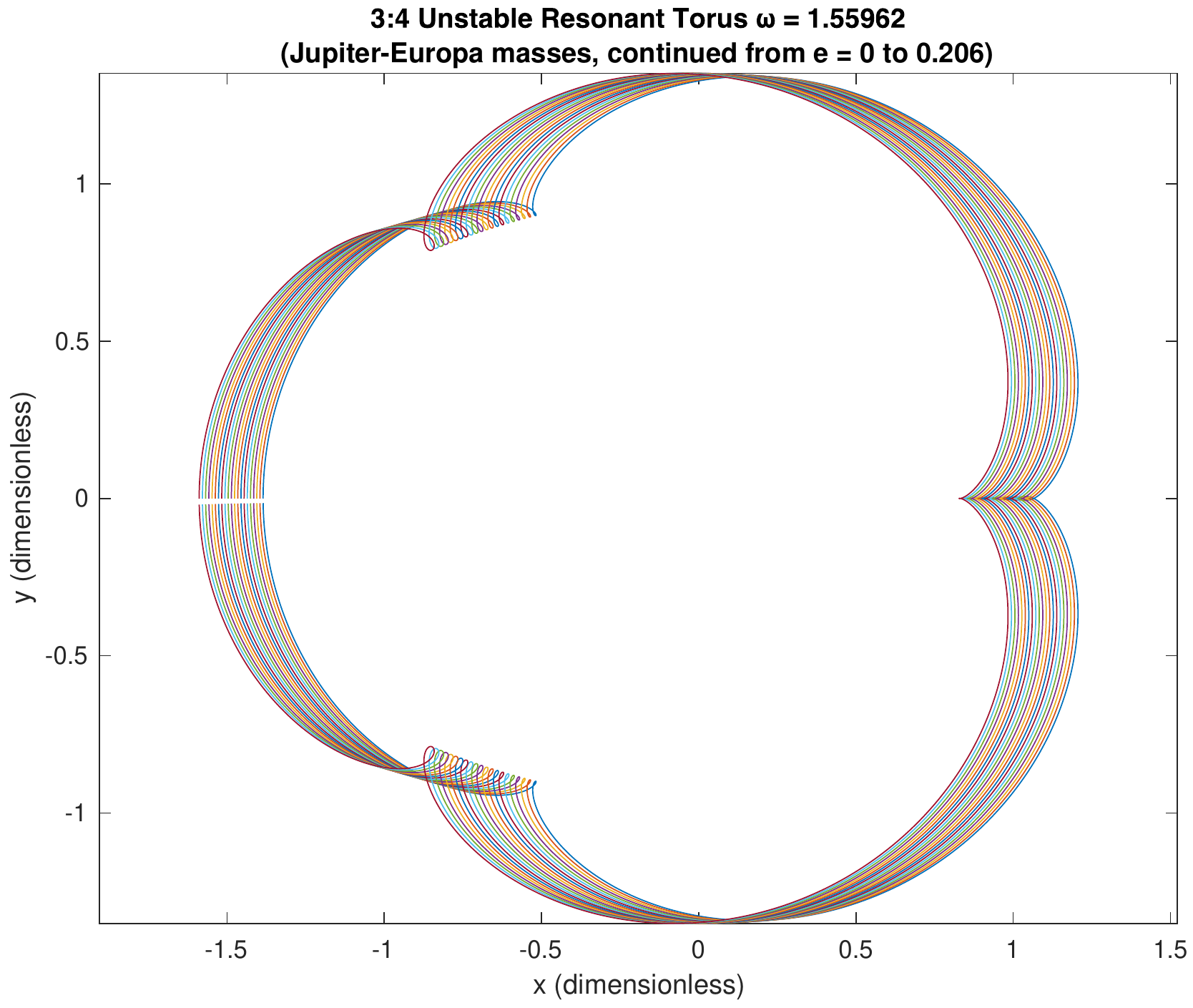}
\includegraphics[width=0.5\columnwidth]{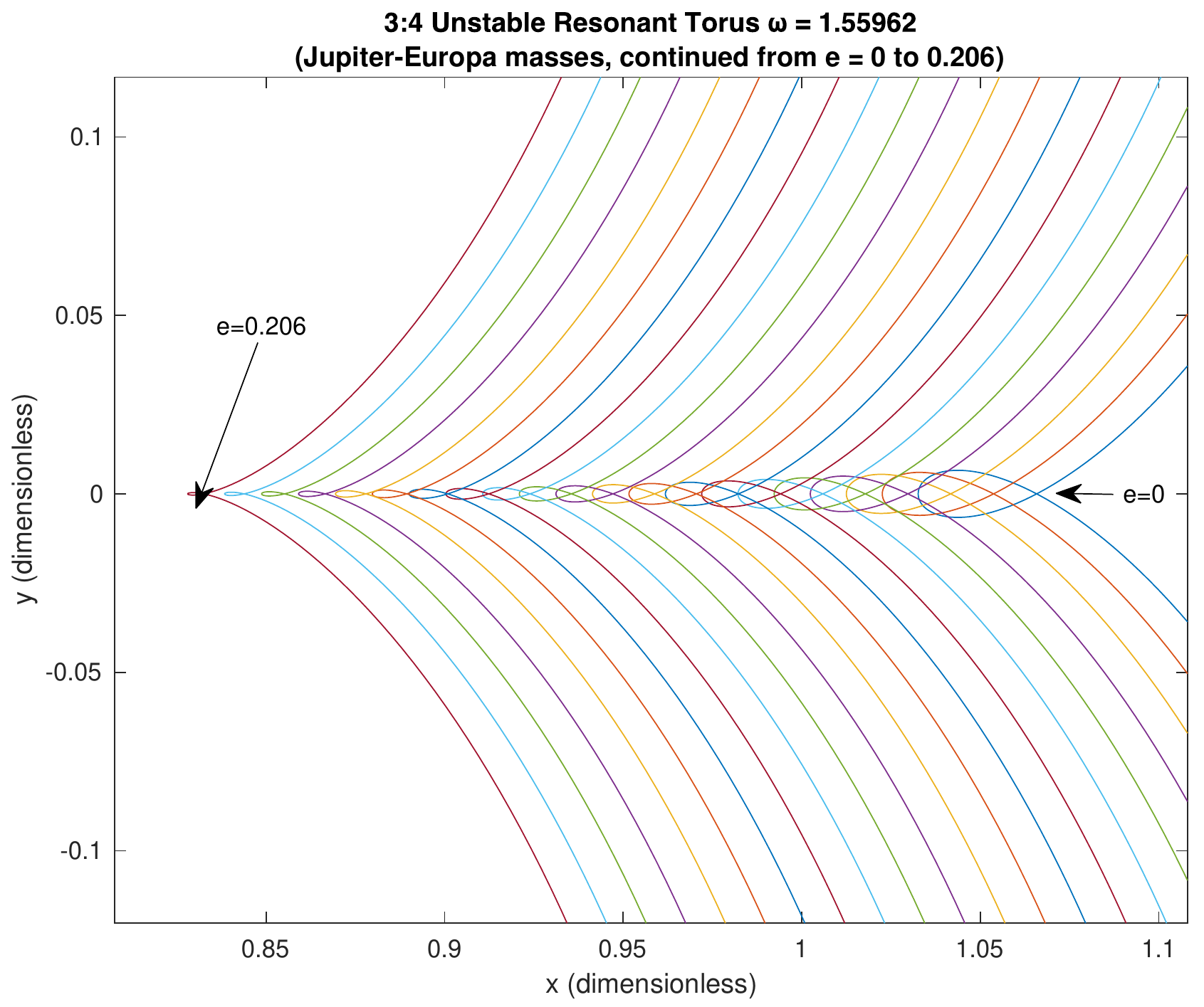}
\caption{\label{fig:34plotall} Selected tori from 3:4 Jupiter-Europa PERTBP continuation from $\varepsilon = 0$ to $0.206$} 
\end{figure}
\section{Parameterization Method for Stable and Unstable Manifolds} \label{parambigsection}
With the invariant circles and their stable and unstable bundles computed, we next turn our attention to accurate computation of torus stable and unstable manifolds. Many studies using manifolds, such as \citet{olikaraThesis}, use linear approximations of invariant manifolds found by adding small vectors in the stable or unstable directions to the points of the torus, and then integrating backwards or forwards. However, we compute high order Fourier-Taylor polynomials which approximate the manifolds very accurately in some domain of validity. The algorithm used here bears many similarities with the method used in previous work \citep{kumar2021journal} for computation of 1D manifolds of period-maps for periodic orbits in the PCRTBP.  A different version of this algorithm was also used by \citet{zhang} in a lower-dimensional setting. 

Since our $F$-invariant circles are 1D and have one stable and one unstable direction at each point, the circles' stable and unstable manifolds will be 2D and diffeomorphic to either an infinite cylinder or a M\"obius strip. A cylinder can be continuously parameterized using an angle $\theta$ and a real number $s$; this actually is also possible for a M\"obius strip, as long as the parameterization is non-injective and wraps around the strip twice as $\theta$ goes from 0 to $2\pi$ (the ``double covering" trick we used in Section \ref{continuationSection}). In the framework of Section \ref{paramsectiongeneral}, we have $\mathcal M = \mathbb{T} \times \mathbb{R}$ and $f(\theta, s) = (\theta+\omega, \lambda s)$, where $\lambda$ is the stable $\lambda_{s}$ or unstable $\lambda_{u}$ entry of $\Lambda$, depending on which manifold we are trying to compute. Without loss of generality, we assume that $\lambda_{s}$ and $\lambda_{u}$ are constant (see Section \ref{constantLambda}). With this, the equation to solve for the parameterization $W:\mathbb{T \times \mathbb{R}} \rightarrow \mathbb{R}^{4}$ of the stable or unstable manifold is 
\begin{equation}  \label{invariancequationfinal}   F(W(\theta, s)) - W(\theta + \omega, \lambda s) =0, \quad (\theta, s) \in \mathbb{T} \times \mathbb{R} \end{equation}

\subsection{Order-by-Order Method to find $W$} \label{paramsection}

We express $W$ as a Fourier-Taylor series of form
\begin{equation}  \label{series} W(\theta, s) = K(\theta) + \sum_{k \geq 1} W_{k}(\theta)s^{k}  \end{equation}
where $s=0$ corresponds to the invariant circle $K(\theta)$ whose manifold we are trying to compute. The $s^{0}$ term of $W$ is $K(\theta)$, and the linear term $W_{1}(\theta)$ is the stable $\bold{v}_{s}(\theta)$ or unstable $\bold{v}_{u}(\theta)$ bundle known from the third or fourth column of $P$. Hence we need to solve for the higher-order ``coefficients" $W_{k}(\theta):\mathbb{T} \rightarrow \mathbb{R}^{4}$, $k \geq 2$. 

Denote $W_{<k} (\theta, s) = K(\theta) + \sum_{j = 1}^{k-1} W_{j}(\theta)s^{j}  $. Assume we have solved for all $W_{j}(\theta)$ for $j <k$, so that $F(W_{<k}(\theta, s)) - W_{<k}( \theta+\omega, \lambda s)$ has only $s^{k}$ and higher order terms. Then, starting with $k=2$, the recursive method to solve for $W_{k}(\theta)$ is:
    \begin{enumerate}
        	\item Find $E_{k}(\theta)= [F(W_{<k}(\theta, s)) - W_{<k}(\theta+\omega,  \lambda s)]_{k}$, where $[\cdot]_{k}$ denotes the $s^{k}$ coefficient of the term inside brackets. We show how to do this in Section \ref{jettransport}. 
	\item Find $W_{k} (\theta)$ such that $W_{<k} (\theta,s)+W_{k} (\theta)s^{k}$ cancels $E_k (\theta)s^{k}$ in Eq. \eqref{invariancequationfinal}, thus satisfying Eq. \eqref{invariancequationfinal} up to order $s^{k}$. The equation to solve for $W_{k}(\theta)$ is 
	\begin{equation}  \label{correctionwk} DF(K(\theta)) W_{k}(\theta) - \lambda^{k} W_{k}(\theta+\omega) = -E_{k}(\theta) \end{equation}
	To solve this, let $W_{k,0}=0$ and iterate the following sequence to convergence:
	\begin{equation}  \label{Witer} W_{k,i+1}(\theta) =   \left\{ \begin{matrix} 
      \lambda^{k} DF(K(\theta))^{-1} W_{k,i}(\theta+\omega) - DF(K(\theta))^{-1}E_{k}(\theta)  &\text{if $|\lambda| < 1$}  \\
        \lambda^{-k} DF(K(\theta-\omega)) W_{k,i}(\theta-\omega)  + \lambda^{-k}E_{k}(\theta-\omega) \,\,\,\, &\text{if $|\lambda| > 1$} 
   \end{matrix} \right. \end{equation} (Fourier methods are an alternate method of solving Eq. \eqref{correctionwk}; see Remark \ref{remarkFourier2})
	\item Set $W_{<k+1} (\theta,s) = W_{<k} (\theta, s) + W_k(\theta) s^{k}$ and return to step 1.
    \end{enumerate}
    The recursion is stopped when we are satisfied with the degree $k$ of $W$. We now prove that the equations and method described in Step 2 to find $W_{k}$ are valid. 
\begin{claim} If $W_{k}$ solves Eq. \eqref{correctionwk}, then for $j \leq k$ (using the $[\cdot]_{k}$ notation defined earlier),
\begin{equation}  \label{jcoeff} \left[  F(W_{<k}(\theta, s)+W_{k}(\theta)s^{k}) - \left(W_{<k}(\theta+\omega, \lambda s)+W_{k}(\theta+\omega)( \lambda s)^{k}\right) \right]_{j}= 0 \end{equation} 
\end{claim} 
\begin{proof} Recall that $ F(W_{<k}(\theta, s)) - W_{<k}(\theta+\omega,  \lambda s)=E_{k}(\theta)s^{k} + \mathcal O(s^{k+1}) $. Expanding Eq. \eqref{jcoeff} in Taylor series and keeping only $s^{k}$ and lower order terms gives
\begin{align}  \begin{split}
\Big[F(W_{<k}&(\theta, s)) + \Big. DF(W_{<k}(\theta, s)) W_{k}(\theta)s^{k} - \\
& \quad \quad \quad \left. \left(W_{<k}(\theta+\omega, \lambda s)+W_{k}(\theta+\omega)( \lambda s)^{k}\right) \right]_{j} \\
=&[E_{k}(\theta)s^{k} + DF(W_{<k}(\theta,s)) W_{k}(\theta)s^{k} - \lambda^{k} W_{k}(\theta+\omega)s^{k} ]_{j} \\
=&\begin{cases} 
      0 &\text{if $j<k$}, \\
      E_{k}(\theta) + DF(K(\theta)) W_{k}(\theta) - \lambda^{k} W_{k}(\theta+\omega) = 0 &\text{if $j=k$}
   \end{cases} \\
\end{split} \end{align}
where the $j=k$ case of the last line follows from the preceding line by dividing $s^{k}$ out from the quantity inside $[.]_{j}$, and then taking $s \rightarrow 0$. \qed
\end{proof}

\begin{claim} The sequence $\{ W_{k,i}\}_{i \in \mathbb{N}}$ defined by $W_{k,0}=0$ and Eq. \eqref{Witer} converges to $W_{k}$.
\end{claim} 
\begin{equation} W_{k,i+1}(\theta) =   \left\{ \begin{matrix} 
      \lambda^{k} DF(K(\theta))^{-1} W_{k,i}(\theta+\omega) - DF(K(\theta))^{-1}E_{k}(\theta)  &\text{if $|\lambda| < 1$}  \\
        \lambda^{-k} DF(K(\theta-\omega)) W_{k,i}(\theta-\omega)  + \lambda^{-k}E_{k}(\theta-\omega) \,\,\,\, &\text{if $|\lambda| > 1$} 
   \end{matrix} \right. \end{equation}
   \begin{proof} Let $P,\Lambda$ be the bundle and Floquet matrices for $K(\theta)$. We assume that $\Lambda$ is constant (as the procedure from Section \ref{constantLambda} gives). Then, it is easy to show that  
   \begin{equation}  \label{Wsum} W_{k,i}(\theta) =   \left\{ \begin{matrix} 
      -P(\theta) \sum_{j=0}^{i-1} \lambda^{kj}\Lambda^{-j-1} [ P^{-1}(\theta+(j+1)\omega)E_{k}(\theta+j\omega) ] &\text{if $|\lambda| < 1$}  \\
        P(\theta) \sum_{j=0}^{i-1} \lambda^{-k(j+1)}\Lambda^j [ P^{-1}(\theta-j\omega)E_{k}(\theta-(j+1)\omega) ]&\text{if $|\lambda| > 1$} 
   \end{matrix} \right. \end{equation}
solves Eq. \eqref{Witer} with $W_{k,0}=0$; simply substitute Eq. \eqref{Wsum} for $W_{k,i}$ in Eq. \eqref{Witer} and use $DF(K(\theta-\omega)) P(\theta-\omega) = P(\theta) \Lambda $ and $DF(K(\theta))^{-1}=P(\theta) \Lambda^{-1} P^{-1}(\theta+\omega)$ to simplify the RHS of the resulting equation.

Now, we will show that $W_{k}(\theta) = \lim_{i \rightarrow \infty} W_{k,i}(\theta)$. First of all, note that
\begin{equation} \lambda_{s}^{j} \Lambda^{-j} =  \begin{bmatrix}
\lambda_{s}^{j} &  -j\lambda_{s}^{j} T  & 0 & 0 \\ 0 &  \lambda_{s}^{j}   & 0 & 0 \\ 0 & 0  & 1 & 0 \\ 0 &  0 & 0 & \lambda_{s}^{j} \lambda_{u}^{-j} \end{bmatrix}  \quad \quad \quad \lambda_{u}^{-j} \Lambda^{j} =  \begin{bmatrix}
\lambda_{u}^{-j} &  j\lambda_{u}^{-j} T & 0 & 0 \\ 0 &  \lambda_{u}^{-j}   & 0 & 0 \\ 0 & 0  & \lambda_s^{j}\lambda_{u}^{-j} & 0 \\ 0 &  0 & 0 & 1 \end{bmatrix}  \end{equation}
for all $j \in \mathbb{N}$, where $\Lambda$ is of the form given in Eq. \eqref{Lambdaform} and has constant $T$, $\lambda_{s}$, and $\lambda_{u}$ as assumed earlier. Since $|\lambda_{s}|<1$ and $|\lambda_{u}|>1$, $ \lambda_{s}^{j} \Lambda^{-j} $ and $\lambda_{u}^{-j} \Lambda^{j} $ are hence bounded for all $j \in \mathbb{N}$. Now, define $\Gamma_{s}(\theta) =\Lambda^{-1} P^{-1}(\theta+\omega)E_{k}(\theta) $ and $\Gamma_{u}(\theta) = \lambda^{-k} P^{-1}(\theta)E_{k}(\theta-\omega) $; also, recall that $|\lambda|<1$ means $\lambda = \lambda_{s}$ and $|\lambda|>1$ means $\lambda = \lambda_{u}$. We can use all this to rewrite Eq. \eqref{Wsum} as
\begin{equation}  \label{Wsum2} W_{k,i}(\theta) =   \left\{ \begin{matrix} 
      -P(\theta) \sum_{j=0}^{i-1} \lambda_{s}^{(k-1)j} [\lambda_{s}^{j}\Lambda^{-j} \Gamma_{s}(\theta+j\omega) ] \quad &\text{if $|\lambda| < 1$}  \\
        P(\theta) \sum_{j=0}^{i-1} \lambda_{u}^{-(k-1)j} [\lambda_{u}^{-j} \Lambda^j \Gamma_u (\theta-j\omega) ] \quad &\text{if $|\lambda| > 1$} 
   \end{matrix} \right. \end{equation}
In both $|\lambda| < 1$ and $|\lambda| > 1$ cases of Eq. \eqref{Wsum2}, the quantities in square brackets are bounded for all $\theta \in \mathbb{T}$ and $j \in \mathbb{N}$. As $k \geq 2$, $\lambda_{s}^{k-1}<1$ and $\lambda_{u}^{-(k-1)}<1$; hence, if $i \rightarrow \infty$, the sum in Eq. \eqref{Wsum2} is absolutely uniformly convergent. Hence $L(\theta) = \lim_{i \rightarrow \infty} W_{k,i}(\theta)$ exists. Letting $i \rightarrow \infty$ on both sides of Eq. \eqref{Witer} gives
\begin{equation} L(\theta) = \left\{ \begin{matrix} 
      \lambda^{k} DF(K(\theta))^{-1} L(\theta+\omega) - DF(K(\theta))^{-1}E_{k}(\theta)  &\text{if $|\lambda| < 1$}  \\
        \lambda^{-k} DF(K(\theta-\omega)) L(\theta-\omega)  + \lambda^{-k}E_{k}(\theta-\omega) \,\,\,\, &\text{if $|\lambda| > 1$} 
   \end{matrix} \right. \end{equation}
   which for both $|\lambda| < 1$ and $|\lambda| > 1$ is equivalent to Eq. \eqref{correctionwk} with $W_{k}=L$. \qed
\end{proof}
\begin{remark} \label{remarkFourier2}
Given $P$ and $ \Lambda$ satisfying Eq. \eqref{bundleEquations}, substituting $W_{k} = PV_{k}$ into Eq. \eqref{correctionwk} and rearranging gives $\Lambda V_{k}(\theta) - \lambda^{k} V_{k}(\theta+\omega) = -P(\theta+\omega)^{-1}E_{k}(\theta)$, which can be solved for $V_{k}$ component by component using the Fourier methods from Section \ref{cohomsection}. We used the iteration method of Eq. \eqref{Witer} instead, to avoid any possible multiplication-translation numerical discretization issues (see Section \ref{discretizationSection}). 

\end{remark}

\subsection{Computing $E_{k}(\theta)$: Automatic Differentiation and Jet Transport} \label{jettransport}

In step 1 of the order-by-order method to find $W$, we compute the $s^{k}$ coefficient
\begin{equation} \label{Ekdef} E_{k}(\theta)=  [F(W_{<k}(\theta, s)) - W_{<k}(\theta+\omega,  \lambda s)]_{k} \end{equation}
In Eq. \eqref{Ekdef}, the $s^{k}$ term of $W_{<k}(\theta+\omega, \lambda s)$ is 0, since $W_{<k}(\theta, s)$ is a Fourier-Taylor series up to order $s^{k-1}$ and $\lambda$ is constant. However, computing the Fourier-Taylor expansion of $F(W_{<k}(\theta, s)) $ is more complicated, as $F$ is a nonlinear stroboscopic map defined by integrating points for a fixed time $2\pi/\Omega_{p}$ by the equations of motion \eqref{H_EOM} and \eqref{perturbed_H}. We will need the tools of automatic differentiation \citep{haroetal} and jet transport \citep{perezpalau2015} for this. Note that some researchers \citep{dast,Berz1998} use the term differential algebra to refer to what we call automatic differentiation. 

Automatic differentiation is an efficient and recursive technique for evaluating operations on Taylor series. For instance, let $f(s)$ and $g(s)$, $s \in \mathbb{R}$, be two series; we can use their known coefficients to compute $d(s) = f(s) / g(s)$ as a Taylor series as well. Let subscript $j$ denote the $s^{j}$ coefficient of a series; since $ f(s) = d(s) g(s)$, we find that $f_{i} = \sum_{j=0}^{i} d_{j} g_{i-j} =  \left( \sum_{j=0}^{i-1} d_{j} g_{i-j}(s) \right)+ d_{i} g_{0} $, which implies that
\begin{equation}  
  \label{autodiff}  d_{i} = \frac{1}{g_{0}} \left( f_{i} - \sum_{j=0}^{i-1} d_{j} g_{i-j} \right)\end{equation}  
Starting with $d_{0} = f_{0}/ g_{0}$, Eq. \eqref{autodiff} allows us to recursively compute $d_{i}$, $i \geq 1$. Similar formulas also exist for recursively evaluating many other functions and operations on Taylor series, including $f(s)^{\alpha}$, $\alpha \in \mathbb{R}$; see \citet{haroetal} for more examples. Most importantly, in all automatic differentiation formulas, the output series $s^{i}$ coefficient is a function of only the $s^{i}$ and lower order coefficients of the input series. Hence, we can use truncated Taylor series with these algorithms when implementing them in computer programs.

Recall from Section \ref{discretizationSection} that on the computer, we represent all functions of $\theta$, including the $W_{j}(\theta)$, as arrays of function values on an evenly spaced grid of $\theta$ values $\theta_{i}= 2\pi i/N$, $i = 0, \dots, N-1$. Note that for fixed $\theta_{i}$,  $W_{<k}(\theta_{i}, s)$ is a Taylor series (not Fourier-Taylor) with coefficients $W_{j}(\theta_{i}) \in \mathbb{R}^{4}$. Using automatic differentiation, we can substitute Taylor series such as $W_{<k}(\theta_{i}, s)$ for $(x, y, p_x, p_y)$ in the equations of motion \eqref{H_EOM}, which gives us series in $s$ for $(\dot x, \dot y,\dot p_x, \dot p_y)$. In terms of computer programming, this means that after overloading the required operators (usually arithmetic and power) to accept Taylor series arguments, we can use numerical integration routines with the series as well. 

To be more clear, consider a Taylor series-valued function of time $V(s,t) = \sum_{j=0}^{\infty}V_{j}(t)s^{j}:\mathbb{R}^{2} \rightarrow \mathbb{R}^{4}$, where $V_{j}(t)$ are its time-varying Taylor coefficients. Write $V_{x}(s,t)$, $V_{y}(s,t)$, ${V}_{p_{x}}(s,t)$, and $V_{p_y}(s,t)$ for the $x$, $y$, $p_x$, and $p_y$ components of $V(s,t)$; similarly write $V_{j,x}(t)$, $V_{j,y}(t)$, ${V}_{j,p_{x}}(t)$, and $V_{j,p_y}(t)$ for the components of $V_{j}(t)$. Substituting $V$ in Eq. \eqref{H_EOM} yields a system of differential equations 
\begin{equation} \label{vxdot}  \frac{d}{dt}{V_{x}(s,t)} = \sum_{j=0}^{\infty}\dot V_{j,x}(t)s^{j} = \frac{\partial H_{\varepsilon}}{\partial p_{x}}\Big(V_{x}(s,t),V_{y}(s,t), V_{p_x}(s,t),V_{p_y}(s,t),\theta_{p}\Big)  \end{equation}
\begin{equation}  \label{vydot} \frac{d}{dt}{V_{y}(s,t)} = \sum_{j=0}^{\infty}\dot V_{j,y}(t)s^{j} = \frac{\partial H_{\varepsilon}}{\partial p_{y}}\Big(V_{x}(s,t),V_{y}(s,t), V_{p_x}(s,t),V_{p_y}(s,t),\theta_{p}\Big)  \end{equation}
\begin{equation}  \label{vpxdot} \frac{d}{dt}{V_{p_x}(s,t)} = \sum_{j=0}^{\infty}\dot V_{j,p_x}(t)s^{j} = -\frac{\partial H_{\varepsilon}}{\partial x}\Big(V_{x}(s,t),V_{y}(s,t), V_{p_x}(s,t),V_{p_y}(s,t),\theta_{p}\Big)  \end{equation}
\begin{equation} \label{vpydot} \frac{d}{dt}{V_{p_y}(s,t)} = \sum_{j=0}^{\infty}\dot V_{j,p_y}(t)s^{j} = -\frac{\partial H_{\varepsilon}}{\partial y}\Big(V_{x}(s,t),V_{y}(s,t), V_{p_x}(s,t),V_{p_y}(s,t),\theta_{p}\Big)  \end{equation}
\begin{equation} \label{thetapdot} \dot \theta_p = \Omega_p \end{equation}
$H_{\varepsilon}$ and its partials are algebraic functions that are suitable for use with automatic differentiation techniques; see, for instance, the PERTBP Hamiltonian Eq. \eqref{pertbpH}. Hence, if the $V_{j,x}(t)$, $V_{j,y}(t)$, ${V}_{j,p_{x}}(t)$, $V_{j,p_y}(t)$, and $\theta_{p}$ are known for $j \in \mathbb{N}$ and some $t \in \mathbb{R}$, automatic differentiation allows us to simplify the RHS of each of Eq. \eqref{vxdot}-\eqref{vpydot} to a series in $s$. Then, for each of Eq. \eqref{vxdot}-\eqref{vpydot} and $j \in \mathbb{N}$, the $s^{j}$ coefficient $\dot V_{j,x}(t)$, $\dot V_{j,y}(t)$, ${\dot V}_{j,p_{x}}(t)$, or $\dot V_{j,p_y}(t)$ from the LHS must be equal to the $s^{j}$ coefficient of the RHS. In other words, $\dot V_{j,x}(t)$, $\dot V_{j,y}(t)$, ${\dot V}_{j,p_{x}}(t)$, and $\dot V_{j,p_y}(t)$, $j \in \mathbb{N}$, are functions of $\theta_{p}$, $ V_{j,x}(t)$, $ V_{j,y}(t)$, ${ V}_{j,p_{x}}(t)$, and $ V_{j,p_y}(t)$, $j \in \mathbb{N}$. This is effectively a system of differential equations for the time-varying Taylor coefficients of $V(s,t)$. Solving Eq. \eqref{vxdot}-\eqref{thetapdot}  with initial condition $V(s,0) = W_{<k}(\theta_{i}, s)$ and initial $\theta_{p}$ equal to the value fixed in Section \ref{stroboscopic}, we can compute $F(W_{<k}(\theta_{i}, s))=V(s,2\pi/\Omega_{p})$.

In summary, we consider the Taylor coefficients of $W_{<k}(\theta_{i}, s)$ as initial state variables to be numerically integrated coefficient by coefficient; propagating by time $2\pi/\Omega_{p}$, we get the Taylor coefficients of $F(W_{<k}(\theta_{i}, s))$. Doing this for each $i = 0, \dots, N-1$ is enough to represent the Fourier-Taylor coefficients of $F(W_{<k}(\theta, s))$ on the computer, up to order $k$; the $s^{k}$ coefficient of this gives us $E_{k}(\theta)$. This approach for numerical integration of Taylor series is called jet transport; see \citet{perezpalau2015} for more details. Truncated Taylor series can be used with jet transport, since the automatic differentiation techniques used to evaluate time derivatives work with truncated series. Note that for an $n$-dimensional state ($n = 4$ in our case) and degree-$d$ truncated series, there are $n(d + 1)$ coefficients, which is the required dimension for the numerical integration. 

\subsection{Notes About Numerical Computation of Manifolds} 

We implemented the parameterization method, automatic differentiation, and jet transport of Sections \ref{paramsection} and \ref{jettransport} in a C program for computation of stable and unstable manifolds. For numerical integration, including jet transport, we used the Runge-Kutta Prince-Dormand (8,9) integrator from the GSL library \citep{gslManual}; integrations were parallelized using OpenMP with one thread for each $\theta_{i}$ value. We tested our tools by computing manifolds of some of the 3:4 and 5:6 Jupiter-Europa PERTBP tori shown in Fig. \ref{fig:56plot1} and \ref{fig:34plot1}, with $N$ ranging from 1024 to 2048. On a quad-core Intel i7 laptop CPU, the program took less than 10 seconds for the computation of $s^{5}$-order parameterizations. 

Note that in each step of order $k$, when $F(W_{<k} (\theta_{i}, s)) - W_{<k}(\theta_{i}+\omega, \lambda s)$ is computed in order to find $E_{k}(\theta_{i})$, the $s^{j}$ coefficients for $j<k$ should be zero (to compute the $W_{<k}(\theta_{i}+\omega, \lambda s)$ coefficients, use the translation algorithm from Section \ref{discretizationSection} on the arrays of $W_{j}(\theta_{i})$ values, and then multiply  $W_{j}(\theta_{i}+\omega)$ by $\lambda^{j}$). This behavior was indeed observed when running the program, serving as a check on the accuracy of the computation. The final $s^{d}$-degree truncated series $W_{\leq d}(\theta, s) = K(\theta) + \sum_{j = 1}^{d} W_{j}(\theta)s^{j} $ satisfies $F(W_{\leq d}(\theta_{i}, s)) - W_{\leq d}(\theta_{i}+\omega, \lambda s) =0$ up to terms of order $s^d$, for each $i = 0, \dots, N-1$.

In the $W_{k}(\theta)$ step, we truncate all series at $s^k$ for the automatic differentiation and jet transport steps; this optimizes computational time and storage requirements. Also, note that given $W(\theta, s)$ solving Eq. \eqref{invariancequationfinal}, $W(\theta, \alpha s)$ is also a solution for any $\alpha \in \mathbb{R}$. Sometimes, the jet transport integration may struggle to converge as a result of fast-growing coefficients $W_{j}(\theta)$ of $W(\theta, s)$; in this case, scaling $W(\theta, s)$ to $W(\theta, \alpha s)$ with $\alpha  <1$ can help. To do this, simply multiply $W_{1}(\theta)$ by $\alpha$ and then restart the order-by-order parameterization method algorithm. 

As a final remark, note that in certain systems, such as the PERTBP with $\theta_{p} = 0$ at $t=0$, the equations of motion have the same time-reversal symmetry as the PCRTBP. In this case, knowledge of the stable manifold $W^{s}(\theta, s)$ gives us the unstable manifold $W^{u}(\theta, s)$ simply by setting $W^{u}(\theta, s) = M W^{s}(2\pi-\theta, s)$ where $M$ is the diagonal matrix with diagonal entries $1, -1, -1$, and $1$. By doing this, we save half the computation time as compared to computing both $W^{s}$ and $W^{u}$. 

\subsection{Fundamental Domains of Parameterizations} 

The $d$ degree Fourier-Taylor parameterization $W_{\leq d}(\theta, s)$ of the manifolds of $K(\theta)$ will be more accurate than linear approximation by the stable or unstable direction at each point $K(\theta)$, However, due to series truncation error, $W_{\leq d}(\theta, s)$ is not exact. Furthermore, even the exact infinite series $W(\theta,s)$ satisfying Eq. \eqref{invariancequationfinal} would only be valid for $s$ within some radius of convergence. Thus, we must determine the values of $s \in \mathbb{R}$ for which $W_{\leq d}(\theta, s)$ accurately represents the invariant manifold. 

Fix an error tolerance, say $E_{tol} = 10^{-5}$ or $10^{-6}$. We now find what \citet{haroetal} call the fundamental domain of $W_{\leq d}(\theta, s)$; this is the largest set $\mathbb{T} \times (-D,D)$ such that for all $(\theta,s) \in \mathbb{T} \times (-D,D)$, the error in invariance Eq. \eqref{invariancequationfinal} is less than $E_{tol}$. That is, we seek the largest $D \in \mathbb{R}^{+}$ such that for all $s$ satisfying $|s| < D$, 
\begin{equation} \max_{\theta \in \mathbb{T}} \left\|F(W_{\leq d}(\theta, s)) - W_{\leq d}(\theta+\omega, \lambda s)\right\| < E_{tol} \end{equation}
In practice, since we know $K(\theta)$ and $W_{j}(\theta)$, $j=1, \dots, d$, at the values $\theta_{i}$, $i=0, \dots, N-1$, we search for the largest $D \in \mathbb{R}^{+}$ such that for all $s$ with $|s| < D $,  
\begin{equation} \max_{i=0, \dots, N-1} \left\|F(W_{\leq d}(\theta_{i}, s)) - W_{\leq d}(\theta_{i}+\omega, \lambda s)\right\| < E_{tol} \end{equation}
The simplest way of finding $D$ is to first use bisection to find the largest $D_{i}$ such that $\left\|F(W_{\leq d}(\theta_{i}, s)) - W_{\leq d}(\theta_{i}+\omega, \lambda s)\right\| < E_{tol} $ for all $s \in (-D_{i},D_{i})$. After doing this for $i=0, \dots, N-1$, $D$ will be the minimum of all the $D_{i}$. 

We computed the fundamental domains of validity for 5 different 3:4 and 5:6 PERTBP resonant torus manifold parameterizations. We found that the domains for $d = 5$ were 50-200 times larger than those for $d = 1$. For linear parameterizations ($d = 1$), the domain size $D$ of all test cases was on the order of $10^{-4}$ at best, generally $10^{-5}$. However, for the degree-5 parameterizations $W_{d \leq 5}(\theta, s)$, $D$ was on the order of $10^{-3}$ or 0.01. Higher degree parameterizations may improve even further. Note that a larger fundamental domain means that less numerical integration is required for manifold globalization, reducing the computation time. 

\section{Globalization, Regularization, and Visualization} \label{globoviz}

At this point, we have accurate local representations of stable and unstable manifolds of our stroboscopic map invariant circles. Given a manifold's Fourier-Taylor parameterization $W_{p}(\theta,s)$ and its fundamental domain $\mathcal{D} =\mathbb{T} \times (-D,D)$, the image $W_{p}(\mathcal D)$ gives us a piece of the manifold in the map phase space $\mathbb{R}^{4}$. However, this subset of the manifold will be close to its base invariant circle $K(\theta)$; generally, it is motions on the manifold further away from the torus that are of interest for applications. Hence, we need to extend our Fourier-Taylor parameterization $W_{p}:\mathcal{D}  \rightarrow \mathbb{R}^{4}$ to a function $W: \mathbb{T} \times \mathbb{R} \rightarrow \mathbb{R}^{4}$ parameterizing the entire manifold, with $W=W_{p}$ on $\mathcal D$. This is referred to as globalization. 

Recall from Eq. \eqref{invariancequationfinal} that $W$ must satisfy $F(W(\theta,s)) = W(\theta+\omega, \lambda s)$. Applying this repeatedly, we have that $F^{k}(W(\theta,s)) = W(\theta+k \omega, \lambda^{k}s)$, where the superscript $k \in \mathbb{Z}^{+}$ refers to function composition. We can rewrite this as 
\begin{equation} \label{globou} W(\theta, s) = F^{k}(W(\theta-k \omega, \lambda^{-k}s)) \end{equation}
\begin{equation} \label{globos} W(\theta, s) = F^{-k}(W(\theta+k \omega, \lambda^{k}s)) \end{equation}
Eq. \eqref{globou}-\eqref{globos} allow us to define $W(\theta, s)$ for all $(\theta, s) \in \mathbb{T} \times \mathbb{R}$. If $W$ is an unstable manifold with $|\lambda| > 1$, choose $k \geq 0$ such that $| \lambda^{-k}s | < D$ and use $W_{p}$ to evaluate Eq. \eqref{globou}; if $W$ is a stable manifold with $|\lambda| < 1$, take $k \geq 0$ such that $| \lambda^{k}s | < D$ and evaluate Eq. \eqref{globos}. The map $F^{k}$ (or $F^{-k}$) is just a time $2\pi k/\Omega_{p}$ (or $-2\pi k/\Omega_{p}$) numerical integration. $W(\theta,s)$ thus defined satisfies $F(W(\theta,s)) = W(\theta+\omega, \lambda s)$ for all $(\theta,s) \in \mathbb{T} \times \mathbb{R}$, so the image $W(\mathbb{T} \times \mathbb{R})$ is $F$-invariant. Hence, Eq. \eqref{globou}-\eqref{globos} give us a global representation of the entire stable or unstable manifold. Note that $W$ can be differentiated easily with respect to $\theta$ and $s$ to get the tangent vectors to the manifold, as $DF^{\pm k}$ is a state transition matrix and $DW_{p}$ only requires polynomial or Fourier series differentiation. This can be useful for differential correction of approximate heteroclinic connections; see \citet{kumar2021feb} for details.

\subsection{Mesh Representations of Globalized Manifolds} \label{meshSection}

For visualization, we often want to calculate a mesh of many points on the manifold, rather than just a few $W(\theta,s)$ values. To do this, we first take an evenly-spaced grid of $L$ $s$-values $\{s_{j}\}$ from $-D$ to $D$, in addition to  our grid of $N$ $\theta$ values $\theta_{i}$, and then directly evaluate the Fourier-Taylor parameterization to  compute $W(\theta_{i}, s_{j})$ for each $i = 0, \dots, N-1$, $j = 1, \dots, L$. Next, we repeatedly apply $F$ or $F^{-1}$ to the $W(\theta_{i}, s_{j})$ to get the points $W(\theta_{i}+k\omega, \lambda^{k}s_{j}) = F^{k}(W(\theta_{i}, s_{j}))$ if $|\lambda| > 1$ or $W(\theta_{i}-k\omega, \lambda^{-k}s_{j}) = F^{-k}(W(\theta_{i}, s_{j}))$ if $|\lambda| < 1$, for  $k = 0, 1, 2, \dots$ up to some $k_{max} \in \mathbb{Z}^{+}$. The numerical integrations required in this step may require use of regularized equations of motion, which we will discuss in Section \ref{regSection}. Finally, we use the translation algorithm given in Eq. \eqref{translate} to find the points $W(\theta_{i}, \lambda^{k}s_{j})$ if $|\lambda| > 1$ or $W(\theta_{i}, \lambda^{-k}s_{j})$ if $|\lambda| < 1$; we need all points to be at the same set of $N$ $\theta_{i}$ values when trying to create a manifold mesh that can be plotted. 

By following this procedure, we get a discretized, plottable representation of the manifold subset $\{ W(\theta,s): (\theta,s) \in \mathbb{T} \times [-M, M] \}$, where $M=\lambda^{k_{max}}D$ if $|\lambda|>1$ and $M=\lambda^{-k_{max}}D$ if $|\lambda|<1$.  Note that the numerical integrations can be parallelized across $\theta_{i}$ values, which we took advantage of. A 3D projection of an example globalized stable manifold mesh (denoted $W^{s}$) of a 3:4 Jupiter-Europa PERTBP invariant circle is given in Fig. \ref{fig:globalmani34}, with $N=1024$, $L=101$, $k_{max} = 6$. 

\begin{figure}
\begin{centering}
\includegraphics[width=0.65\columnwidth]{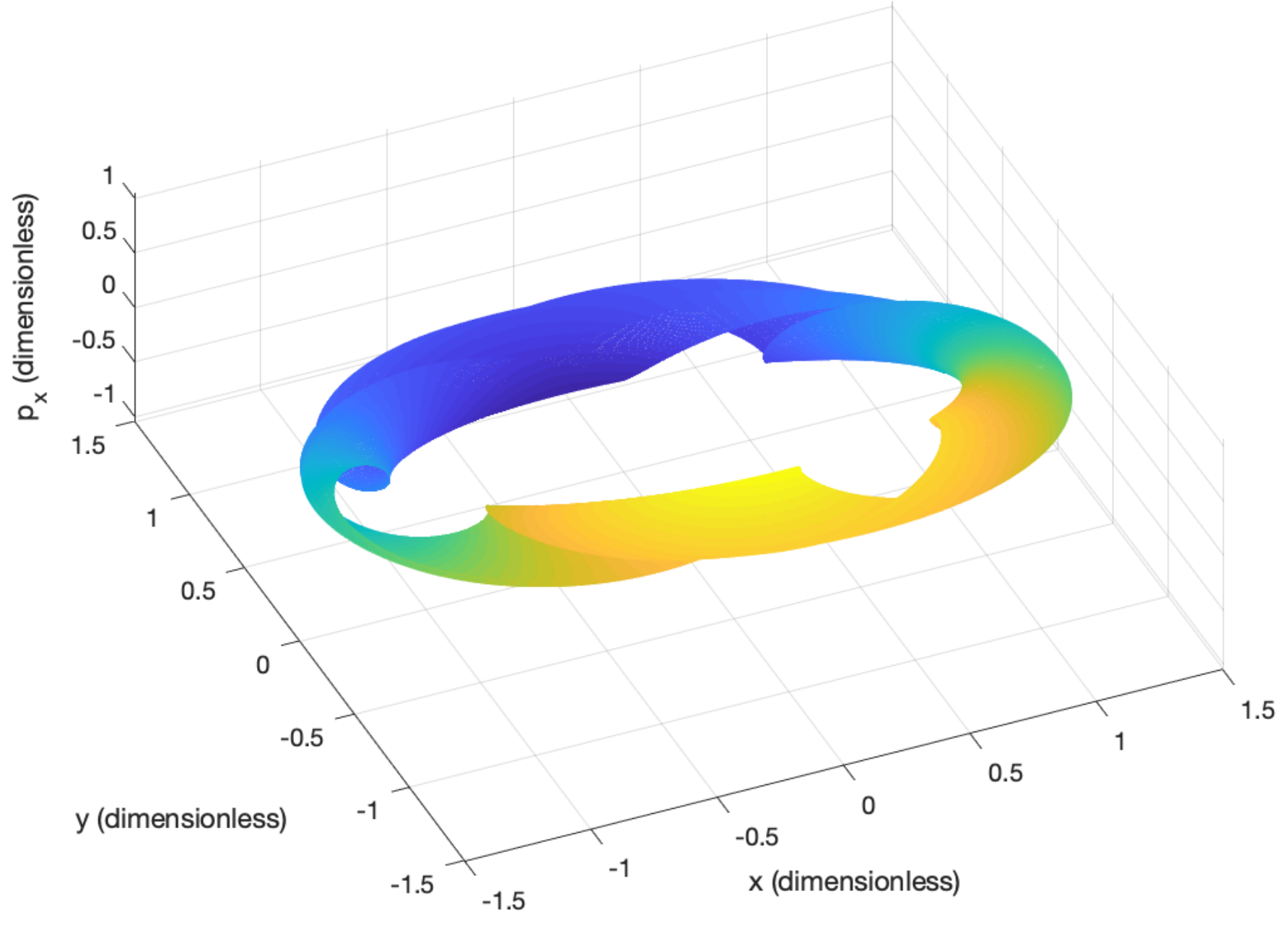}
\caption{ \label{fig:globalmani34} $(x,y,p_{x})$ projection of Jupiter-Europa PERTBP 3:4 $W^{s}$ for $\omega=1.559620297$}
\end{centering}
\end{figure}

\subsection{The Need for Regularization: An Extension of Levi-Civita to the PERTBP} \label{regSection}

In the equations of motion for the PERTBP and other periodically-perturbed PCRTBP models, the positions of the two large masses $m_{1}$ and $m_{2}$ are singularities. However, when numerically integrating points forwards or backwards during the manifold mesh computation described in Section \ref{meshSection}, it is possible for some points' trajectories to pass extremely close to the singularity at $m_{2}$. Moreover, this can indicate that the manifold being computed actually passes through the $m_{2}$ singularity. Such behavior was observed, for example, during computation of manifold meshes for 5:6 Jupiter-Europa PERTBP tori. These close approaches to $m_{2}$ can result in numerical issues, including lack of integrator convergence. 

In the PCRTBP, the Levi-Civita regularization is very commonly used to compute trajectories which pass near or through a singularity; see \citet{celletti} for full details. First, a canonical coordinate transformation is applied to the PCRTBP Hamiltonian $H_{0}$ from Eq. \eqref{pcrtbpH}. This is followed by the addition of a pair of action-angle variables to the transformed Hamiltonian; the new action's value is set to $-H_{0}$, which has a constant value along the trajectory. This finally allows a time-rescaling to be used which cancels the singularity. This method, however, relies on the fact that $H_{0}$ is constant along PCRTBP trajectories. For our periodically-perturbed models, this is not the case. Hence, some modification is required. 

For the PERTBP, the singularity corresponding to $m_{2}$ is the time-varying point $(x,y) = ( (1-\mu)(1-\varepsilon \cos E(t)),0 )$. We now present the derivation of the modified Levi-Civita regularization of $m_{2}$ for the PERTBP; we expect very similar methods to apply for other periodically-perturbed PCRTBP models as well. Readers primarily interested in using the final regularized equations for numerical integration should skip to Section \eqref{regUsage}. The following is heavily inspired by \citet{celletti}. 

First, take the PERTBP Hamiltonian $H_{\varepsilon}$ from Eq. \eqref{pertbpH} and add a momentum variable $p_{t}$ conjugated to $t$. The Hamiltonian and equations of motion become
\begin{equation}  \label{pertbpH_new} \bar H_{\varepsilon}(p_{x}, p_{y}, p_{t}, x,y,t)= p_{t}+ \frac{p_{x}^{2}+p_{y}^{2}}{2} + n(t)(p_{x}y -p_{y}x) - \frac{1-\mu}{r_{1}} - \frac{\mu}{r_{2}} \end{equation}
\begin{equation} \label{H_EOM_new} \dot x = \frac{\partial \bar H_{\varepsilon}}{\partial p_{x}} \quad \dot y = \frac{\partial \bar H_{\varepsilon}}{\partial p_{y}} \quad \dot t = \frac{\partial \bar H_{\varepsilon}}{\partial p_{t}} \quad \quad \dot p_{x} = -\frac{\partial \bar H_{\varepsilon}}{\partial x} \quad \dot p_{y} = -\frac{\partial \bar H_{\varepsilon}}{\partial y} \quad \dot p_{t} = -\frac{\partial \bar H_{\varepsilon}}{\partial t} \end{equation}
where $r_{1} = \sqrt{\big(x+\mu(1+\chi(t))\big)^{2} + y^{2}} $, $ r_{2} = \sqrt{\big(x-(1-\mu)(1+\chi(t))\big)^{2} + y^{2}}$, and $\chi(t) = -\varepsilon \cos E(t)$. Note that adding $p_{t}$ does not change the values of $\dot x$, $\dot y$, $\dot p_x$, $\dot p_y$, and $\dot t = \partial \bar H_{\varepsilon} / \partial p_{t} = 1$ as compared to using Eq. \eqref{pertbpH}. However, unlike $H_{\varepsilon}$, the new Hamiltonian $\bar H_{\varepsilon}$ does remain constant along trajectories in $(p_{x}, p_{y}, p_{t}, x,y,t)$ space. Also, given an initial condition $(x,y,p_{x},p_{y},t)$ to be propagated, the initial value of $p_{t}$ should be set so that $\bar H_{\varepsilon}=0$; this will be important later on.

Now, we perform a canonical coordinate transformation. This is required in order to ``straighten out" certain trajectories passing through the singularity which make sharp bends in physical space \citep{cellettiReg}. Define a generating function 
\begin{equation}  \label{genfunc} W(p_{x}, p_{y}, p_{t}, X, Y, T)= p_{x}\left(X^{2}-Y^{2}+(1-\mu)(1+\chi(T)) \right) + p_{2}(2XY) +p_{t}T \end{equation}
which is a function of the old momenta and new configuration space coordinates. Then, this defines a transformation between the old $(p_{x}, p_{y}, p_{t}, x, y, t)$ variables and new $(P_{X}, P_{Y}, P_{T}, X, Y, T)$ variables through the relations \citep{thirring}
\begin{align} \begin{split} \label{canontrans}
&x = \frac{\partial W}{\partial p_{x}} = X^{2}-Y^{2}+(1-\mu)(1+\chi(T))  \quad \quad  y = \frac{\partial W}{\partial p_{y}} = 2XY \\
 &P_{X} = \frac{\partial W}{\partial X} = 2p_{x}X+2p_{y}Y \quad \quad  \quad   P_{Y} = \frac{\partial W}{\partial Y} = -2p_{x}Y+2p_{y}X \\
  &\quad t = \frac{\partial W}{\partial p_{t}} = T   \quad \quad \quad \quad \quad \quad \quad \quad P_{T} = \frac{\partial W}{\partial T} = (1-\mu)p_{x} \frac{d \chi}{d t}(T) + p_{t}  \\
  \end{split} \end{align}
Eq. \eqref{canontrans} gives us $x$, $y$, and $t$ in terms of the new variables. We can also solve for $p_{x}$ and $p_{y}$ to get $p_{x}=\frac{2}{R}(P_{X}X-P_{Y}Y)$ and  $p_{y}=\frac{2}{R}(P_{X}Y+P_{Y}X)$, where $R = 4(X^{2}+Y^{2})$. This then gives us $p_{t} =  P_{T}-\frac{2}{R}(1-\mu)(P_{X}X-P_{Y}Y) \frac{d \chi}{d t}(T) $. Note that in the new variables, $r_{2} = \sqrt{(X^{2}-Y^{2})^{2}+(2XY)^{2}}=R/4$. 

Substituting the previous expressions for $(p_{x}, p_{y}, p_{t}, x, y, t)$ into Eq. \eqref{pertbpH_new} gives  
\begin{align} \begin{split}  \label{pertbpH_transformed} \mathcal H_{\varepsilon}(P_{X},& P_{Y}, P_{T}, X, Y, T)= \\
& P_{T}-\frac{2}{R}(1-\mu)(P_{X}X-P_{Y}Y) \frac{d \chi}{d t}(T) + \frac{P_{X}^{2}+P_{Y}^{2}}{2R}  \\
&+ 2 n(T) \left[\frac{1}{4}(P_{X}X-P_{Y}Y) - \frac{(1-\mu)}{R}(1+\chi(T))(P_{X}Y+P_{Y}X)\right] \\
&- \frac{1-\mu}{\sqrt{(X^{2}+Y^{2})^{2}+(1+\chi(T))^{2}+2(X^{2}-Y^{2})(1+\chi(T))}} - \frac{4\mu}{R}  
 \end{split} \end{align}
The $m_{2}$ singularity is now at $(X,Y)=(0,0)$, where $R=4r_{2}=0$. Since this was a canonical transformation, the equations of motion in the new coordinates will be
\begin{equation}  \dot X = \frac{\partial \mathcal H_{\varepsilon}}{\partial P_{X}} \,\,\, \dot Y = \frac{\partial \mathcal H_{\varepsilon}}{\partial P_{Y}} \,\,\, \dot T = \frac{\partial \mathcal H_{\varepsilon}}{\partial P_{T}} \,\,\,\, \,\,\,\, \dot P_{X} = -\frac{\partial \mathcal H_{\varepsilon}}{\partial X} \,\,\, \dot P_{Y} = -\frac{\partial \mathcal H_{\varepsilon}}{\partial Y} \,\,\, \dot P_{T} = -\frac{\partial \mathcal H_{\varepsilon}}{\partial T} \end{equation}
To regularize the singularity at $R=0$, we want to be able to use $R \mathcal H_{\varepsilon}$ instead of $\mathcal H_{\varepsilon}$. For this, define a rescaled time $s$ such that $dt = R \, ds$. Then, we have that $\frac{d}{ds}=\frac{d}{dt} \frac{dt}{ds}=R \frac{d}{dt}$. Thus, letting prime ($'$) denote $d/ds$,
\begin{equation} \begin{gathered} X' = R \frac{\partial \mathcal H_{\varepsilon}}{\partial P_{X}} \quad Y' = R \frac{\partial \mathcal H_{\varepsilon}}{\partial P_{Y}} \quad T' = R \frac{\partial \mathcal H_{\varepsilon}}{\partial P_{T}} \\
  P_{X}' = -R \frac{\partial \mathcal H_{\varepsilon}}{\partial X} \quad P_{Y}' = -R \frac{\partial \mathcal H_{\varepsilon}}{\partial Y} \quad P_{T}' = -R \frac{\partial \mathcal H_{\varepsilon}}{\partial T} \end{gathered} \end{equation}
Since $R$ is a function of only $X$ and $Y$, it is immediate that $X' = \frac{\partial [R \mathcal H_{\varepsilon}]}{\partial P_{X}} $, $ Y' =  \frac{\partial [R \mathcal H_{\varepsilon}]}{\partial P_{Y}} $, $T' =  \frac{\partial [R \mathcal H_{\varepsilon}]}{\partial P_{T}} $, and $P_{T}' = - \frac{\partial [R \mathcal H_{\varepsilon}]}{\partial T}$. Furthermore, since $p_{t}$ was chosen earlier to ensure $\bar H_{\varepsilon}=0$, we also will have $\mathcal H_{\varepsilon}=0$ along the trajectory in $(P_{X}, P_{Y}, P_{T}, X, Y, T)$ space. Hence, we find that $\frac{\partial [R \mathcal H_{\varepsilon}]}{\partial X} = R \frac{\partial  \mathcal H_{\varepsilon}}{\partial X} + \mathcal H_{\varepsilon} \frac{\partial R }{\partial X}=R \frac{\partial  \mathcal H_{\varepsilon}}{\partial X} $. This yields $P_{X}' = - \frac{\partial [R \mathcal H_{\varepsilon}]}{\partial X}$; we similarly find $P_{Y}' = - \frac{\partial [R \mathcal H_{\varepsilon}]}{\partial Y}$. As $R \mathcal H_{\varepsilon}$ has no singularity at $R=0$, we thus obtain the $m_{2}$-regularized time-$s$ equations of motion
\begin{equation} \label{primehatH} \begin{gathered}  X' = \frac{\partial [R \mathcal H_{\varepsilon}]}{\partial P_{X}} \quad Y' = \frac{\partial [R \mathcal H_{\varepsilon}]}{\partial P_{Y}} \quad T' = \frac{\partial [R \mathcal H_{\varepsilon}]}{\partial P_{T}} \\
P_{X}' = - \frac{\partial [R \mathcal H_{\varepsilon}]}{\partial X} \quad P_{Y}' = - \frac{\partial [R \mathcal H_{\varepsilon}] }{\partial Y} \quad P_{T}' = - \frac{\partial [R \mathcal H_{\varepsilon}]}{\partial T} \end{gathered} \end{equation}
 
\subsubsection{Usage of Regularized PERTBP Equations of Motion} \label{regUsage}
 
Let $(x^{i}, y^{i}, p^{i}_{x}, p^{i}_{y})$ be an initial state we wish to integrate from $t=t_{i}$ to $t_{f}$ in the PERTBP. Recall $H_{\varepsilon}$ from Eq. \eqref{pertbpH}, and $\bar H_{\varepsilon}$ from Eq. \eqref{pertbpH_new}, with $\chi(t) = -\varepsilon \cos E(t)$. To use the $m_{2}$-regularized equations of motion for this integration, we:
\begin{enumerate}
\item Set $p^{i}_{t} = - H_{\varepsilon}(p^{i}_{x}, p^{i}_{y}, x^{i},y^{i},t_{i})$, so that $\bar H_{\varepsilon}(p^{i}_{x}, p^{i}_{y}, p^{i}_{t}, x^{i},y^{i},t_{i}) = 0$. 
\item Compute initial $(P^{i}_{X}, P^{i}_{Y}, P^{i}_{T}, X^{i}, Y^{i}, T^{i})$ using the equations (where $j = \sqrt{-1}$)
\begin{align} \label{oldtonewtrans} \begin{split} 
 &X+jY = \big( x - (1-\mu)(1+\chi(t)) + j y \big)^{1/2}  \\
 P_{X} = &2p_{x}X+2p_{y}Y \quad \quad  \quad\quad   P_{Y} = -2p_{x}Y+2p_{y}X \\
   &T = t  \quad   \quad \quad \quad \quad \quad \quad P_{T} = (1-\mu)p_{x} \frac{d \chi}{d t}(t) + p_{t}  \\
  \end{split} \end{align} 
The sign chosen during the complex square root for $X+jY$ does not matter.
\item Integrate the initial condition $(P^{i}_{X}, P^{i}_{Y}, P^{i}_{T}, X^{i}, Y^{i}, T^{i})$ using Eq. \eqref{primehatH}, where $R=4(X^{2}+Y^{2})$ and $\mathcal H_{\varepsilon}$ is given by Eq. \eqref{pertbpH_transformed}. Only stop the integration when $T = t_{f}$; do not stop before this occurs, even if the integration time reaches $t_{f}$.  
\item Transform the resulting final state back to $(p_{x}, p_{y}, p_{t}, x,y,t)$ coordinates using
\begin{align}  \label{newtooldtrans} \begin{split} 
&x =  X^{2}-Y^{2}+(1-\mu)(1+\chi(T))  \quad \quad \quad  y = 2XY \\
 &p_{x}=\frac{2}{R}(P_{X}X-P_{Y}Y) \quad \quad \quad \quad   p_{y}=\frac{2}{R}(P_{X}Y+P_{Y}X) \\
  & t = T =t_{f}  \quad \quad \quad \quad \quad p_{t} =  P_{T}-\frac{2}{R}(1-\mu)(P_{X}X-P_{Y}Y) \frac{d \chi}{d t}(T)  \\
  \end{split} \end{align}
\end{enumerate}
The first line of Eq. \eqref{oldtonewtrans} should be interpreted as two real equations corresponding to setting real and imaginary parts of both sides equal. It follows from the first line of Eq. \eqref{canontrans} combined with the relation $ (X+jY)^{2} = X^{2}-Y^{2} + j (2XY) $. Also, for step 3 above, the requirement to stop integration when $T=t_{f}$ can be implemented using the ``events" functionality of MATLAB's ODE solvers, or the callback features in Julia's DifferentialEquations.jl library. 

The partial derivatives of $R \mathcal H_{\varepsilon}$ appearing in the equations of motion Eq. \eqref{primehatH} are straightforward to compute but lengthy, so we do not write them here. Throughout the steps listed above, as well as for computing the partial derivatives, we need the quantities $\frac{d\chi}{dt}$, $\frac{d^{2}\chi}{dt^{2}}$, and $\frac{dn}{dt}$. These are given by the equations 
\begin{equation*} \frac{d\chi}{dt} = \frac{\varepsilon \sin E}{1-\varepsilon \cos E} \quad \quad \frac{d^{2}\chi}{dt^{2}} = \frac{\varepsilon \cos E - \varepsilon^{2}}{(1-\varepsilon \cos E)^{3}}  \quad \quad  \frac{dn}{dt} = \frac{-2 \varepsilon  \sqrt{1-\varepsilon^{2}}}{(1-\varepsilon \cos E)^{4}}\sin E \end{equation*}
which can be derived from $\chi(t) = -\varepsilon \cos E(t)$, $n(t) = \frac{\sqrt{1-\varepsilon^{2}}}{(1-\varepsilon \cos E(t))^{2}}$, and the relation $\frac{dE}{dt} =  \frac{1}{1-\varepsilon \cos E}$ (which in turn follows from taking the time derivative of the standard Kepler's equation $M = E - \varepsilon \sin E$; see \citet{bmw}). 

\subsubsection{Computational Results}

\begin{figure}
\begin{centering}
\includegraphics[width=0.65\columnwidth]{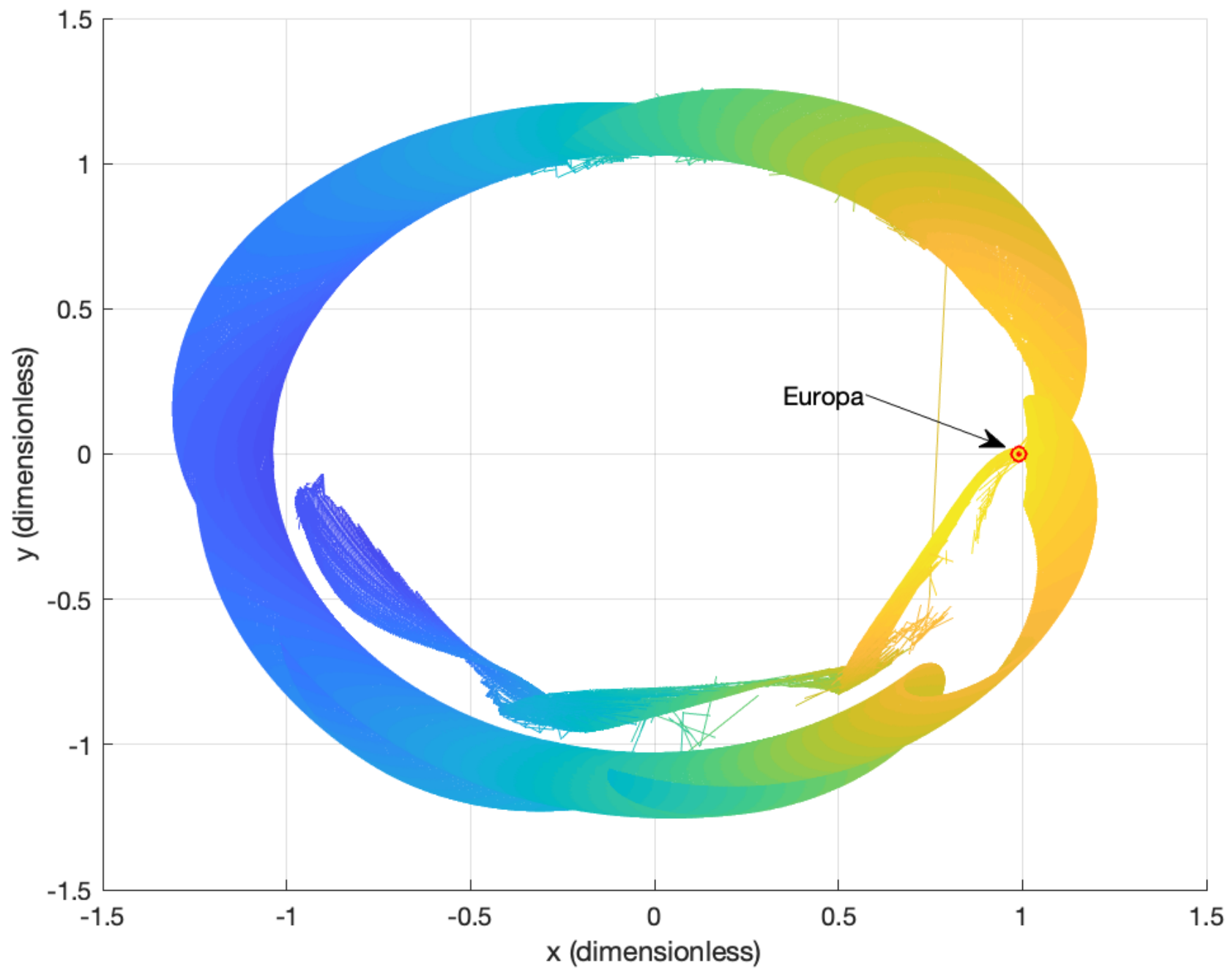}
\caption{ \label{fig:globalmani56} $(x,y)$ projection of Jupiter-Europa PERTBP 5:6 $W^{s}$ for $\omega=1.030011437$}
\end{centering}
\end{figure}

In Fig. \ref{fig:globalmani56}, we show a 2D projection of an example globalized stable manifold mesh for a 5:6 Jupiter-Europa PERTBP invariant circle. This was computed using the regularized PERTBP equations of motion to evaluate $F^{-k}$ in the procedure described in Section \ref{meshSection}; in this case, $N=2048$, $L=101$, $k_{max} = 6$. We have filtered the computed mesh points so as to only plot those which did not result in a very ``visually discontinuous" mesh; this filtering is needed during visualization, since close flybys of $m_{2}$ can send points which started close together in extremely different directions. Nevertheless, even after discarding some mesh points, it is clearly visible in the figure that the manifold passes through the singularity at $m_{2}$ (Europa, marked as a red circle). Using the regularized equations, we experienced no warnings of integrator divergence during program runtime, which were encountered when using the unregularized equations. 

We also used the regularized equations of motion to recompute the 3:4 $W^{s}$ manifold mesh shown earlier in Fig. \ref{fig:globalmani34}. The results matched those which were obtained earlier when using Eq. \eqref{H_EOM} and \eqref{pertbpH} for the numerical integrations, thus verifying the correctness of the regularization procedure. We carried out this computation in Julia, using the DP5 integrator and parallelizing across $\theta_{i}$ with the EnsembleProblem feature of DifferentialEquations.jl \citep{DifferentialEquations.jl-2017}; the computation of the 3:4 manifold with $N=1024$, $L=101$, $k_{max} = 15$ took approximately 250 seconds on the same quad core i7 laptop CPU used earlier. 

\section{Conclusion}

In this paper, we first developed a quasi-Newton method for the simultaneous computation of unstable invariant circles and their symplectic conjugate center, stable, and unstable bundles for stroboscopic maps of periodically-perturbed PCRTBP models. Our method improves the computational complexity of the torus calculation to $O(N \log N)$ as compared to $O(N^{3})$ for the methods \hl{used in almost all the existing astrodynamics literature}, in addition to giving useful information on the torus stable and unstable directions. \hl{Our method also extends the $O(N \log N)$ method of \citet{haroLlave} and \citet{haroetal} to unstable tori with center directions, as is the case for the vast majority of celestial mechanics applications}. We used this quasi-Newton method for continuation of tori and bundles by both perturbation parameter and rotation number, and described how to initialize the continuation from PCRTBP periodic orbits. We also gave a set of numerical best practices to aid in quasi-Newton method convergence. 

After finding the tori and bundles, we used the results of the continuation to start an order-by-order method for the computation of Fourier-Taylor series parameterizations of stable and unstable manifolds for the invariant circles. We found significant improvements in accuracy and fundamental domain size compared to linear manifold approximations. Finally, we were able to extend the parameterizations to compute points of the manifolds outside the fundamental domain, with the aid of a modified Levi-Civita regularization which we derived for the PERTBP. We expect that similar methods can be used to regularize other periodically-perturbed PCRTBP models as well. 

The tools developed were tested in the Jupiter-Europa PERTBP, with the calculations of the circles, bundles, and manifold parameterizations taking just a few seconds on a 2017-era laptop with a quad-core Intel i7 CPU. Our Julia program for computation of meshes of globalized manifold points took a few minutes for each manifold, due to the large number of numerical integrations involved. As we describe in another paper \citep{kumar2021feb}, with the help of modern computer graphics processing units, these manifold parameterizations and meshes can be used to very rapidly search for and accurately compute intersections of stable and unstable manifolds leading to heteroclinic connections. The methods presented in this paper can form an important component for low-energy mission design and transfer trajectories in such periodically-perturbed PCRTBP models. 

\begin{acknowledgements}
This work was supported by a NASA Space Technology Research Fellowship. This research was carried out in part at the Jet Propulsion Laboratory, California Institute of Technology, under a contract with the National Aeronautics and Space Administration.  R.L was supported in part by NSF grant DMS 1800241. Special thanks to Dr. Lei Zhang for providing code from his earlier work  \citep{zhang} which formed a base off of which to build much the code written for the quasi-Newton method and manifold parameterization. Also many thanks to Professors Alex Haro, Josep-Maria Mondelo, Angel Jorba, and Josep Masdemont for useful discussions in Barcelona which were of great assistance in this work. B.K.'s 2019 visit to Barcelona was funded by European Commission grant H2020-MSCA-RISE-2015, Project \#734557 (Trax). Many thanks to the Italian Gruppo Nazionale Per La Fisica Matematica for hosting  B.K. at the Fall 2019 summer school in Ravello, Italy. An earlier and less detailed version of this paper was presented and included in the proceedings of the 2020 AAS/AIAA Astrodynamics Specialist Conference (paper number AAS 20-694). 
\end{acknowledgements}

\appendix
\section{Proof of Vanishing $E_{CC}$ Average} \label{appendix}
In Section \ref{Pstep}, it was mentioned that the average of $E_{CC}(\theta)$ goes to zero with each quasi-Newton step. We can prove this using a method somewhat inspired by the vanishing lemma proof of \citet{fontichLlaveSire}. For ease of notation, denote this average as $\lambda_{c} = \hat E_{CC}(0)$, and $\tilde E_{CC}(\theta) = E_{CC}(\theta) - \lambda_{c}$, so that $\tilde E_{CC}$ has zero average. Also write $\bold{e}_{2} = \left[ 0 \,\, 1 \, \,0 \,\, 0\right]^{T}$. 
\begin{proof}
Let $\bold{v}_{c}(\theta)$ denote the second column of $P$, and define $E_{C}(\theta) = \left[ E_{LC} \, \tilde E_{CC} \, E_{SC} \, E_{UC} \right]^{T}$; note that $E_{C}(\theta) + \lambda_{c} \bold{e}_{2}$ is simply the second column of $E_{red}(\theta)$. Left multiplying the definition of $E_{red}$ (Eq. \eqref{Ereddef}) by $P(\theta+\omega)$ and taking column 2 of the result gives
\begin{equation} \label{col2PEred} P(\theta+\omega) \left( E_{C}(\theta) + \lambda_{c} \bold{e}_{2} \right) = DF(K(\theta)) \bold{v}_{c}(\theta) -  \bold{v}_{c}(\theta+ \omega) - T(\theta) DK(\theta+\omega) \end{equation} 
Define $\mathcal E_{C}(\theta) = P(\theta+\omega) E_{C}(\theta)$, and note that $P(\theta+\omega) \bold{e}_{2} = \bold{v}_{c}(\theta+\omega)$. Thus, Eq. \eqref{col2PEred} gives
\begin{equation} DF(K(\theta)) \bold{v}_{c}(\theta) = (1+\lambda_{c}) \bold{v}_{c}(\theta+ \omega) + T(\theta) DK(\theta+\omega) +\mathcal E_{C}(\theta) \end{equation} 
Now, differentiating Eq. \eqref{Edef} yields $DF(K(\theta))DK(\theta) = DK(\theta+\omega)+DE(\theta)$. As mentioned in the proof of Lemma \ref{sympconjLemma}, $F$ satisfies $\Omega(\bold{v}_{1}, \bold{v}_{2}) = \Omega(DF(K(\theta)) \bold{v}_{1}, DF(K(\theta)) \bold{v}_{2})$ for all $\bold{v}_{1}$, $\bold{v}_{2} \in \mathbb{R}^{4}$, where $\Omega$ is the symplectic form defined by $ \Omega(\bold{v}_{1}, \bold{v}_{2}) =  \bold{v}_{1}^{T} J \bold{v}_{2} $. Thus, 
\begin{align} \begin{split} \Omega(\bold{v}_{c}(\theta), &DK(\theta)) = \Omega(DF(K(\theta)) \bold{v}_{c}(\theta), DF(K(\theta)) DK(\theta))  \\
&= \Omega \Big((1+\lambda_{c}) \bold{v}_{c}(\theta+ \omega) + T(\theta) DK(\theta+\omega) +\mathcal E_{C}(\theta) , DK(\theta+\omega)+DE(\theta) \Big) \\
&= (1+\lambda_{c}) \Omega \Big( \bold{v}_{c}(\theta+ \omega) , DK(\theta+\omega) \Big)  + \mathcal O(DE(\theta)) + \mathcal O(\mathcal E_{C}(\theta)) 
\end{split} \end{align} 
where we use $\Omega(DK(\theta+\omega) , DK(\theta+\omega))=0$ to get the last line. This yields
\begin{align} \begin{split} \int_{0}^{2\pi} \Omega(\bold{v}_{c}(\theta), DK(\theta)) \, d\theta= (1+\lambda_{c})  \int_{0}^{2\pi} \Omega \big( \bold{v}_{c}(\theta+ \omega) &, DK(\theta+\omega) \big) \, d\theta \\ 
&+ \mathcal O(DE(\theta)) + \mathcal O(\mathcal E_{C}(\theta)) 
 \end{split}\end{align} 
Recognizing that $ \int_{0}^{2\pi} \Omega(\bold{v}_{c}(\theta), DK(\theta)) \, d\theta= \int_{0}^{2\pi} \Omega ( \bold{v}_{c}(\theta+ \omega) , DK(\theta+\omega) ) \, d\theta $, we have 
 \begin{equation} \lambda_{c} \int_{0}^{2\pi} \Omega(\bold{v}_{c}(\theta), DK(\theta)) \, d\theta=  \mathcal O(DE(\theta)) + \mathcal O(\mathcal E_{C}(\theta)) 
\end{equation} 
 Now, for $E$ and $ E_{red}$ small enough, $\bold{v}_{c}(\theta)$ is an approximate symplectic conjugate to $DK(\theta)$. This means that $\Omega(\bold{v}_{c}(\theta), DK(\theta)) \approx 1$ (see Eq. \eqref{areaOne}), so $ \int_{0}^{2\pi} \Omega(\bold{v}_{c}(\theta), DK(\theta)) \, d\theta = \mathcal O(1)$. Hence, it must be that $\lambda_{c} =  \mathcal O(DE(\theta)) + \mathcal O(\mathcal E_{C}(\theta)) $, so that as the quasi-Newton method reduces $DE(\theta)$ and $E_{C}(\theta)$ (and thus also $\mathcal E_{C}(\theta)$) to zero, $\lambda_{c}$ goes to zero as well. \qed
 \end{proof}
 
When carrying out the quasi-Newton step of Section \ref{Pstep} for correcting $P$ and $\Lambda$, Eq. \eqref{E_LC}, \eqref{E_SC}, and \eqref{E_UC} can be solved exactly (including for non-zero averages on the LHS), which quadratically reduces the $E_{LC}$, $E_{SC}$, and $E_{UC}$ components of $E_{C}(\theta)$ (using the definitions given in the above proof). On the other hand, Eq. \eqref{E_CC} for $E_{CC}$ can be written as 
\begin{equation} -E_{CC}(\theta) = -\tilde E_{CC}(\theta) -\lambda_{c}= Q_{CC}(\theta) - Q_{CC}(\theta+\omega) \\
\end{equation} 
As mentioned near the end of Section \ref{Pstep}, we ignore the nonzero LHS average $-\lambda_{c}=-\hat E_{CC}(0)$ when solving for $Q_{CC}$. Thus, what happens is that the zero-average part $\tilde E_{CC}(\theta)$ is quadratically reduced by the quasi-Newton step, but $\lambda_{c}$ may initially remain in $E_{red}$. However, the quadratic reductions in $E_{LC}$,  $\tilde E_{CC}(\theta)$, $E_{SC}$, and $E_{UC}$, and subsequently also in $E(\theta)$ during the following $K$-correction step, quadratically reduce $E_{C}(\theta)$ and $DE(\theta)$. This necessitates a reduction in $\lambda_{c}$ as described at the end of the above proof. 

\bibliographystyle{spbasic}   
\bibliography{references}   

\begin{thebibliography}{41}
\providecommand{\natexlab}[1]{#1}
\providecommand{\url}[1]{{#1}}
\providecommand{\urlprefix}{URL }
\expandafter\ifx\csname urlstyle\endcsname\relax
  \providecommand{\doi}[1]{DOI~\discretionary{}{}{}#1}\else
  \providecommand{\doi}{DOI~\discretionary{}{}{}\begingroup
  \urlstyle{rm}\Url}\fi
\providecommand{\eprint}[2][]{\url{#2}}

\bibitem[{Ahlfors(1979)}]{ahlfors}
Ahlfors LV (1979) Complex analysis : an introduction to the theory of analytic
  functions of one complex variable, 3rd edn. International series in pure and
  applied mathematics, McGraw-Hill, New York

\bibitem[{Anderson et~al.(2016)Anderson, Campagnola, and
  Lantoine}]{Anderson2016}
Anderson RL, Campagnola S, Lantoine G (2016) Broad search for unstable resonant
  orbits in the planar circular restricted three-body problem. Celestial
  Mechanics and Dynamical Astronomy 124(2):177--199

\bibitem[{Anderson et~al.(2019)Anderson, Campagnola, Koh, McElrath, and
  Woollands}]{Anderson2019}
Anderson RL, Campagnola S, Koh D, McElrath TP, Woollands RM (2019) Endgame
  design for {Europa} lander: {Ganymede} to {Europa} approach. In: AAS/AIAA
  Astrodynamics Specialist Conference, Portland, ME, AAS 19-745

\bibitem[{Andreu(1998)}]{andreu1998quasi}
Andreu M (1998) The quasi-bicircular problem. PhD thesis, Citeseer

\bibitem[{Bate et~al.(1971)Bate, Mueller, and White}]{bmw}
Bate RR, Mueller DD, White JE (1971) {Fundamentals} of {Astrodynamics}. Dover
  Publications, New York

\bibitem[{Berz and Makino(1998)}]{Berz1998}
Berz M, Makino K (1998) Verified integration of {ODEs} and flows using
  differential algebraic methods on high-order taylor models. Reliable
  Computing 4(4):361--369, \doi{10.1023/A:1024467732637},
  \urlprefix\url{https://doi.org/10.1023/A:1024467732637}

\bibitem[{Bosanac(2018)}]{Bosanac2018}
Bosanac N (2018) Bounded motions near resonant orbits in the {Earth-Moon} and
  {Sun-Earth} systems. In: AAS/AIAA Astrodynamics Specialist Conference,
  Snowbird, Utah

\bibitem[{Cabr{\'e} et~al.(2005)Cabr{\'e}, Fontich, and de~la
  Llave}]{CabreFontichLlave}
Cabr{\'e} X, Fontich E, de~la Llave R (2005) The parameterization method for
  invariant manifolds {III}: overview and applications. Journal of Differential
  Equations 218(2):444 -- 515, \doi{https://doi.org/10.1016/j.jde.2004.12.003},
  \urlprefix\url{http://www.sciencedirect.com/science/article/pii/S0022039604005170}

\bibitem[{Capi{\'n}ski et~al.(2016)Capi{\'n}ski, Gidea, and de~la
  Llave}]{capinski2016}
Capi{\'n}ski MJ, Gidea M, de~la Llave R (2016) Arnold diffusion in the planar
  elliptic restricted three-body problem: mechanism and numerical verification.
  Nonlinearity 30(1):329

\bibitem[{Celletti(2006)}]{cellettiReg}
Celletti A (2006) Basics of regularization theory. In: Steves BA, Maciejewski
  AJ, Hendry M (eds) Chaotic Worlds: From Order to Disorder in Gravitational
  N-Body Dynamical Systems, Springer Netherlands, Dordrecht, pp 203--230

\bibitem[{Celletti(2010)}]{celletti}
Celletti A (2010) Stability and Chaos in Celestial Mechanics. Astronomy and
  Planetary Sciences, Springer-Verlag Berlin Heidelberg,
  \doi{10.1007/978-3-540-85146-2}

\bibitem[{Chicone(2006)}]{chicone2006}
Chicone C (2006) Ordinary differential equations with applications, vol~34.
  Springer Science \& Business Media

\bibitem[{Farr{\'e}s et~al.(2017)Farr{\'e}s, Jorba, and Mondelo}]{Farres2017}
Farr{\'e}s A, Jorba {\`A}, Mondelo JM (2017) Numerical study of the geometry of
  the phase space of the augmented {Hill} three-body problem. Celestial
  Mechanics and Dynamical Astronomy 129(1):25--55,
  \doi{10.1007/s10569-017-9762-z},
  \urlprefix\url{https://doi.org/10.1007/s10569-017-9762-z}

\bibitem[{Fenichel(1971)}]{fenichel1971persistence}
Fenichel N (1971) Persistence and smoothness of invariant manifolds for flows.
  Indiana University Mathematics Journal 21(3):193--226

\bibitem[{Fontich et~al.(2009)Fontich, {de la Llave}, and
  Sire}]{fontichLlaveSire}
Fontich E, {de la Llave} R, Sire Y (2009) Construction of invariant whiskered
  tori by a parameterization method. part i: Maps and flows in finite
  dimensions. Journal of Differential Equations 246(8):3136--3213,
  \doi{https://doi.org/10.1016/j.jde.2009.01.037},
  \urlprefix\url{https://www.sciencedirect.com/science/article/pii/S0022039609000655}

\bibitem[{Galassi et~al.(2009)Galassi, Davies, Theiler, Gough, Jungman, Alken,
  Booth, and Rossi}]{gslManual}
Galassi M, Davies J, Theiler J, Gough B, Jungman G, Alken P, Booth M, Rossi F
  (2009) GNU Scientific Library Reference Manual - Third Edition, 3rd edn.
  Network Theory Ltd.

\bibitem[{Gol{\'e}(2001)}]{gole}
Gol{\'e} C (2001) Symplectic Twist Maps. World Scientific, \doi{10.1142/1349},
  \urlprefix\url{https://www.worldscientific.com/doi/abs/10.1142/1349},
  \eprint{https://www.worldscientific.com/doi/pdf/10.1142/1349}

\bibitem[{G{\'o}mez et~al.(2001)G{\'o}mez, Llibre, Mart{\'\i}nez, and
  Sim{\'o}}]{simoetal}
G{\'o}mez G, Llibre J, Mart{\'\i}nez R, Sim{\'o} C (2001) Dynamics and Mission
  Design Near Libration Points. World Scientific

\bibitem[{Haro and de~la Llave(2006)}]{haroLlave}
Haro {\`A}, de~la Llave R (2006) A parameterization method for the computation
  of invariant tori and their whiskers in quasi-periodic maps: Numerical
  algorithms. Discrete \& Continuous Dynamical Systems - B 6(6):1261--1300

\bibitem[{Haro and de~la Llave(2007)}]{haroLlave2007}
Haro A, de~la Llave R (2007) A parameterization method for the computation of
  invariant tori and their whiskers in quasi-periodic maps: explorations and
  mechanisms for the breakdown of hyperbolicity. SIAM J Appl Dyn Syst
  6(1):142--207, \doi{10.1137/050637327},
  \urlprefix\url{https://doi.org/10.1137/050637327}

\bibitem[{Haro and Mondelo(2021)}]{haro2021flow}
Haro A, Mondelo J (2021) Flow map parameterization methods for invariant tori
  in hamiltonian systems. Communications in Nonlinear Science and Numerical
  Simulation 101:105859, \doi{https://doi.org/10.1016/j.cnsns.2021.105859},
  \urlprefix\url{https://www.sciencedirect.com/science/article/pii/S1007570421001714}

\bibitem[{Haro et~al.(2016)Haro, Canadell, Figueras, Luque, and
  Mondelo}]{haroetal}
Haro {\`A}, Canadell M, Figueras J, Luque A, Mondelo J (2016) The
  Parameterization Method for Invariant Manifolds: From Rigorous Results to
  Effective Computations, Applied Mathematical Sciences, vol 195. Springer
  International Publishing

\bibitem[{Hiday-Johnston and Howell(1994)}]{hiday1994transfers}
Hiday-Johnston L, Howell K (1994) Transfers between libration-point orbits in
  the elliptic restricted problem. Celestial Mechanics and Dynamical Astronomy
  58(4):317--337

\bibitem[{Hirsch et~al.(1977)Hirsch, Pugh, and Shub}]{hirschPughShub}
Hirsch MW, Pugh CC, Shub M (1977) {Invariant manifolds}. Lecture Notes in
  Mathematics, Springer, Berlin, \doi{10.1007/BFb0092042},
  \urlprefix\url{https://cds.cern.ch/record/1690874}

\bibitem[{Huguet et~al.(2012)Huguet, de~la Llave, and Sire}]{huguet2012}
Huguet G, de~la Llave R, Sire Y (2012) Computation of whiskered invariant tori
  and their associated manifolds: New fast algorithms. Discrete \& Continuous
  Dynamical Systems - A 32(4):1309--1353, \doi{10.3934/dcds.2012.32.1309},
  \urlprefix\url{http://aimsciences.org//article/id/f497a57f-e25e-4dcb-abb5-15b091455833}

\bibitem[{Koon et~al.(2011)Koon, Lo, Marsden, and Ross}]{KoLoMaRo}
Koon WS, Lo MW, Marsden JE, Ross SD (2011) Dynamical systems, the three-body
  problem and space mission design. Marsden Books,
  \urlprefix\url{http://www.dept.aoe.vt.edu/~sdross/books/KoLoMaRo_DMissionBook_2011-04-25.pdf}

\bibitem[{Kumar et~al.(2021{\natexlab{a}})Kumar, Anderson, and {de la
  Llave}}]{kumar2021journal}
Kumar B, Anderson RL, {de la Llave} R (2021{\natexlab{a}}) High-order resonant
  orbit manifold expansions for mission design in the planar circular
  restricted 3-body problem. Communications in Nonlinear Science and Numerical
  Simulation 97:105691, \doi{https://doi.org/10.1016/j.cnsns.2021.105691},
  \urlprefix\url{http://www.sciencedirect.com/science/article/pii/S1007570421000022}

\bibitem[{Kumar et~al.(2021{\natexlab{b}})Kumar, Anderson, and {de la
  Llave}}]{kumar2021feb}
Kumar B, Anderson RL, {de la Llave} R (2021{\natexlab{b}}) Using {GPUs} and the
  parameterization method for rapid search and refinement of connections
  between tori in periodically perturbed planar circular restricted 3-body
  problems. In: AAS/AIAA Space Flight Mechanics Meeting, AAS 21-349

\bibitem[{De~la Llave(2001)}]{kamTutorial}
De~la Llave R (2001) A tutorial on {KAM} theory. In: Smooth Ergodic Theory and
  Its Applications, Seattle, WA, 1999, American Mathematical Society,
  Providence, RI, vol~69, pp 175--292, \doi{10.1090/pspum/069/1858536}

\bibitem[{de~la Llave et~al.(2005)de~la Llave, Gonz{\'{a}}lez, Jorba, and
  Villanueva}]{Llave_2005}
de~la Llave R, Gonz{\'{a}}lez A, Jorba {\`{A}}, Villanueva J (2005) {KAM}
  theory without action-angle variables. Nonlinearity 18(2):855--895,
  \doi{10.1088/0951-7715/18/2/020},
  \urlprefix\url{https://doi.org/10.1088/0951-7715/18/2/020}

\bibitem[{Olikara(2016)}]{olikaraThesis}
Olikara ZP (2016) Computation of quasi-periodic tori and heteroclinic
  connections in astrodynamics using collocation techniques. PhD thesis

\bibitem[{P{\'e}rez-Palau et~al.(2015)P{\'e}rez-Palau, Masdemont, and
  G{\'o}mez}]{perezpalau2015}
P{\'e}rez-Palau D, Masdemont JJ, G{\'o}mez G (2015) Tools to detect structures
  in dynamical systems using jet transport. Celestial Mechanics and Dynamical
  Astronomy 123(3):239--262, \doi{10.1007/s10569-015-9634-3},
  \urlprefix\url{https://doi.org/10.1007/s10569-015-9634-3}

\bibitem[{Rackauckas and Nie(2017)}]{DifferentialEquations.jl-2017}
Rackauckas C, Nie Q (2017) Differentialequations.jl -- a performant and
  feature-rich ecosystem for solving differential equations in julia. The
  Journal of Open Research Software 5(1), \doi{10.5334/jors.151},
  \urlprefix\url{https://app.dimensions.ai/details/publication/pub.1085583166
  and
  http://openresearchsoftware.metajnl.com/articles/10.5334/jors.151/galley/245/download/},
  exported from https://app.dimensions.ai on 2019/05/05

\bibitem[{Rasotto et~al.(2016)Rasotto, Morselli, Wittig, Massari, Lizia,
  Armellin, Valles, and Ortega}]{dast}
Rasotto M, Morselli A, Wittig A, Massari M, Lizia PD, Armellin R, Valles C,
  Ortega G (2016) Differential algebra space toolbox for nonlinear uncertainty
  propagation in space dynamics. In: 6th International Conference on
  Astrodynamics Tools and Techniques (ICATT),
  \urlprefix\url{http://epubs.surrey.ac.uk/813477/}

\bibitem[{R{\"u}ssmann(1975)}]{Russ75}
R{\"u}ssmann H (1975) On optimal estimates for the solutions of linear partial
  differential equations of first order with constant coefficients on the
  torus, Springer Berlin Heidelberg, Berlin, Heidelberg, pp 598--624.
  \doi{10.1007/3-540-07171-7_19},
  \urlprefix\url{https://doi.org/10.1007/3-540-07171-7_19}

\bibitem[{Scheeres(1998)}]{scheeres1998restricted}
Scheeres D (1998) The restricted {Hill} four-body problem with applications to
  the {Earth--Moon--Sun} system. Celestial Mechanics and Dynamical Astronomy
  70(2):75--98

\bibitem[{Sim{\'o} et~al.(1995)Sim{\'o}, G{\'o}mez, Jorba, and
  Masdemont}]{simo1995bicircular}
Sim{\'o} C, G{\'o}mez G, Jorba {\`A}, Masdemont J (1995) The bicircular model
  near the triangular libration points of the {RTBP}. In: From Newton to chaos,
  Springer, pp 343--370

\bibitem[{Szebehely(1967)}]{szebehely1969}
Szebehely V (1967) Theory of Orbits: The Restricted Problem of Three Bodies.
  Academic Press Inc., New York

\bibitem[{Thirring(1992)}]{thirring}
Thirring W (1992) A course in mathematical physics. Walter Thirring ;
  translated by Evans M. Harrell [t.] 1 and 2, Classical dynamical systems and
  classical field theory, 2nd edn. Springer, New York;

\bibitem[{Vaquero et~al.(2014)Vaquero, Hahn, Stumpf, Valerino, Wagner, and
  Wong}]{vaqueroCassini}
Vaquero M, Hahn Y, Stumpf P, Valerino PN, Wagner SV, Wong M (2014) Cassini
  Maneuver Experience for the Fourth Year of the Solstice Mission.
  \doi{10.2514/6.2014-4348},
  \urlprefix\url{https://arc.aiaa.org/doi/abs/10.2514/6.2014-4348},
  \eprint{https://arc.aiaa.org/doi/pdf/10.2514/6.2014-4348}

\bibitem[{{Zhang} and {de la Llave}(2018)}]{zhang}
{Zhang} L, {de la Llave} R (2018) {Transition state theory with quasi-periodic
  forcing}. Communications in Nonlinear Science and Numerical Simulations
  62:229--243, \doi{10.1016/j.cnsns.2018.02.014}

\end{thebibliography}

\end{document}